\documentclass{amsart}
\usepackage{amsmath,amsthm,amssymb,latexsym,epic,bbm,comment}
\usepackage{graphicx,enumerate,stmaryrd, xcolor}
\usepackage[all,2cell]{xy}
\xyoption{2cell}

\usepackage[parfill]{parskip}
\usepackage[
colorlinks=true,
linkcolor=black, 
anchorcolor=black,
citecolor=black, 
urlcolor=black, 
]{hyperref}
\usepackage[e]{esvect}
\usepackage{mathrsfs}

\newtheorem{theorem}{Theorem}
\newtheorem{lemma}[theorem]{Lemma}
\newtheorem{corollary}[theorem]{Corollary}
\newtheorem{proposition}[theorem]{Proposition}

\theoremstyle{definition}

\newtheorem{example}[theorem]{Example}
\newtheorem{remark}[theorem]{Remark}

\newtheorem{question}[theorem]{Question}
\newtheorem{definition}[theorem]{Definition}

\newcommand{\cC}{\mathscr{C}}
\newcommand{\cK}{\mathscr{K}}

\newcommand{\cS}{\mathscr{S}}

\newcommand{\cD}{\mathscr{D}}

\newcommand{\cB}{\mathscr{B}}

\newcommand{\cZ}{\mathscr{Z}}

\newcommand{\bsfF}{\text{\textbf{\textsf{F}}}}
\newcommand{\bsfR}{\text{\textbf{\textsf{R}}}}
\newcommand{\bsfP}{\text{\textbf{\textsf{P}}}}

\newcommand{\csf}{\text{-}\mathrm{csf}}
\newcommand{\comp}{\mathrm{cp}}

\newcommand{\plmod}{\text{-}\mathrm{mod}^{\mathrm{dg}}}
\newcommand{\prmod}[1]{\mathrm{mod}^{\mathrm{dg}}_{#1}\text{-}}

\newcommand{\csfcat}{\mathfrak{M}^{\mathrm{dg}}}
\newcommand{\gcsfcat}{\mathfrak{M}^{\mathrm{dg}}_{\mathrm{g}}}
\newcommand{\triancat}{\mathfrak{T}^{\mathrm{dg}}}
\newcommand{\tworep}{\text{-}\mathfrak{mod}^{\mathrm{pre}}}
\newcommand{\gtworep}{\text{-}\mathfrak{mod}^{\mathrm{pre}}_{\mathrm{g}}}
\newcommand{\ctworep}{\text{-}\mathfrak{mod}^{\mathrm{pre}}_{\mathrm{c}}}
\newcommand{\cgtworep}{\text{-}\mathfrak{mod}^{\mathrm{pre}}_{\mathrm{cg}}}
\newcommand{\intworep}{\text{-}\mathfrak{mod}^{\mathrm{pre}}_{\mathrm{int}}}
\newcommand{\dgmod}{\text{-}\mathrm{mod}^{\mathrm{dg}}}
\newcommand{\cdgmod}{\text{-}\mathrm{mod}^{\mathrm{dg,cp}}}

\newcommand*\cocolon{%
        \nobreak
        \mskip6mu plus1mu
        \mathpunct{}%
        \nonscript
        \mkern-\thinmuskip
        {:}%
        \mskip2mu
        \relax
}

\newcommand{\adj}[4]{#1\colon #2\rightleftarrows #3\cocolon #4}
\newcommand{\shift}[1]{\langle #1\rangle}

\newcommand{\coker}{\operatorname{coker}}
\newcommand{\diag}{\operatorname{diag}}
\newcommand{\Hom}{\mathrm{Hom}}
\newcommand{\End}{\mathrm{End}}

\newcommand{\op}{\mathrm{op}} 
\newcommand{\im}{\operatorname{Im}}

\newcommand{\rA}{\mathrm{A}}
\newcommand{\rB}{\mathrm{B}}
\newcommand{\rC}{\mathrm{C}}

\newcommand{\rF}{\mathrm{F}}
\newcommand{\rG}{\mathrm{G}}
\newcommand{\rH}{\mathrm{H}}
\newcommand{\rI}{\mathrm{I}}
\newcommand{\rK}{\mathrm{K}}

\newcommand{\rM}{\mathrm{M}}
\newcommand{\rN}{\mathrm{N}}

\newcommand{\rX}{\mathrm{X}}
\newcommand{\rY}{\mathrm{Y}}

\newcommand{\A}{\mathcal{A}}

\newcommand{\C}{\mathcal{C}}
\newcommand{\DD}{\mathcal{D}}
\newcommand{\E}{\mathcal{E}}

\newcommand{\I}{\mathcal{I}}
\newcommand{\J}{\mathcal{J}}
\newcommand{\K}{\mathcal{K}}

\newcommand{\T}{\mathcal{T}}

\newcommand{\Z}{\mathcal{Z}}

\newcommand{\del}{\partial}
\newcommand{\id}{\mathrm{id}}
\newcommand{\Id}{\mathrm{Id}}

\newcommand{\one}{\mathbbm{1}}
\newcommand{\eval}{\mathsf{ev}}
\newcommand{\Eval}{\mathsf{Ev}}

\newcommand{\bfM}{\mathbf{M}}
\newcommand{\bfK}{\mathbf{K}}

\newcommand{\bfG}{\mathbf{G}}
\newcommand{\bfP}{\mathbf{P}}
\newcommand{\bfN}{\mathbf{N}}

\newcommand{\bfI}{\mathbf{I}}

\newcommand{\tth}{\mathtt{h}}
\newcommand{\tone}{\mathtt{1}}

\newcommand{\ti}{\mathtt{i}}
\newcommand{\tj}{\mathtt{j}}
\newcommand{\tk}{\mathtt{k}}
\newcommand{\tl}{\mathtt{l}}

\newcommand{\tn}{\mathtt{n}}

\newcommand{\tX}{\mathtt{X}}

\newcommand{\ov}[1]{\overline{#1}}
\newcommand{\mat}[1]{\left(\begin{smallmatrix}#1\end{smallmatrix}\right)}

\makeatletter
\newcommand{\colim@}[2]{%
  \vtop{\m@th\ialign{##\cr
    \hfil$#1\operator@font colim$\hfil\cr
    \noalign{\nointerlineskip\kern1.5\ex@}#2\cr
    \noalign{\nointerlineskip\kern-\ex@}\cr}}%
}
\newcommand{\colim}{%
  \mathop{\mathpalette\colim@{\rightarrowfill@\scriptscriptstyle}}\nmlimits@
}
\renewcommand{\varprojlim}{%
  \mathop{\mathpalette\varlim@{\leftarrowfill@\scriptscriptstyle}}\nmlimits@
}
\renewcommand{\varinjlim}{%
  \mathop{\mathpalette\varlim@{\rightarrowfill@\scriptscriptstyle}}\nmlimits@
}
\makeatother

\numberwithin{equation}{section}
\numberwithin{theorem}{section}

\begin{document}

\title[pretriangulated 2-representations]{pretriangulated 2-representations via dg algebra 1-morphisms}
\author{Robert Laugwitz}
\address{School of Mathematical Sciences,
University of Nottingham, University Park, Nottingham, NG7 2RD, UK}
\email{robert.laugwitz@nottingham.ac.uk}

\author{Vanessa Miemietz}
\address{
School of Mathematics, University of East Anglia, Norwich NR4 7TJ, UK}
\email{v.miemietz@uea.ac.uk}
\urladdr{https://www.uea.ac.uk/~byr09xgu/}

\date{April 22, 2025}

\begin{abstract}
This paper develops a theory of pretriangulated $2$-representations of dg $2$-categories. We characterize cyclic pretriangulated $2$-representations, under certain compactness assumptions, in terms of dg modules over dg algebra $1$-morphisms internal to associated dg $2$-categories of compact objects. Further, we investigate the Morita theory and quasi-equivalences for such dg $2$-representations. We relate this theory to various classes of examples of dg categorifications from the literature. 
\end{abstract}


\maketitle

\tableofcontents

\section{Introduction}\label{s0}

Categorification aims to lift important algebraic structures to a higher categorical level, where elements are replaced by functors (or $1$-morphisms) and equations are upgraded to natural isomorphisms (or invertible $2$-morphisms). The underlying classical structure can then be recovered through the process of passing to Grothendieck groups. Categorifications of Hecke algebras, through Soergel bimodules, or quantum groups have been among the most celebrated achievements of the field and featured in solutions to long-standing conjectures in representation theory, see e.g. \cite{CR,EW,W}. A natural framework for categorification is that of $2$-categories, which contain monoidal or tensor categories as the one-object case. 

Studying categorical representations was proposed in \cite{Ro}. A systematic theory for categorical representations of tensor categories (such as fusion categories) was developed by Etingof--Ostrik and others \cite{EGNO}. These categories are usually abelian while several categorifications are based on non-abelian additive categories. 
A systematic theory of $2$-representations tailored to the latter, categorifying in some sense the theory of representations of finite-dimensional algebras, has been developed in a programme starting with \cite{MM1}, see also \cite{MM2,MM3,MM5,MMMT,MMMZ} among others.

Many of the important categorifications, such as categorified braid groups or categorical braid group actions \cite{ST,KS,Ro0,Ro1}, only emerge when passing from additive categories to triangulated categories. This is manifest, for example, in the fact that braid group relations only become isomorphisms on the level of homotopy or derived categories. This observation motivates the need for a theory of $2$-representations that allows working with homotopy categories. In fact, the literature contains several constructions of categorifications of $2$-representations of specific categorified algebras on the level of triangulated categories, including, e.g., \cite{BFK,RZ,Ca,CL, Sa,JY} to list a few examples.

Working directly with triangulated categories involves technical obstacles to a formal treatment of $2$-representations. For a formal theory of $2$-representations, it is important to account for coherence conditions of the structural isomorphisms involved. The definition of triangulated categories, in addition to being very involved, lacks fundamental properties such as functoriality of cones. A solution to these technical issues was provided by Bondal--Kapranov \cite{BK} by working with \emph{pretriangulated} categories. Pretriangulated categories are \emph{differentially graded (dg)} categories closed under taking cones and shifts, thus ensuring that the associated homotopy or derived categories are triangulated. In this case, the dg category is an enhancement of the triangulated category \cite{LO}. This approach solves key technical issues associated with triangulated categories such as functoriality of cones, taking duals and tensor products \cite{LO,BLL}. Dg enhancements have been used in numerous constructions in categorification, for example, in the theory of spherical twists underlying categorical braid group actions~\cite{AL}. 

The present paper proposes a theory of $2$-representations suitable for working with pretriangulated categories. We extend the setup used in \cite{MM1,MM2,MMMT} to pretriangulated categories. In particular, this setup allows to consider pretrianguated hulls of finitary $2$-categories, $\mathbb{Z}$-gradings, and relaxes the strict finiteness requirements on spaces of $2$-morphisms for some of the constructions of finitary $2$-representation theory. This first paper in this direction focuses on representing cyclic $2$-representations through modules over dg algebra $1$-morphisms generalizing results of \cite{EGNO} and \cite{MMMT}.

The theory of pretriangulated $2$-representations proposed in this paper applies to several classes of $2$-categories, some of which we start to explore here.
\begin{itemize}
\item Pretriangulated hulls of finitary $2$-categories. This allows us to consider $2$-representa\-tions acting on (bounded) complexes of (projective) modules over finite-dimensional $\Bbbk$-al\-ge\-bras. These $2$-representations descend to triangulated categories by acting on the associated homotopy categories.
\item Categorifications involving dg $2$-categories such as \cite{Ti,Ti2,Ti3,KT}. As a first example, we explore the categorification of $\mathbb{Z}[\sqrt{-1}]$ and its natural action on $\mathbb{Z}^{\oplus 2}$ of \cite{Ti} in detail in Section~\ref{sec:Zi}.
\item Dg $2$-categories $\cC_A$ built from projective bimodules over finite-dimensional dg $\Bbbk$-al\-gebras. In finitary $2$-representation theory, these $2$-categories can be used to classify simple transitive $2$-representation of categorified finite-dimensional Lie algebras, see \cite[Section~7.2]{MM5}. The first example of such dg $2$-categories is given by $\cC_\Bbbk$, the $2$-category of bounded complexes of finite-dimensional $\Bbbk$-vector spaces. Using results of Orlov \cite{Or2}, we prove in Proposition~\ref{lem:Ck-simpletrans} that $\cC_\Bbbk$ has a unique non-acyclic quotient-simple pretriangulated $2$-representation.
\item Finally, we explain how the theory developed in this paper can be applied to categorical braid group actions of \cite{KS,Ro0,Ro1} in Section \ref{sec:braid}.  We relate $2$-representations of categorified braid groups in the approach of \cite{KS} to dg $2$-representations of categories $\cC_A$, where $A$ is a zigzag algebra with trivial differential. More generally, categorified braid group representations of \cite{Ro0} are related to pretriangulated $2$-representations of pretriangulated hulls of $2$-categories of Soergel bimodules.
\end{itemize}

Given a dg $2$-category $\cC$, we define the dg $2$-category of \emph{pretriangulated $2$-rep\-re\-sentations} $\cC\tworep$ in Section \ref{dg2reps}. Given a pretriangulated $2$-representation  $\bfM$, one ob\-tains the \emph{homotopy $2$-representation} $\bfK\bfM$ by passing to the associated homotopy categories $\bfK\bfM(\ti)$ at each object $\ti$ of $\cC$, see Section~\ref{triang2rep-sec}. 

The main results of this paper involve an assignment of internal dg algebra $1$-morphisms to pretriangulated $2$-representations which admit a generator under the action of the dg $2$-category $\cC$ and taking thick closures. We obtain the following results:
\begin{itemize}
\item We define the \emph{internal} dg $2$-representation $\bfM_\rA$ consisting of certain modules over an internal dg algebra $1$-morphism $\rA$ which lives in the dg $2$-category $\vv{\cC}$ introduced in this paper. The dg $2$-category $\vv{\cC}$ provides a generalization and dg enhancement of the completion of a finitary additive $2$-category under cokernels.
\item In Corollary~\ref{cor:stronggenMAequiv} we prove that if $X$ is a $\cC$-generator of a \emph{compact} pretriangulated $2$-representation $\bfM$, then the dg idempotent completion $\bfM^\circ$ is dg equivalent to the internal dg $2$-representation $\bfM_{\rA_X}$, where $\rA_X=[X,X]$ is the internal endomorphism dg algebra $1$-morphism of $X$.

A version of this result, where $X$ is only a $\cC$-\emph{quasi}-generator, i.e., $X$ generates the associated homotopy $2$-representation $\bfK\bfM$, is given in  Corollary \ref{cor:quasi-gen}.
\item We define a possible notion of ``simple'' dg $2$-representations of dg $2$-categories called \emph{quotient-simple pretriangulated $2$-representations}. We show in Corollary \ref{cor4.10} that a cyclic pretriangulated $2$-representation $\bfM$ is quotient-simple if and only if $\rA_X$ is simple 
as a dg algebra  $1$-morphism.
\item Using results by Keller \cite{Ke1}, we characterize equivalences (and quasi-equivalences) between internal dg $2$-representations $\bfM_\rA$ and $\bfM_\rB$ over $\cC$ in terms of Morita equivalence, see Sections \ref{sec:Morita} and \ref{sec:quasi-Morita}.
\item In Section \ref{sec:local-quasi} we explain how a dg $2$-functor $\bsfF\colon \cC\to \cD$ which is a local quasi-equivalence and essentially surjective on objects induces a correspondence between internal dg $2$-representations of $\cC$ and $\cD$ up to quasi-equivalence, see Proposition \ref{prop-unit-quasi}, \ref{prop-counit-quasi}, and \ref{prop:preserve-quasi-eq} as well as Corollary \ref{cor:quasi-2-equiv} for precise statements.
\end{itemize}
The definition of quotient-simple pretriangulated $2$-representations was inspired by an extension of the theory of cell $2$-representations of \cite{MM1} to the pretriangulated setup that will appear in a forthcoming paper in preparation. 

The results of this paper refine and extend some of the constructions given in the setup of $p$-dg $2$-representations from \cite{LM}. It would be possible to adapt the constructions of algebra $1$-morphisms of the present paper to the $p$-dg setup (for $\operatorname{char} \Bbbk=p$) in order to be applied to the categorifications of quantum groups at $p$-th roots of unity of \cite{EQ,EQ2}. Further, the setup presented here could be adapted to incorporate $\mathbb{Z}/2\mathbb{Z}$-gradings as used, e.g., in \cite{ElQ,EL}.

The paper is organized as follows. Section~\ref{sec:dggen} introduces the technical constructions for dg categories required in this paper, including a description of the dg category of compact objects as a dg enriched generalization of the projective abelianization of a finitary category of \cite{Fr} by adding cokernels. Next, Section~\ref{dg2} introduces dg $2$-categories and pretriangulated $2$-representations, including the concepts of \emph{cyclic} and \emph{quotient-simple} pretriangulated $2$-representations and explains how pretriangulated $2$-representations induce triangulated homotopy $2$-representations. The main results of the paper are found in Section~\ref{alg1morsec} revolving around associating dg algebra $1$-morphisms to pretriangulated $2$-representations. Finally, Section~\ref{sec:examples} discusses various classes of examples.

\subsection*{Acknowledgements}

The authors would like to thank Ben Elias, Gustavo Jasso, and Marco Mackaay for helpful conversations related to this work. We further thank the anonymous referees for their helpful comments.
V.~M. is supported by EPSRC grant EP/S017216/1 and R.~L. is supported by a Nottingham Research Fellowship.

\section{Generalities}
\label{sec:dggen}

\subsection{Dg categories}\label{dgsect}

A \textbf{dg category} $\C$ is a category enriched over the symmetric monoidal category of cochain complexes of $\Bbbk$-modules. If $C$ is a cochain complex of $\Bbbk$-modules, we denote by $Z(C)$ the subspace of cocycles in $C$ of degree zero.

We denote by $\Bbbk\dgmod$ the category of \textbf{dg $\Bbbk$-modules}. Its objects are $\mathbb{Z}$-graded $\Bbbk$-vector spaces equipped with a differential $\del$ of degree $+1$ such that $\del^2=0$. A morphism $f\colon V\to W$ in $\Bbbk\dgmod$ of degree $n$ is a $\Bbbk$-linear map satisfying $f(V^k)\subset W^{k+n}$ on the $k$-th graded piece. The differential of such a morphism is given by 
$\del(f)=\del_W\circ f-(-1)^n f\circ \del_V.$ Hence, the category $\Bbbk\dgmod$ is enriched over  the symmetric monoidal category of cochain complexes of $\Bbbk$-modules (see e.g. \cite[Section~2.1]{Ke}). It is equipped with a \textbf{shift functor} defined by $(V\shift{1})^k=V^{k+1}$ on the $k$-th graded piece with differential given by $-\del_V$. For a homogeneous morphism $f\colon V\to W$ in $\Bbbk\dgmod$ of degree $n$, the morphism $f\shift{1}\colon V\shift{1}\to W\shift{1}$ is given by $(-1)^n f$. By abuse of notation, for morphisms of degree zero, we often omit the shifts in the notation to simplify the exposition.

Given a dg category, we denote by $\Z(\C)$ the $\Bbbk$-linear category given by the same objects as $\C$ with morphisms  given by the $\Bbbk$-vector spaces
\begin{align*}\Hom_{\Z(\C)}(X,Y)=Z(\Hom_{\C}(X,Y))=\left\{f\in \Hom_{\C}(X,Y)\vert \deg f=0\text{ and }\del f=0 \right\}.\end{align*}
We call the morphisms in $\Z(\C)$ \textbf{dg morphisms}. We use the terminology of dg isomorphism, dg idempotent, dg direct summand, dg indecomposable, dg equivalent etc. for properties referring to or defined by morphisms in $\Z(\C)$.

Let $\C$ be a small dg category. We define the dg category $\C\dgmod$ of \textbf{dg modules over $\C$} to be the dg category of dg functors $\C^\op\to \Bbbk\dgmod$ (cf. \cite[Section 1.2]{Ke}, \cite[Section 2.2]{Or1}). 
Throughout this paper, we will only consider dg modules over small dg categories without further mention. 
Given a dg category $\C$, the Yoneda lemma gives a fully faithful dg functor 
\begin{align*}\C\to \C\dgmod, \qquad X\mapsto X^\vee:=\Hom_\C(-,X).\end{align*}
We refer to $X^\vee$ as a \textbf{free} dg $\C$-module.
It has the property that given any dg module $M$ over $\C$, there is a dg isomorphism 
\begin{align*}\Hom_{\C\dgmod}(X^\vee,M)\cong M(X).\end{align*}

We require the following notion of idempotent completion to obtain dg categories closed under taking dg direct summands. Given a dg category $\C$, the \textbf{dg idempotent completion} $\C^\circ$ is defined as having objects $X_e$ for any dg idempotent $e=e^2 \in \End_\C(X)$, i.e., $\deg(e)=0$ and $\del(e)=0$, and morphism spaces
\begin{align*}\Hom_{\C^\circ}(X_e,Y_f)=f\Hom_{\C}(X,Y)e.\end{align*}
Then $\C^\circ$ is a dg category and the embedding $\C\hookrightarrow \C^\circ$ is a dg functor. Furthermore, $(\C^\circ)^\circ$ is dg equivalent to $\C^\circ$.

Given a dg functor $F\colon \C\to \DD$, we obtain an induced dg functor 
\begin{align}
F^\circ \colon \C^\circ \to \DD^\circ, \qquad F(X_e)=F(X)_{F(e)}.\label{functorcirc}
\end{align}
Further, given a natural transformation of dg functors $\phi\colon F\to G$, 
\begin{align}
\phi^\circ_{X_e}:=G(e)\circ\phi_X\circ F(e)\label{natcirc}
\end{align}
defines a natural transformation $\phi^\circ\colon F^\circ\to G^\circ$.

\subsection{Pretriangulated categories}\label{section:pretriang}

A dg category $\C$ is \textbf{pretriangulated} if its Yoneda embedding into $\C\dgmod$ via $X\mapsto X^\vee=\Hom_{\C}(-,X)$ is closed under shifts and cones, cf. \cite[Section 4.5]{Ke}.
Note that we do not require $\C$ to be closed under dg direct summands. To alleviate notation, we refer to a pretriangulated dg category simply as a \textbf{pretriangulated category.} We say that a full dg subcategory $\DD\subseteq \C$ is a \textbf{pretriangulated subcategory} if it is closed under shifts, cones and dg isomorphisms, and say that $\DD$ is \textbf{thick} if it is, in addition, closed under all direct summands which exist in $\C$.

Let $\C$ be a pretriangulated category. 
Given a pretriangulated category $\C$, and a set $\tX$ of objects in $\C$, we define the \textbf{thick closure} of $\tX$ in $\C$ to be the full subcategory $\widehat{\tX}$ of $\C$ generated by all objects in $\tX$ under shifts, cones, direct sums, dg direct summands, 
and closed under dg isomorphisms. That is, $\widehat{\tX}$ is the smallest thick subcategory of $\C$ containing $\tX$. An object $X$ in $\C$ is a \textbf{(classical) generator} for $\C$ if the thick closure $\widehat{\lbrace X \rbrace}$ of $X$ is all of $\C$. Note that these constructions are taken inside the pretriangulated category, rather than at the homotopy level and do thus not coincide with the notions of thick closure and classical generator for the associated triangulated category. This is desirable since most constructions in this paper will take place on the pretriangulated level. For constructions involving passing to homotopy, we will usually use the prefix quasi, see e.g. Section \ref{triang2rep-sec}.

Given a dg category $\C$, we denote by $\ov{\C}$ the dg category of (one-sided) twisted complexes in $\C$. Explicitly, we define $\overline{\C}$ as the dg category whose
\begin{itemize}
\item objects are given by pairs $(X=\bigoplus_{m=1}^s X_m, \alpha=(\alpha_{k,l})_{k,l})$ where the $F_m$ are shifts of objects in $\C$ and $\alpha_{k,l}\in \Hom_\C(X_l,X_k)$, $\alpha_{k,l} = 0$ for all $k\geq l$ such that the matrix $\del \cdot \mathtt{I}_s +\left((\alpha_{kl})_*\right)_{kl}$ acts as a differential on $\bigoplus_{m=1}^s X_m^\vee$ in  $\C\dgmod$ (here $\mathtt{I}_s$ is the identity matrix), or, equivalently, $\del(\alpha)+\alpha^2=0$;
\item morphisms are matrices of morphisms between the corresponding objects, with the differential of a homogeneous morphism
\begin{align*}\gamma=(\gamma_{n,m})_{n,m} \colon \bigg(\bigoplus_{m=1}^s X_m, \alpha=(\alpha_{k,l})_{k,l}\bigg)\longrightarrow \bigg(\bigoplus_{n=1}^t Y_n, \beta=(\beta_{k,l})_{k,l}\bigg)\end{align*}
defined as 
\begin{align*} \partial\left((\gamma_{n,m})_{n,m} \right):= (\partial \gamma_{n,m}+(\beta\gamma)_{n,m}-(-1)^{\deg \gamma}(\gamma\alpha)_{n,m})_{n,m}.\end{align*}
\end{itemize}

The notation $\bigoplus_{i=m}^s X_m$ denotes an ordered list of objects rather than a direct sum internal to $\C$. We give $\ov{\C}$ the explicit additive structure
\begin{align*}(X,\alpha)\oplus (Y,\beta):=\left(X\oplus Y,\begin{pmatrix}
\alpha&0\\0&\beta
\end{pmatrix}\right),\end{align*}
where $X\oplus Y$ is the concatenation of ordered lists of objects. This additive structure on $\ov{\C}$ is strict, in the sense that $(X\oplus Y)\oplus Z=X\oplus (Y\oplus Z)$.

Note that $\ov{\C}$ is a pretriangulated category, and the smallest pretriangulated category containing $\C$, cf. \cite[Section~1]{LO}, \cite[Section 3.2]{AL}. Note that, in particular, $\ov{\C}\dgmod$ is dg equivalent to $\C\dgmod$, see e.g.\ \cite[Section 4.5]{Ke}.

Let $f\colon X\to Y$  be a dg morphism in $\ov{\C}$, where $X=(\bigoplus_{i=1}^tX_i,\alpha)$, $Y=(\bigoplus_{i=1}^tY_i,\beta)$. The \textbf{cone} $C_f$ of $f$ is the object 
\begin{align*}C_f=\Big(Y\oplus X\shift{1}, \mat{\beta& -f\\ 0&\alpha}\Big).\end{align*}
The cone $C_f$ comes equipped with the dg morphisms
\begin{align*}C_f\shift{-1}\to X, \qquad Y\to C_f,\end{align*}
such that pre- (respectively, post-) composition with $f$ yields a null-homotopic morphism.

Given a dg functor $F\colon \C\to \DD$, we obtain an induced dg functor 
\begin{align}\label{functorov}
\ov{F}\colon \ov{\C}\to \ov{D}, \quad \ov{F}\bigg(\bigoplus_{i=1}^t X_i,\alpha\bigg)=\bigg(\bigoplus_{i=1}^t F(X_i),F(\alpha)\bigg),
\end{align}
by applying $F$ component-wise to $\alpha$ and to morphisms in $\ov{\C}$. If follows directly that $\ov{G\circ F}=\ov{G}\circ\ov{F}$ for compatible dg functors $F\colon \C\to \DD$, $G\colon \DD\to \E$. Moreover, a natural transformation $\tau\colon F\to G$ induces a natural transformation $\ov{\tau}\colon \ov{F}\to \ov{G}$ using the diagonal matrix
\begin{align}
\tau_{(\oplus_{i=1}^t X_i,\alpha)}=\diag(\tau_{X_1},\ldots, \tau_{X_t}).\end{align}
This way, we obtain a dg $2$-functor $\ov{(-)}$ that sends dg categories $\C$ to pretriangulated categories $\ov{\C}$.

Recall that, for a dg category $\C$, a \textbf{finitely-generated semifree dg $\C$-module} is a dg functor from $\C^\op$ to $\Bbbk\plmod$ which has a finite filtration by shifts of free dg $\C$-modules, cf. \cite[Section~2.2]{Or1}.

\begin{lemma}[{\cite[Section 2.3]{Or2}}] \label{csf-lemma}
The dg category $\ov{\C}$ is dg equivalent to the dg category of finitely-generated semi-free dg $\C$-modules.
\end{lemma}

We say that an ideal $\I$ in a dg category $\C$ is a \textbf{dg ideal} provided that, for any morphism $f$ in $\I$, we also have $\del(f)$ in $\I$. For future reference, we record the following lemmas.

\begin{lemma}\label{idealtrivial}
Let $\C$ be a full subcategory of $\C'$, and $\I$ a dg ideal in $\C$. Then the restriction to $\C$ of the dg ideal generated by $\I$ in $\C'$ is equal to $\I$.
\end{lemma}

\begin{lemma}\label{idealonoverline}
Suppose $\C$ is a dg category equivalent to $\widehat{\{X\}}$ for some $X\in \C$ and let $\I$ be a dg ideal in $\C$. Then $\I$ is completely determined on $\overline{\{X\}}$.  Thus, the subset $\I\cap\End_{\C}(X)$ generates $\I$.
\end{lemma}

\begin{proof}
Let $Y, Z\in \widehat{\{X\}}$. Then there exist objects $Y', Z'\in \overline{\{X\}}$ and dg idempotents $e_Y\in \End_{\overline{\{X\}}}(Y')$ and $e_Z\in \End_{\overline{\{X\}}}(Z')$ such that $ \Hom_{\widehat{\{X\}}}(Y,Z)=  e_Z\Hom_{\widehat{\{X\}}}(Y',Z')e_{Y}$.
We claim that 
$\Hom_{\I}(Y, Z) \cong e_Z\Hom_{\I}(Y', Z')e_Y$.

Indeed, the inclusion $\End_{\I}(Y, Z) \supseteq e_Z\End_{\I}(Y', Z')e_Y$ is trivial from $\I$ being an ideal. An element of $\Hom_{\I}(Y, Z)$ is of the form $e_Z\circ f\circ e_Y = e_Z\circ (e_Z\circ f\circ e_Y) \circ e_Y$ and  $f=e_Z\circ f\circ e_Y \in \Hom_{\I}(Y',Z')$, so the inclusion $\End_{\I}(Y, Z) \subseteq e_Z\End_{\I}(Y', Z')e_Y$ also holds. This proves the lemma.
\end{proof}

Given a dg ideal $\I$ in $\C$ we denote by $\ov{\I}$ the dg ideal of morphisms in $\ov{\C}$ with components contained in $\I$.

\begin{lemma}\label{quotientpretri}
Let $\C$ be a dg category and $\I$ a dg ideal in $\C$. 
\begin{enumerate}[(a)]
\item \label{quotientpretri1}
If $\J$ is a dg ideal in $\ov{\C}$, then $\J=\ov{\left.\J\right|_{\C}}$, where $\left.\J\right|_{\C}$ is the restriction of $\J$ to $\C$.
\item \label{quotientpretri2}
There is a fully faithful dg functor $\ov{\C}/\ov{\I}\hookrightarrow \ov{\big(\C/\I\big)}$.
\item \label{quotientpretri3}
Assume that $\I$ has the property that if $\del(f)$ is in $\I$, then $f$ itself is in $\I$. Then the dg functor from \eqref{quotientpretri2} becomes a dg equivalence.  
\end{enumerate}
\end{lemma}
\begin{proof}
\eqref{quotientpretri1} Denote $\I=\ov{\left.\J\right|_{\C}}$. Let $f=(f_{ij})\in \J$, then using the injection $\iota_i$ and projection $\pi_j$, $f_{ij}=\pi_jf\iota_i\in \J$. Thus, $f\in \ov{\I}$. Conversely, $\I\subseteq \J$ implies $\ov{\I}\subseteq \J.$

\eqref{quotientpretri2} Note that the quotient dg functor $P\colon \C\to \C/\I$ induces a dg functor $\ov{P}\colon \ov{\C}\to \ov{\big(\C/\I\big)}$. On morphisms in $\ov{\C}$, this functor is given by component-wise application of $\pi$. Thus, the kernel of $\ov{P}$ is precisely the ideal $\ov{\I}$.

\eqref{quotientpretri3} It remains to show that the quotient $\ov{\C}/\ov{\I}$ is closed under cones. Given a dg morphism $f$ in $\ov{\C}$ maps to the cone of the image of this morphism in $\ov{\C}/\ov{\I}$. Under the assumptions, any dg morphism in $\ov{\C}/\ov{\I}$ is the image of a dg morphism in $\ov{\C}$, proving the claim. 
\end{proof}

\begin{remark}
Under additional conditions, see, e.g., \cite[Lemma~3.4]{CW}, the dg category $\ov{\C}/\ov{\I}$ is itself pretriangulated and hence dg equivalent to  $\ov{\C/\I}$.
\end{remark}

It follows that if $\C$ has a generator $X$, then $\ov{\C/\I}$ has a generator, which is the image of $X$ under the projection functor.

\begin{lemma}\label{Ccirctri}
If $\C$ is a pretriangulated category, then the dg idempotent completion $\C^\circ$ is also pretriangulated and $\iota\colon \C\to\C^\circ$ displays $\C$ as a full pretriangulated subcategory. In particular, this implies $\ov{\C}^\circ\simeq\ov{\ov{\C}^\circ}$ for any dg category $\C$.
\end{lemma}
\begin{proof}
It is clear that $X_e\shift{1}=(X\shift{1})_e$. To show that $\C^\circ$ is closed under cones, we observe that for a dg morphism $g=fge\colon X_e\to Y_f$ in $\C^\circ$,
the object $Z_{\mat{f&0\\0&e}}$, for
\begin{align*}Z=\bigg(Y\oplus X\shift{1}, \mat{0&-g\\0&0}\bigg)\end{align*}
is the cone of $g$ in $\C^\circ$. This shows $\C^\circ$ is pretriangulated. 
\end{proof}

\subsection{Compact objects}\label{compactsect}

In this section, we give a diagrammatic description of the dg category of compact objects in $\C\dgmod$, for a dg category $\C$.

 Let $\vv{\C}$ denote the category
\begin{itemize}
\item  whose objects are diagrams of the form $X_1\stackrel{x}{\to} X_0$ for $X_0,X_1\in \overline{\C}$ and $x$ a dg morphism in $\ov{\C}$;
\item  whose morphism are pairs $(\phi_0,\phi_1)$ of morphisms in $\overline{\C}$ producing solid commutative diagrams of the form
\begin{equation}\label{vvmorph}
\vcenter{\hbox{
\xymatrix{ X_1 \ar[rr]^{x} \ar[d]^{\phi_1}&&\ar@{-->}[dll]_{\eta} X_0 \ar[d]^{\phi_0}  \\ 
Y_1\ar[rr]^{y} &&Y_0,
}}}
\end{equation}
modulo the dg submodule generated by diagrams where there exists a morphism $\eta$, as indicated by the dashed arrow, such that $\phi_0=y\eta$;
\item the differential of such a pair $(\phi_0,\phi_1)$ is simply the component-wise differential in $\ov{\C}$.
\end{itemize}

Alternatively, one could define $\vv{\C}$ as the quotient of the category of dg functors from the path category of the quiver $1\to 0$ to $\Bbbk\dgmod$ by the ideal defined in the second bullet point.

We note that $\ov{\C}$ is a full dg subcategory of $\vv{\C}$ by mapping $X$ to $0\to X$. We denote this inclusion functor by $\Theta$, but will usually just write $X$ instead of $\Theta(X)$ or $0\to X$. We denote by $\Theta^*\colon \vv{\C}\dgmod \to \ov{\C}\dgmod\simeq \C\dgmod$ the restriction functor sending $M$ to $M\circ \Theta$.

Observe that, for $X_0,X_1, Z\in \ov{\C}$, 
\begin{equation}\label{homfromovC}
\Hom_{\vv{\C}}(Z, X_1 \xrightarrow{x} X_0)  \cong \coker(\Hom_{\ov{\C}}(Z, X_1) \xrightarrow{x\circ (-)}\Hom_{\ov{\C}}(Z, X_0) )
\end{equation}
and for $X_0,X_1, Z_0, Z_1\in \ov{\C}$, 
\begin{equation}\begin{split}\label{hominvvC}
\Hom_{\vv{\C}}&(Z_1\xrightarrow{z} Z_0, X_1 \xrightarrow{x} X_0) \\
&\cong \ker( \Hom_{\vv{\C}}(Z_0, X_1 \xrightarrow{x} X_0) \xrightarrow{ (-)\circ z} \Hom_{\vv{\C}}(Z_1, X_1 \xrightarrow{x} X_0))
.\end{split}\end{equation}
Note that the components $\phi_0$ and $\phi_1$ of a morphism as in \eqref{vvmorph} are sums of homogeneous morphisms of necessarily the same degrees, and $\eta$ is necessarily also a sum of homogeneous morphisms of those degrees.

Any dg functor $F\colon \overline{\C}\to\overline{\DD}$ extends to a dg functor 
$\vv{F}\colon\vv{\C}\to \vv{\DD}$ component-wise, that is,
\begin{align*}\vv{F}(X_1\xrightarrow{x} X_0)=F(X_1)\xrightarrow{F(x)} F(X_0).\end{align*}

Consider the two Yoneda embeddings
\begin{align*}\Upsilon_{{\C}}\colon \ov{\C}\to \C\dgmod, \qquad X\mapsto X^\vee=\Hom_{\ov{\C}}(-,X)\end{align*}
and 
\begin{align*}\Upsilon_{\vv{\C}}\colon\vv{\C}\to \vv{\C}\dgmod \qquad (X_1\stackrel{x}{\to} X_0)\mapsto \Hom_{\vv{\C}}(-, (X_1\stackrel{x}{\to} X_0)).\end{align*}
Let $\Gamma \colon {\C}\dgmod \to \vv{\C}\dgmod $ be defined by 
\begin{align*}\Gamma(M) (Z_1\xrightarrow{z} Z_0) = \ker (M(Z_0)\xrightarrow{M(z)} M(Z_1)).\end{align*} This is a dg functor and $\Theta^*\Gamma \cong \Id_{\ov{\C}\dgmod }$, which implies that $\Gamma$ is faithful.

\begin{lemma}\label{Yonedafactor}
There is a fully faithful dg functor \begin{align*}\Upsilon \colon \vv{\C}\to \C\dgmod, \;\; (X_1\stackrel{x}{\to} X_0)\mapsto \coker(\Upsilon_{{\C}} (x))\end{align*}
inducing a commutative diagram (up to isomorphism)
\begin{align*}\xymatrix{
\vv{\C}\ar^{\Upsilon_{\vv{\C}}}[rr]\ar^{\Upsilon}[dr]&& \vv{\C}\dgmod\\
&\C\dgmod\ar^{\Gamma}[ur]&.
}\end{align*}
\end{lemma}

\begin{proof}
It is clear that $\Upsilon$ defines a dg functor. Assume that the diagram commutes. Then, since $\Upsilon_{\vv{\C}}$ is full and faithful, we see that $\Upsilon$ is faithful and $\Gamma$ is full. Since $\Gamma$ is also faithful, this then implies that $\Upsilon$ is also full. So it remains to show that the diagram commutes.

First, let $X\in \ov{\C}$ and identify $X$ with its image in $\vv{\C}$. Then by construction $\Upsilon(X) = X^\vee$, so 
\begin{equation*}\begin{split}
\Gamma \Upsilon (X^\vee) (Z_1\xrightarrow{z}Z_0) &= \ker (X^\vee(Z_0) \xrightarrow{X^\vee(z)} X^\vee(Z_1)) \\
&= \ker (\Hom_{\ov{\C}} (Z_0, X) \xrightarrow {-\circ z}\Hom_{\ov{\C}} (Z_1, X) ) \\
&\cong \Hom_{\vv{\C}}(Z_1\xrightarrow{z}Z_0, X) \\
&= \Upsilon_{\vv{\C}}(X) (Z_1\xrightarrow{z}Z_0).
\end{split}\end{equation*}

Next, consider an object $X_1 \xrightarrow{x} X_0$ in $\vv{\C}$. Then, by definition,
\begin{align*}\Gamma \Upsilon (X_1 \xrightarrow{x} X_0) (Z_1\xrightarrow{z} Z_0)\cong \ker (\coker(\Upsilon_{{\C}} (x))(Z_0) \to \coker(\Upsilon_{{\C}} (x))(Z_1))  \end{align*}
Now $\coker(\Upsilon_{{\C}} (x))(Z_i) \cong\Hom_{\vv{\C}}(Z_i, X_1\xrightarrow{x}X_0)$ by \eqref{homfromovC}, so 
\begin{equation*}\begin{split}
\Gamma \Upsilon (X_1 \xrightarrow{x} X_0) (Z_1\xrightarrow{z} Z_0) &\cong \ker( \Hom_{\vv{\C}}(Z_0, X_1 \xrightarrow{x} X_0) \xrightarrow{ -\circ z} \Hom_{\vv{\C}}(Z_1, X_1 \xrightarrow{x} X_0))\\
&\overset{\eqref{hominvvC}}{\cong} \Hom_{\vv{\C}}(Z_1\xrightarrow{z} Z_0, X_1 \xrightarrow{x} X_0)\\
&=\Upsilon_{\vv{\C}} (X_1 \xrightarrow{x} X_0) (Z_1\xrightarrow{z} Z_0)
\end{split}\end{equation*}
proving $\Upsilon_{\vv{\C}} (X_1 \xrightarrow{x} X_0) \cong \Gamma \Upsilon (X_1 \xrightarrow{x} X_0) $, as required.
\end{proof}

\begin{lemma}
The category $\vv{\C}$ is pretriangulated.
\end{lemma}
\begin{proof}
Direct sums and shifts can be taken component-wise, so it suffices to check that $\vv{\C}$ is closed under cones in the sense that for any dg morphism $\phi$ in $\vv{\C}$, there exists an object in $\vv{\C}$ which represents the cone of $\Upsilon_{\vv{\C}}(\phi)$ in $\vv{\C}\dgmod$. In light of Lemma \ref{Yonedafactor} and since any dg functor preserves cones, it suffices to find an object $C_\phi$ in $\vv{\C}$ such that $\Upsilon(C_\phi)$ is the cone of the morphism $\Upsilon (\phi)$ in $\C\dgmod$.

For this, consider a dg morphism 
\begin{equation}\label{dgmor}
\phi = \vcenter{\hbox{
\xymatrix{ X_1 \ar[rr]^{x} \ar[d]^{\phi_1}&&
 X_0 \ar[d]^{\phi_0}  \\ 
Y_1\ar[rr]^{y} &&Y_0.
}}}
\end{equation}
In particular, there exists $\eta\colon X_0\to Y_1$ such that $\del \phi_0+y\eta=0.$ We observe that, by replacing $Y_1\xrightarrow{y}Y_0$ by the dg isomorphic object 
\begin{align*}X_0\shift{-1}\oplus Y_1\xrightarrow{(-\del(\phi_0),~y)} Y_0,\end{align*}
we can assume that $\del \eta=0$.

Thus, we now consider a dg morphism $\phi$ as in \eqref{dgmor} with $\del \eta=0$.
We claim that the cone $C_\phi = (C_{\phi,1} \xrightarrow{c_\phi} C_{\phi,0})$ is then given by 
\begin{align*}
\xymatrix@R=30pt{
C_{\phi,1}=
\left(
 Y_1\oplus Y_1\shift{1}\oplus X_0\oplus X_1\oplus X_1 \shift{1}, {\left(\begin{smallmatrix}
0&-\id&-\eta&\del\phi_1+\eta x&-\phi_1\\
0&0&0&0&0\\
0&0&0&0&-x\\
0&0&0&0&-\id\\
0&0&0&0&0
\end{smallmatrix}\right)}
\right)
\ar[d]_{c_{\phi}=\left(\begin{smallmatrix}
y&0&0&0&0\\
0&\id&\eta&0&0\\
0&0&0&0&x
\end{smallmatrix}\right)}
\\
C_{\phi,0}=\left(Y_0\oplus Y_1\shift{1}\oplus X_0\shift{1}, {\left(\begin{smallmatrix}
0&-y&-\phi_0\\
0&0&-\eta\\
0&0&0
\end{smallmatrix}\right)}
\right)
.}
\end{align*}

To prove this, we need to compare the two following objects. On the one hand, we have $\Upsilon(C_\phi)$, which is the cokernel of $c_\phi \circ (-) \colon C_{\phi,1}^\vee \to C_{\phi,0}^\vee$ in $\C\dgmod$. On the other hand, $\phi$ induces a morphism $\Upsilon(\phi)$ between the cokernels of $x\circ (-)$ and $y\circ (-)$ in $\C\dgmod$, and we can consider the cone of $\Upsilon(\phi)$. We need to verify that these two objects are isomorphic in  $\C\dgmod$. 

For any $Z$ in $\C$, consider 
\begin{align*}\Upsilon(C_\phi)(Z)\cong\coker \left(  \Hom_{\ov{\C}}(Z, C_{\phi,1})\xrightarrow{c_\phi \circ (-)}  \Hom_{\ov{\C}}(Z, C_{\phi,0})\right) \in \Bbbk\dgmod.\end{align*}
Direct computation shows that the upper commutative square in the diagram
\begin{align*}
\xymatrix@R=8pt{
\Hom_\C(Z,Y_1\oplus Y_1\shift{1}\oplus X_0\oplus X_1\oplus X_1 \shift{1})\ar[rr]^-{\left(\begin{smallmatrix}
\id&0&0&0&0\\
0&0&0&0&\id
\end{smallmatrix}\right)\circ(-)}
\ar[dddd]_{\left(\begin{smallmatrix}
y&0&0&0&0\\
0&\id&\eta&0&0\\
0&0&0&0&x
\end{smallmatrix}\right)\circ(-)}
\ar[rrdddd]|-{\left(\begin{smallmatrix}
y&0&0&0&0\\
0&0&0&0&x
\end{smallmatrix}\right)\circ(-)}
&&\Hom_\C(Z,Y_1\oplus X_1 \shift{1})
\ar[dddd]_{\left(\begin{smallmatrix}
y&0\\
0&x
\end{smallmatrix}\right)\circ(-)}
\\\\\\\\
\Hom_\C(Z,Y_0\oplus Y_1\shift{1}\oplus X_0\shift{1})\ar[rr]_-{\left(\begin{smallmatrix}
\id&0&0\\
0&0&\id
\end{smallmatrix}\right)\circ(-)}
\ar@{->>}[ddd]&&
\Hom_\C(Z,Y_0\oplus X_0\shift{1})
\ar@{->>}[ddd]\\\\
\\
\Upsilon(C_\phi)(Z)
\ar[rr]^-{\sim} &&\Upsilon(Y)(Z)\oplus \Upsilon(X\shift{1})(Z)
}
\end{align*}
induces an isomorphism of $\Bbbk$-modules on cokernels as indicated. It remains to show that this isomorphism on cokernels is compatible with the differentials. Every element in  $\Upsilon(C_\phi)(Z)$ can be represented as $\left(\begin{smallmatrix}
\alpha\\
0\\
\beta
\end{smallmatrix}\right)$, for morphisms  $\alpha\colon Z\to Y_0$ and $\beta\colon Z\to X_0\shift{1}$ in $\C$ and we have 
\begin{align*}\del_{\Upsilon(C_\phi)}\begin{pmatrix}
\alpha\\
0\\
\beta
\end{pmatrix}
= \begin{pmatrix}
\del\alpha\\
0\\
\del\beta
\end{pmatrix}
+
\begin{pmatrix}
0&-y&-\phi_0\\
0&0&-\eta\\
0&0&0
\end{pmatrix}
\begin{pmatrix}
\alpha\\
0\\
\beta
\end{pmatrix}=
\begin{pmatrix}
\del\alpha - \phi_0\beta\\
-\eta\beta\\
\del\beta
\end{pmatrix}.
\end{align*}
By the induced map on cokernels, this element is mapped to 
\begin{equation*}\begin{split}
\begin{pmatrix}
\id&0&0\\
0&0&\id
\end{pmatrix}
\left(\del_{\Upsilon(C_\phi)}\begin{pmatrix}
\alpha\\
0\\
\beta
\end{pmatrix}\right)
&=
\begin{pmatrix}
\id&0&0\\
0&0&\id
\end{pmatrix}
\begin{pmatrix}
\del\alpha - \phi_0\beta\\
-\eta\beta\\
\del\beta
\end{pmatrix}
\\
&= \begin{pmatrix}
\del\alpha - \phi_0\beta\\
\del\beta
\end{pmatrix}\\
&= \begin{pmatrix}
\del\alpha\\
\del\beta
\end{pmatrix}
+ \begin{pmatrix}
0&-\phi_0\\
0&0
\end{pmatrix}
\begin{pmatrix}
\alpha\\
\beta
\end{pmatrix}\\
& = \del_{C_{\Upsilon(\phi)}}\begin{pmatrix}
\alpha\\
\beta
\end{pmatrix}\\
&= \del_{C_{\Upsilon(\phi)}}\left(\begin{pmatrix}
\id&0&0\\
0&0&\id
\end{pmatrix}
\begin{pmatrix}
\alpha\\
0\\
\beta
\end{pmatrix}\right),
\end{split}\end{equation*}
so the induced map on the cokernels indeed commutes with the differential.
\end{proof}

We call an object $X$ in a dg category $\DD$ \textbf{compact} in $\DD$ if the \emph{corepresentable} functor
\begin{align*}\Hom_{\DD}(X,-)\colon \DD\to \Bbbk\dgmod\end{align*}
commutes with filtered conical colimits, see e.g. \cite[Section~7.5]{Ri}. More precisely, for a directed set $I$ and a diagram $(X_i)_{i\in I}$ of objects in $\DD$ with dg morphisms between them, there is a canonical dg isomorphism
\begin{align*}\Hom_{\DD}(X,\colim_{i\in I} X_i)\cong \colim_{i\in I} \Hom_{\DD}(X,X_i),\end{align*}
provided that the conical colimit $\colim_{i\in I} X_i$ exists in $\DD$. 
Setting $\DD=\C\dgmod$, we denote the full dg subcategory of $\C\dgmod$ on compact objects by $\C\cdgmod$.

We note that $\Bbbk\dgmod$ is complete and cocomplete with respect to small weighted (and hence, in particular, conical) limits and colimits. This follows from \cite[Corollary~7.6.4]{Ri} using \cite[Example~3.7.5]{Ri} and that $\Bbbk\dgmod$ is the category of cochain complexes of $\Bbbk$-vector spaces enriched over itself. Note that this implies that $\C\dgmod$ is also (co)complete under weighted (co)limits by \cite[Section~3.3]{K} with (co)limits being computed object-wise.

\begin{lemma}\label{Yoncompact}
The image of $\ov{\C}$ under the Yoneda embedding is contained in $\C\cdgmod$.
\end{lemma}
\begin{proof}
We first show that for $X$ in $\C$, $X^\vee$ is compact in $\C\dgmod$. Indeed, for a directed set $I$ and a  diagram $(M_i)_{i\in I}$ in $\C\dgmod$ with dg morphisms, we have a chain of dg isomorphisms
\begin{gather*}
\Hom_{\C\dgmod}(X^\vee,\colim_{i\in I} M_i)\cong \big(\colim_{i\in I} M_i\big)(X) \\
\cong\colim_{i\in I}M_i(X)
\cong \colim_{i\in I}\Hom_{\C\dgmod}(X^\vee,M_i),
\end{gather*}
where the first and last isomorphisms use the dg Yoneda Lemma of \cite[Lemma~7.3.5]{Ri}, the second isomorphism follows from the discussion preceding this lemma.

Note that the inclusion of $\Bbbk\dgmod$ into $\ov{\Bbbk\dgmod}$ is a dg equivalence. Thus, sending a dg $\C$-module $M$ to the dg $\ov{\C}$-module $\ov{M}$ gives a dg equivalence of $\C\dgmod$ and $\ov{\C}\dgmod$. This proves the claim.
\end{proof}

Note that this implies that the category of finitely-generated semifree dg $\C$-modules consists of compact objects.

\begin{lemma}\label{filtered}
Every object in $\C\dgmod$ is isomorphic to a filtered conical colimit of objects in the image of $\Upsilon\colon \vv{\C}\to \C\dgmod$.
\end{lemma}

\begin{proof}Let $M$ be a dg module over $\C$. 
By \cite[Section 3.2]{Ke} we find a semi-free dg module $N$ (with a possibly infinite filtration by free dg modules) together with a surjective dg morphism $N\stackrel{\pi}{\to}M$. Denote $K:=\ker \pi$. As $N$ is semi-free, there exists a filtration $0\subseteq F^1N\leq F^2N\subseteq ...$ of $N$ with free subquotients. We denote by $\mathcal{S}$ the set of pairs $(i,U)$ where $i\geq 0$ and $U$ is a dg submodule of $K\cap F^iN$ such that there exists a compact semifree dg $\C$-module  with a dg surjection onto $U$. It follows that $\mathcal{S}$ is a directed set, where $(i,U)\leq (j,V)$ if and only if $i\leq j$ and $U\subseteq V$. The meet operation is given by the sum of dg submodules. A standard argument now shows that $M$ is dg isomorphic to the filtered conical colimit over quotients $F^iN/U$, where $(i,U)$ is an index from $\mathcal{S}$. We note that all quotients $F^iN/U$ are in the image of the functor $\Upsilon$ since $F^iN$ is in $\ov{\C}$.

Alternatively, the claim also follows from the implication (i) implies (iv) of \cite[(6.11) Theorem]{K2}, with $\vv{\C}^\op$ for $\T$, and the dg equivalence of $\C\dgmod$ and the full dg subcategory of left exact functors in $\vv{\C}\dgmod$ from \cite[Theorem~5.35]{K}.
\end{proof}

\begin{lemma}\label{vv=compact}
The full subcategory $\C\cdgmod$ of $\C\dgmod$ given by compact objects is dg equivalent to the full subcategory on the image of $\vv{\C}$ under $\Upsilon$.
\end{lemma}

\begin{proof}

Assume given an object $X:=(X_1\stackrel{x}{\to} X_0)$ in $\vv\C$. By construction, its image under $\Upsilon$ is the cokernel of the dg morphism $\Upsilon(X_1)\xrightarrow{\Upsilon(x)} \Upsilon(X_0)$ in $\C\dgmod$. Thus, we have an exact sequence 
\begin{align*}\Upsilon(X_1)\xrightarrow{\Upsilon(x)} \Upsilon(X_0)\to \Upsilon(X)\to 0\end{align*}
of dg morphisms in $\C\dgmod$.
Thus, by \cite[(4.14)]{K2}, using that $\Upsilon(X_1)$ and $\Upsilon(X_0)$ are compact by Lemma \ref{Yoncompact}, $\Upsilon(X)$ is also compact.

Conversely, let $M$ be a compact object in $\C\dgmod$. By Lemma \ref{filtered}, $M$ is dg isomorphic to a filtered conical colimit $\colim_{i\in I}M_i$, with $M_i$ in the image of $\Upsilon$ extended to $\vv{\C}$.
Since $M$ is compact, $\Hom_{\C\dgmod}(M,M)\cong \colim_{i\in I}\Hom_{\C\dgmod}(M,M_i).$
Thus, $\id_M$ factors through $M_i\to M$ for some $i$ via some dg morphism. Therefore, $M$ is a dg direct summand of $M_i$ and hence contained in the image of $\Upsilon$ itself.
\end{proof}

\begin{corollary}\label{cokerlemma}
Let $\C$ be a dg category. Then the dg category $\vv{\C}$ has conical cokernels. In particular, $\Z(\vv{\C})$ has cokernels. 
\end{corollary}
\begin{proof}
One easily checks that for a dg morphism $(f_1,f_0)$ from  $X_1\xrightarrow{x} X_0$ to $Y_1\xrightarrow{y} Y_0$, the morphism 
\begin{equation*}
\vcenter{\hbox{
\xymatrix{ Y_1 \ar[rr]^{y} \ar[d]^{\left(\begin{smallmatrix}\id\\ 0\end{smallmatrix}\right)}&&Y_0 \ar[d]^{\id}  \\ 
Y_1\oplus X_0\ar[rr]^{(y,f_0)} &&Y_0
}}}
\end{equation*}
in $\vv{\C}$ satisfies the universal property of a conical cokernel.
\end{proof}

\subsection{Compactness and adjunctions}\label{sec:compadj}

Let $\C$ and $\DD$ be pretriangulated categories. Assume given a dg functor $F\colon \C\to \DD$. We now establish compactness conditions to ensure that $\vv{F}\colon \vv{\C}\to \vv{\DD}$, obtained by extending $F$ as in Section \ref{compactsect}, has a right adjoint $G\colon \vv{\DD}\to \vv{\C}$.

Recall the Yoneda embedding
$\Upsilon=\Upsilon_\C\colon \C\to \C\dgmod$, $X\mapsto X^\vee=\Hom_{\C}(-,X)$.
Then any dg functor $F\colon \C\to \DD$ induces a dg functor 
\begin{align*}F^*\colon \DD\dgmod\to \C\dgmod, \qquad M\mapsto M\circ F.\end{align*}
Given a $\DD$-module $M$, we refer to $F^*M$ as the \textbf{pullback} of $M$ along $F$. By \cite[Section 6.1]{Ke1}, the dg functor $F^*$ is right adjoint to the dg functor $F_*\colon \C\dgmod\to \DD\dgmod$ defined by $M\mapsto F_*M$, where for $D\in\DD$, $F_*M(D)$ is the cokernel of the dg morphism
\begin{gather*}
\bigoplus_{C_1,C_2\in \C}M(C_2)\otimes \Hom_\C(C_1,C_2)\otimes\Hom_\DD(D,FC_1)\xrightarrow{\phi} \bigoplus_{C}M(C)\otimes \Hom_\DD(D,FC),\\
\phi(m\otimes f\otimes g)=M(f)(m)\otimes g-m\otimes F(f)\circ g.
\end{gather*}
The dg functor $F_*$ extends $F$ under the Yoneda embedding, i.e. $F_*\circ \Upsilon_\C=\Upsilon_\DD\circ F$. 

Note that, in particular, since $F_*$ is a left adjoint and hence commutes with cokernels, it restricts to the dg functor $\vv{F}\colon \vv{\C}\to \vv{\DD}$. The following lemma characterizes when this restriction has a right adjoint.

\begin{lemma}\label{compactadj}
The functor $F^*$ restricts to a dg functor from $\vv{\DD}$ to $\vv{\C}$ if and only if for any object $D$ of $\DD$, the dg $\C$-module $F^*(D^\vee)=\Hom_\DD(F(-), D)$ is compact. 
\end{lemma}
\begin{proof}
Using Lemma \ref{vv=compact}, the right adjoint dg functor $F^*$ restricts to a dg functor $R\colon \vv{\DD}\to \vv{\C}$ if and only if $F^*(\Upsilon(D))$ is a compact object in $\C\dgmod$, for any object $D=(D_1\stackrel{d}{\to}D_0)$ in $\vv{\DD}$. If the latter condition is satisfied for all $D\in \vv{\DD}$, then it holds, in particular, for all $D^\vee=\Upsilon(D)$ for $D\in \DD\subseteq \vv{\DD}$. Conversely, assume that $F^*(D^\vee)$ is a compact object in $\C\dgmod$, for any $D\in \DD$. Using a similar argument as in the proof of Lemma \ref{vv=compact}, the dg $\C$-module 
\begin{align*}F^*\Big(\Upsilon\big(D_1\stackrel{d}{\to} D_0\big)\Big)=\coker\big( F^*D_1^\vee\xrightarrow{F^*\Upsilon(d)} F^*D_0^\vee \big)\end{align*}
is a compact object, as the cokernel of a dg morphism between two compact objects is a compact object.
\end{proof}

If a dg functor $F$ satisfies the equivalent conditions of Lemma \ref{compactadj}, we say that $F$ \textbf{has a compact right adjoint}. In the presence of generators for $\C$ and $\DD$, we have alternative characterizations of the existence of a right adjoint for $\vv{F}$, i.e., for $F$ having a compact right adjoint.

If $\C$ has a compact generator $X$, then by Lemma \ref{csf-lemma}, the category $\ov{\C}$ is dg equivalent to that of compact objects in the dg category of right semi-free modules over the dg algebra $A:=\End_\C(X)$. Here, the product in $A$ is $fg:=f\circ g$.

\begin{lemma}\label{compactadjgen} Let $F\colon \C\to \DD$ be a dg functor.
\begin{enumerate}[(a)]
\item\label{compactadjgen1} Assume $\DD$ has a generator $D$. Then $\vv{F}$ has a right adjoint if and only if $F^*(D^\vee)=\Hom_\C(F(-), D)$ is a compact object in $\C\dgmod$.
\item\label{compactadjgen2} Assume $\DD$ has a compact generator $D$, $\C$ has a generator $C$, and denote by $A$ the dg algebra $\End_\C(C)$. Then $\vv{F}$ has a right adjoint if and only if $M_F:=\Hom_\C(F(C), D)$ is a compact object in $A\dgmod$. The action of $A$ on $M_F$ is given by $m\otimes a:=m\circ F(a)$, for $m\in M_F$ and $a\in A$.
\end{enumerate}
\end{lemma}
\begin{proof}
Part \eqref{compactadjgen1} follows from the fact that if $F^*(D^\vee)$ is a compact object, then for any object $X$ in the thick closure of $D$ inside of $\ov{\DD}$, $F^*(X^\vee)$ is also compact. 

Part \eqref{compactadjgen2} follows from Part \eqref{compactadjgen1} under use of the dg equivalence between $\C\dgmod$ and the category of right dg $A$-modules, which is given by sending $X^\vee$ to $X^\vee(C)=\Hom_\C(C,X)$, for $X\in \C$.
\end{proof}

\subsection{The homotopy category of compact objects in semi-free dg modules}\label{homotopy-sec}

In this section, we discuss the passage from dg categories to triangulated categories in the setup used for this paper. 

Let $\C$ be a dg category. We denote the \textbf{homotopy category} of $\C$ by $\K(\C)$ and recall that $\K(\C)$ has the same objects as $\C$ but morphism spaces are given by $0$-th cohomology, i.e.
\begin{align*}\Hom_{\K(\C)}(X,Y)=H^0(\Hom_{\C}(X,Y)).\end{align*}
It follows that if $\C$ is pretriangulated, then $\K(\C)$ is a triangulated category,
see \cite[\S 1, Proposition~2]{BK}, \cite[Section~2.2]{Ke1}. A dg functor $F\colon\C\to \DD$ induces a functor $\K(F)\colon \K(\C)\to \K(\DD)$. The functor $\K(\ov{F})$ is a \emph{triangle functor} (also called exact functor of triangulated categories) \cite[\S 3]{BK}.
Moreover, a dg natural transformation $\tau\colon F\to G$ between two dg functors descends to a natural transformation $\K(\tau)\colon \K(F)\to \K(G)$.

We say that a morphism $f\colon X\to Y$ is a \textbf{homotopy isomorphism} if it descends to an isomorphism in $\K(\C)$.

A dg functor $F\colon \C\to \DD$ is (part of) a \textbf{quasi-equivalence} if, for any pair of objects $X,Y$ in $\C$, it induces quasi-isomor\-phisms (i.e., isomorphisms in all cohomological degrees)
between $\Hom_\C(X,Y)$ and $\Hom_{\DD}(FX,FY)$
 \emph{and} $\K(F)$ is an equivalence of categories \cite[Section~2.3]{Ke}. Note that if $\C$ and $\DD$ are both pretriangulated, closure under shifts implies that $\Hom_\C(X,Y)\to \Hom_{\DD}(FX,FY)$ being a quasi-isomorphism for all $X,Y\in\C$ is equivalent to $\K(F)$ being fully faithful.

Given a dg category $\C$, a morphism $f\colon M\to N$ of dg $\C$-modules is a \textbf{quasi-isomor\-phism} if $\del(f)=0$ and $f$ induces a quasi-isomorphism $f(X)\colon N(X)\to M(X)$ for each object $X$ of $\C$.
The \textbf{derived category} $\DD(\C\dgmod)$ is the localization of the category $\Z(\C\dgmod)$ of dg morphisms (or, equivalently, of $\K(\C\dgmod)$), see Section \ref{dgsect}, by the all quasi-isomorphisms \cite[Section~3.2]{Ke}. The derived category can also be defined as the Verdier quotient of $\K(\C\dgmod)$ by the full subcategory of acyclic modules. Here, a $\C$-module $M$ is \textbf{acyclic} if the complex $M(X)$ is acyclic for any object $X$ of $\C$ \cite[Section~2.2]{Or1}.

The full subcategory of \textbf{perfect objects} of $\C\dgmod$ is given by those dg modules whose images in $\DD(\C\dgmod)$ are in the idempotent completion of compact semi-free modules. The full subcategory of perfect objects yields the subcategory $\DD(\C\dgmod)^{\comp}$ of compact objects in $\DD(\C\dgmod)$ \cite[Corollary~3.7]{Ke}.
 
The following well-known result justifies that we will mostly be working with $\K(\ov{\C})$ as a model for the compact derived category of dg modules over $\C$.

\begin{proposition}\label{compact-triang}
The functor
$\K(\ov{\C})^\circ\hookrightarrow \DD(\C\dgmod)^{\comp}$
forms part of an equivalence of triangulated categories. 
\end{proposition}
\begin{proof}
By \cite[Corollary 3.7]{Ke}, every compact object in $\DD(\C\dgmod)$ is a direct summand of an object in the image of $\K(\ov{\C})$ in the derived category. As the derived category is idempotent complete \cite[Proposition~3.2]{BN}, this implies that the idempotent completion $\K(\ov{\C})^\circ$ (see \cite{BS}) admits a triangle functor into $\DD(\C\dgmod)^\comp$. This functor is fully faithful by \cite[Section~3.2]{Ke} and hence gives an equivalence of categories.
\end{proof}

We remark that the image of $\K(\vv{\C})$ in $\DD(\C\dgmod)$ does not usually consist of compact objects in the latter. In particular, this already fails for finite-dimensional algebras with zero differential and infinite global dimension. Nevertheless, one can consider $\K(\vv{\C})$, as the bounded homotopy category of compact objects in $\C\dgmod$, which motivates the following lemma.

Let $F\colon \C\to \DD$ be a dg functor and recall the adjoint dg functors
\begin{align*}\adj{F_*}{\C\dgmod}{\DD\dgmod}{F^*}\end{align*}
from Section \ref{sec:compadj}.

\begin{lemma}\label{vvFquasieq}
If $F\colon \C\to \DD$ is a quasi-equivalence, then so is $\vv{F}\colon \vv{\C}\to \vv{\DD}$.
\end{lemma}
\begin{proof}
Recall that if $F$ is a dg functor that has a compact right adjoint (cf. Section~\ref{sec:compadj}), then both dg functors $F^*,F_*$ restrict to compact objects to a pair of adjoint dg functors
\begin{align*}\adj{F_*}{\C\cdgmod}{\DD\cdgmod}{F^*}.\end{align*}

We first consider the unit $\eta$ of this adjunction. For a representable dg $\C$-module $X^\vee$, we have $F^*F_*(X^\vee) = \Hom_{\DD}(F(-),FX)$. Then $\eta_{X^\vee} \colon \Hom_{\C}(-,X)\to \Hom_{\DD}(F(-),FX)$ is simply given by applying $F$ to morphism spaces, and is a quasi-isomorphism by assumption. By Lemma \ref{vv=compact}, any compact object $M$ in $\C\dgmod$  is dg isomorphic to a cokernel of a morphism between dg modules in the pretriangulated closure of representable dg $\C$-modules. By right exactness of $F_*$ and exactness of $F^*$, we have a commutative diagram 
\begin{align*}\xymatrix{
X_1^\vee \ar[rr] \ar_{\eta_{X_1^{\vee}}}[d]&& X_0^\vee \ar^{\eta_{X_0^{\vee}}}[d]\ar@{->>}[rr] && M\ar^{\eta_M}[d]\\
F^*F_*(X_1^{\vee}) \ar[rr] && F^*F_*(X_0^{\vee}) \ar@{->>}[rr] && F^*F_*M 
}\end{align*}
of dg morphisms in $\C\cdgmod$.
Since $\eta_{X_1^{\vee}},\eta_{X_0^{\vee}}$ are quasi-isomorphisms, so is the induced morphism $\eta_M$ on the cokernel. Thus $\eta$ is a quasi-isomorphism on all compact objects in $\C\dgmod$.

Now consider the counit of the adjunction. Let $X\in \DD$ and consider the dg $\DD$-module $F_*F^*(X^\vee)$. This is given as the cokernel of 
\begin{gather*}
\xymatrix@C=-20pt{
\bigoplus\limits_{C_1,C_2\in \C}\Hom_{\DD}(FC_2, X)\otimes \Hom_\C(C_1,C_2)\otimes\Hom_\DD(-,FC_1)\ar[d]^{\phi}&h\otimes f\otimes g\ar@{|->}[d]\\
 \bigoplus\limits_{C\in \C}\Hom_{\DD}(FC, X)\otimes \Hom_\DD(-,FC),&hF(f)\otimes g-h\otimes F(f)g.
}
\end{gather*}
There is a natural map $\psi\colon \bigoplus_{C\in \C}\Hom_{\DD}(FC, X)\otimes \Hom_\DD(-,FC) \to \Hom_\DD(-, X)$ induced by composition, which annihilates the image of $\phi$ and hence factors over the cokernel, producing the component $\epsilon_{X^\vee}\colon F_*F^*(X^\vee) \to X^\vee$ of the counit. 

Next, we show that in the special case $X=FY$, for $Y\in \C$, $\epsilon_{(FY)^\vee}$ is a quasi-isomorphism. In fact, we can compute the cokernel of $\phi$ by evaluating the maps $\phi_Z$ and $c_Z$ for any object $Z$ in $\DD$. Note that $\psi_Z$ is surjective, with $\id_Z\otimes f$ being a preimage of $f\in \Hom_{\DD}(Z,FY)$.
Moreover, if $h\in \Hom_\DD(FC,FY)$, $g\in \Hom_\DD(Z,FC)$ are dg morphisms such that $\psi_Z(h\otimes g)=hg=0$, there exists a morphism $f\colon C\to Y$ such that $F(f)=h$ in $H^0(\Hom_D(FC,FY))$, i.e. $F(f)=h+\del(a)$ for some morphism $a$. Thus, we compute
\begin{align*}
\phi(\id\otimes f\otimes g)&= F(f)\otimes g-\id \otimes F(f)g\\
&=(h+\del(a))\otimes g-\id \otimes (h+\del(a))g\\
&=h\otimes g+\del(a)\otimes g-\id \otimes hg-\id\otimes \del(a)g\\
&=h\otimes g+\del(a)\otimes g-\id\otimes \del(a)g\\
&=h\otimes g+\del(a\otimes g-\id\otimes ag),
\end{align*}
where in the last step we use that $\del(g)=0$.
This shows that, after passing to cohomology, $\ker \phi_Z=\im \psi_Z$ and hence the cokernel of $\phi$ is given by the representable functor associated to $FY$. In other words, $\epsilon_{(FY)^\vee}$ gives an isomorphism in $\K(\DD\cdgmod)$.

Now consider a general compact object $X_1^\vee\xrightarrow{x}X_0^\vee\twoheadrightarrow M$, with $X_1,X_0$ in $\ov{\DD}$. Then by right exactness of $F_*$ and exactness of $F^*$, we have a commutative diagram 
\begin{align*}\xymatrix{
F_*F^*(X_1^{\vee}) \ar[rr] \ar_{\epsilon_{X_1^{\vee}}}[d]&& F_*F^*(X_0^{\vee}) \ar^{\epsilon_{X_0^{\vee}}}[d]\ar@{->>}[rr] && F_*F^*M\ar^{\epsilon_M}[d]\\
X_1^{\vee} \ar[rr] && X_0^{\vee} \ar@{->>}[rr] && M
}\end{align*}
of dg morphisms in $\DD\cdgmod$. Arguing similarly to before, the morphism $\epsilon_M$ is a quasi-isomorphism as both $\epsilon_{X_1^{\vee}}$ and $\epsilon_{X_0^{\vee}}$ are quasi-isomorphisms. Thus $\epsilon$ is a quasi-isomorphism on all compact objects in $\DD\dgmod$ which shows that $\K(F^*)$ and $\K(F_*)$ are equivalences. Thus, as categories of compact objects are closed under shift, $F^*$, $F_*$ induce quasi-equivalences on compact objects as claimed. Since $\vv{F}$ corresponds to $F_*$ under the dg equivalence of Lemma \ref{vv=compact}, it is also a quasi-equivalence. 
\end{proof}

\section{Dg \texorpdfstring{$2$}{2}-categories and \texorpdfstring{$2$}{2}-representations}\label{dg2}

In this section, we collect generalities on pretriangulated $2$-categories and define different classes of pretriangulated $2$-representations.

\subsection{Dg \texorpdfstring{$2$}{2}-categories}\label{dg2cats}

We call a $2$-category $\cC$ a \textbf{dg $2$-category} if the categories $\cC(\ti,\tj)$ are   dg categories for any pair of objects $\ti,\tj\in \cC$, and horizontal composition is a dg functor.

A \textbf{dg pseudofunctor} is a pseudofunctor  whose component functors are dg functors and whose coherers are dg isomorphisms. A \textbf{dg $2$-functor} is a dg pseudofunctor whose coherers are identities. A dg pseudofunctor $\Phi\colon \cC\to \cD$ is part of a \textbf{dg biequivalence} if it is surjective on objects up to dg equivalence and each component functor is a dg equivalence.

Given a dg $2$-category $\cC$, we can associate a new dg $2$-category of \textbf{(one-sided) twisted complexes}, denoted by $\ov{\cC}$. It consists of 
\begin{itemize}
\item the same objects as $\cC$;
\item $1$-morphism categories $\overline{\cC(\ti,\tj)}$;
\item  with horizontal composition of two $1$-morphisms $\rX=(\bigoplus_{m=1}^s\rF_m,\alpha)\in \overline{\cC(\tk,\tl)}$ and $X'=(\bigoplus_{n=1}^t \rF'_n,\alpha')\in\overline{\cC(\tj,\tk)}$ given by
\begin{align}\label{orderingconvention}
\rX\circ \rX'=\bigg( \bigoplus_{(m,n)}\rF_m\rF_n', (\delta_{k',l'}\alpha_{k,l}\circ_0\id_{\rF'_{k'}}+\delta_{k,l}\id_{\rF_k}\circ_0\alpha'_{k',l'})_{(k,k'),(l,l')} \bigg),
\end{align}
where pairs $(m,n)$ are ordered lexicographically; 
\end{itemize}
The above composition of $1$-morphism gives a strict operation (for details, see \cite[Proposition 3.5]{LM}).
We say that a dg $2$-category is \textbf{pretriangulated} if the embedding $\cC\hookrightarrow \ov{\cC}$ is part of a dg biequivalence, i.e. if.

We can further extend the dg $2$-category $\ov{\cC}$ by formally adding dg quotients to obtain a dg $2$-category $\vv{\cC}$.

\begin{lemma}\label{vvcC}
There exists a dg $2$-category $\vv{\cC}$ containing $\cC$ as a dg $2$-subcategory. A given dg $2$-functor $\Phi\colon \cC\to \cD$ extends naturally to a dg $2$-functor $\vv{\Phi}\colon \vv{\cC}\to \vv{\cD}$.
\end{lemma}
\begin{proof}
We define $\vv{\cC}(\ti,\tj):=\vv{\cC(\ti,\tj)}$. The composition in $\vv{\cC}$ is defined as 
\begin{align*}
\Big(\rG_1\xrightarrow{g}\rG_0\Big)\Big(\rH_1\xrightarrow{h}\rH_0\Big):=\Big(\rG_1\rH_0\oplus \rG_0\rH_1\xrightarrow{\mat{g\circ_0\id_{\rH_0}&\id_{\rG_0}\circ_0 h}}\rG_0\rH_0\Big),
\end{align*}
for objects $(\rG_1\xrightarrow{g}\rG_0)$ in $\vv{\cC}(\tj,\tk)$ and $(\rH_1\xrightarrow{h}\rH_0)$ in $\vv{\cC}(\ti,\tj)$.
Here, we crucially use the explicit additive structure on $\overline{\C}$ making composition of $1$-morphisms in $\vv{\cC}$ strictly associative.

Given morphisms
\begin{gather*}
(\phi_0,\phi_1)\colon \Big(\rG_1\xrightarrow{g}\rG_0\Big)\to \Big(\rG'_1\xrightarrow{g'}\rG'_0\Big),\\ \quad\text{ and }\quad
(\psi_0,\psi_1)\colon \Big(\rH_1\xrightarrow{h}\rH_0\Big)\to \Big(\rH'_1\xrightarrow{h'}\rH'_0\Big)
\end{gather*}
in $\vv{\cC}$, we define
\begin{align*}(\phi_0,\phi_1)\circ_0 (\psi_0,\psi_1):=\left(\phi_0\circ_0\psi_0, \mat{\phi_1\circ_0\psi_0&0\\0&\phi_0\circ_0\psi_1} \right).
\end{align*}
Note that given a homotopy $\eta\colon \rG_0\to \rG_1'$ rendering the morphism $(\phi_0,\phi_1)$ zero in $\vv{\cC}(\tj,\tk)$, the homotopy $\mat{\eta\circ_0 \psi_0\\0}$ shows that the morphism $(\phi_0,\phi_1)\circ_0 (\psi_0,\psi_1)$ is zero in $\vv{\cC}(\ti,\tk)$. Similarly, if $\epsilon\colon \rH_0\to \rH_1'$ is a homotopy rendering the morphism $(\psi_0,\psi_1)$ zero, use the homotopy $\mat{0\\ \phi_0\circ_0\epsilon}$. Thus, horizontal composition in $\vv{\cC}$ is well-defined.

Next, we define vertical composition $\circ_1$ and the  differential on morphism spaces compo\-nent-wise. One directly verifies that in this way, $\vv{\cC}$ obtains the structure of dg $2$-category.

For a dg $2$-functor $\Phi\colon \cC\to \cD$, the underlying dg functors $\Phi_{\ti,\tj}\colon \cC(\ti,\tj)\to \cD(\Phi\ti,\Phi\tj)$ extend by component-wise application to give dg functors
\begin{align*}\vv{\Phi_{\ti,\tj}}\colon \vv{\cC}(\ti,\tj)\to \vv{\cD}(\Phi\ti,\Phi\tj),\end{align*}
which assemble into a dg $2$-functor as required.
\end{proof}

\begin{corollary}
If the dg $2$-functor $\Phi$ in Lemma \ref{vvcC} is a local quasi-equivalence, meaning that every component functor is a quasi-equivalence, then so is $\vv{\Phi}$.
\end{corollary}

\begin{proof}
This follows directly from Lemma \ref{vvFquasieq}.
\end{proof}

\begin{definition}\label{Cgenadj}Let $\cC$ be a dg $2$-category.
\begin{enumerate}[(a)]
\item\label{Cgenadj1} We say that a dg $2$-category $\cC$ \textbf{has generators} if each category $\ov{\cC(\ti,\tj)}$ has a generator. 
\item\label{Cgenadj2} 
We say that a dg $2$-category $\cC$ \textbf{has compact left adjoints} if for any $\rF$ in $\ov{\cC(\ti,\tj)}$, there exists a morphism ${}^*\rF\in \vv{\cC}(\tj,\ti)$ together with dg morphisms 
\begin{align*}u_\rF\colon \one_\ti\to \rF({}^*\rF)\;\in \;\vv{\cC}(\ti,\ti),\qquad c_\rF\colon ({}^*\rF)\rF\to \one_\tj\;\in\; \vv{\cC}(\tj,\tj)\end{align*}
satisfying the usual identities of an adjunction internal to $\vv{\cC}$.
\end{enumerate}
\end{definition}

Given a dg $2$-category $\cC$, the \textbf{dg idempotent completion} $\cC^\circ$ is the dg $2$-category with the same objects as $\cC$ and $\cC^{\circ}(\ti,\tj)=\cC(\ti,\tj)^{\circ}$, the closure under dg idempotents, for two objects $\ti,\tj$.  We note that by Lemma \ref{Ccirctri}, if $\cC$ is pretriangulated, then $\cC^\circ$ is pretriangulated. 

To a dg $2$-category $\cC$, we can associate the \textbf{$2$-category $\cZ\cC$ of dg $2$-morphisms} which has the same objects and $1$-morphisms as $\cC$ but only those $2$-morphisms which are of degree zero and annihilated by the differential.

The \textbf{homotopy $2$-category} $\cK\cC$ associated to $\cC$ is defined by having the same objects as $\cC$ and categories of $1$-morphisms given by the homotopy categories
\begin{align*}\cK\cC(\ti,\tj):= \K(\cC(\ti,\tj)).\end{align*}
If $\cC$ is pretrianguated, each $\cK\cC(\ti,\tj)$ is triangulated, and we will hence call $\cK\cC$ a triangulated $2$-category. Note that horizontal composition with $1$-morphisms is, in particular, a triangle functor.

\subsection{Dg \texorpdfstring{$2$}{2}-representations}\label{dg2reps}

For the purpose of giving targets for dg $2$-representations in this section, we define the dg $2$-category $\csfcat$ as the $2$-category whose
\begin{itemize}
\item objects are pretriangulated categories $\C$;
\item $1$-morphisms are dg functors between such categories;
\item $2$-morphisms are all morphisms of such dg functors.
\end{itemize}
The dg $2$-subcategory $\csfcat_{\mathrm{g}}$ of $\csfcat$ consists of those dg categories which have a generator.
 
A \textbf{pretriangulated $2$-representation} of a dg $2$-category $\cC$ is a $2$-functor $\bfM\colon \cC \to \csfcat$ such that locally the functors from $\cC(\ti,\tj)$ to $\csfcat(\bfM(\ti),\bfM(\tj))$ are dg functors.
Explicitly, $\bfM$ sends 
\begin{itemize}
\item an object $\ti\in \cC$ to a pretriangulated category $\bfM(\ti)$,
\item a $1$-morphism $\rG \in \cC(\ti,\tj)$ to a dg functor $\bfM(\rG)\colon \bfM(\ti)\to\bfM(\tj)$,
\item a $2$-morphism $\alpha\colon  \rG\to\rH \in \cC(\ti,\tj)$ to a morphism of dg functors
$\bfM(\alpha)\colon \bfM(\rG)\to\bfM(\rH)$.
\end{itemize}

A \textbf{morphism of dg $2$-representations } $\Phi\colon \bfM\to \bfN$ consists of 
\begin{itemize}
\item  dg functors $\Phi_\ti\colon \bfM(\ti)\to \bfN(\ti)$ for each $\ti\in\cC$ and
\item natural dg isomorphisms $\eta_\rF\colon \Phi_\tj\circ \bfM(\rF)\longrightarrow \bfN(\rF)\circ \Phi_\ti$ for each  $\rF\in\cC(\ti,\tj)$ such that for composable $\rF,\rG$
\begin{equation*}
\eta_{\mathrm{F}\mathrm{G}}=(\mathrm{id}_{\mathbf{N}(\mathrm{F})}\circ_0\eta_{\mathrm{G}})\circ_1
(\eta_{\mathrm{F}}\circ_0\mathrm{id}_{\mathbf{M}(\mathrm{G})}).
\end{equation*}
\end{itemize}

The collection of pretriangulated $2$-representations of $\cC$, together with morphisms of dg $2$-representations and modifications satisfying the same condition as in  \cite[Section 2.3]{MM3}, form a dg $2$-category, which we denote by $\cC\tworep$. 

We say a pretriangulated $2$-representation  
\textbf{has generators} if its target is
$\csfcat_{\mathrm{g}}$, that is, if $\bfM(\ti)$ has a generator for any object $\ti$. The $1$-full and $2$-full dg $2$-subcategory of $\cC\tworep$ consisting of pretriangulated $2$-representations with generators is denoted by $\cC\gtworep$. 

A pretriangulated $2$-representation $\bfM$ is \textbf{acyclic} if, for any object $\ti$, any $X\in \bfM(\ti)$ is acyclic, i.e. there exists a morphism $f$ such that $\del(f)=\id_X$.


Note that  a pretriangulated $2$-representation $\bfM$ of a dg $2$-category $\cC$ extends to a pretriangulated $2$-representation $\bfM$ of $\ov{\cC}$.

\begin{lemma}
Extending from $\cC$ to $\ov{\cC}$ and restricting from $\ov{\cC}$ to $\cC$ define mutually inverse dg biequivalences between $\cC\tworep$ and $\ov{\cC}\tworep$.
\end{lemma}

The following lemma describes how to extend pretriangulated $2$-representations from $\cC$ to $\vv{\cC}$.

\begin{lemma}\label{2extend}Let $\cC$ be a dg $2$-category
\begin{enumerate}[(a)]
\item \label{2extend1}
Given a pretriangulated $2$-representation $\bfM$ of $\cC$, we can define a pretriangulated $2$-representation $\vv{\bfM}$ by $\vv{\bfM}(\ti)=\vv{\bfM(\ti)}$. Moreover, $\vv{\bfM}$ extends to a pretriangulated $2$-representation of $\vv{\cC}$.
\item \label{2extend2}
Given a morphism of $2$-representations $\Phi\colon \bfM\to \bfN$ of $\cC$, we obtain a morphism of dg $2$-representations $\vv{\Phi}\colon \vv{\bfM}\to \vv{\bfN}$ of $\vv{\cC}$.
\item \label{2extend3}
Given morphisms of dg $2$-representations $\Phi,\Psi\colon  \bfM\to \bfN$, a modification $\theta\colon \Phi\to \Psi$ extends to a modification $\vv{\theta}\colon \vv{\Phi}\to \vv{\Psi}$.
This assignment yields a faithful dg morphism
\begin{align*}\Hom_{\Hom_{\cC}(\bfM,\bfN)}(\Phi,\Psi)\to \Hom_{\Hom_{\vv{\cC}}(\vv{\bfM},\vv{\bfN})}(\vv{\Phi},\vv{\Psi}).\end{align*}
\end{enumerate}
\end{lemma}
\begin{proof}
We first prove \eqref{2extend1}. The fact that $\vv{\bfM}(\ti):=\vv{\bfM(\ti)}$ defines a dg $2$-representation $\vv{\bfM}$ of $\cC$ follows from the observation that for $\rG\in \cC(\ti,\tj)$ and $\rH$ in $\cC(\tj,\tk)$, the induced dg functor satisfy the equality 
$\vv{\bfM(\rH\rG)}=\vv{\bfM(\rH)}\vv{\bfM(\rG)},$
since these dg functors are defined by component-wise application of $\bfM$.

Given an object $(\rG_1\xrightarrow{g} \rG_0)$ in $\vv{\cC}(\ti,\tj)$, we define
\begin{align*}
\resizebox{\linewidth}{!}{$
\vv{\bfM}(\rG_1\xrightarrow{g} \rG_0)(X_1\xrightarrow{x}X_0)=\Big( \vv{\bfM}(\rG_1)X_0\oplus\vv{\bfM}(\rG_0)X_1\xrightarrow{\mat{\vv{\bfM}(g)(X_0)&\vv{\bfM}(\rG_0)(x)}} \vv{\bfM}(\rG_0)X_0 \Big),$
}\end{align*}
for any $(X_1\xrightarrow{x}X_0)$ in $\vv{\bfM}(\ti)$, and extend to morphisms accordingly. It is readily verified that this extends $\vv{\bfM}$ to a dg $2$-representation of $\vv{\cC}$ proceeding similarly to the proof of Lemma \ref{vvcC}. We note that the dg $2$-representation is pretriangulated as $\vv{\bfM(\ti)}$ is pretriangulated for any object $\ti$, cf. Section \ref{compactsect}.

To prove \eqref{2extend2}, we denote by $\eta_\rG\colon \Phi_{\tj}\circ\bfM(\rG)\to \bfN(\rG)\circ\Phi_{\ti}$ the dg natural isomorphism given as part of the data of a morphism of dg $2$-representation, which is also natural in $\rG\in \cC(\ti,\tj)$. Given an object $(\rG_1\xrightarrow{g}\rG_0)$ in $\vv{\cC}(\ti,\tj)$ and $(X_1\xrightarrow{x}X_0)$ in $\vv{\bfM}(\ti)$, we define the dg morphism 
\begin{align*}
\resizebox{\textwidth}{!}{$
\left(\vv{\eta}_{(\rG_1\xrightarrow{g}\rG_0)}\right)_{(X_1\xrightarrow{x}X_0)}
\colon \vv{\Phi_{\tj}}\circ\vv{\bfM}(\rG_1\xrightarrow{g}\rG_0)(X_1\xrightarrow{x}X_0)\to \vv{\bfN}(\rG_1\xrightarrow{g}\rG_0)\circ\vv{\Phi_{\ti}}(X_1\xrightarrow{x}X_0)
$}
\end{align*} 
by the diagram
\begin{align*}
\xymatrix{
\Phi_\tj\bfM(\rG_1)(X_0)\oplus \Phi_\tj\bfM(\rG_0)(X_1)\ar[rrrr]^-{\mat{\Phi_\tj\bfM(g)_{X_0}&\Phi_\tj\bfM(\rG_0)(x)}}\ar[dd]^{\mat{(\eta_{\rG_1})_{X_0}&0\\0&(\eta_{\rG_0})_{X_1}}}&&&&\Phi_\tj\bfM(\rG_0)(X_0)\ar[dd]^{(\eta_{\rG_0})_{X_0}}\\\\
\bfN(\rG_1)\Phi_\ti(X_0)\oplus \bfN(\rG_0)\Phi_\ti(X_1)\ar[rrrr]^-{\mat{\bfN(g)_{\Phi_\ti(X_0)}&\bfN(\rG_0)\Phi_\ti(x)}}&&&&\bfN(\rG_0)\Phi_\ti(X_0).
}
\end{align*}
This diagram commutes since
\begin{align*}\bfN(g)_{\Phi_\ti(X_0)}\circ\big(\eta_{\rG_1}\big)_{X_0}=\big(\eta_{\rG_0}\big)_{X_0}\circ\Phi_\tj\bfM(g)_{X_0}\end{align*}
by naturality of $\big(\eta_{\rG}\big)_{X_0}$ in $\rG$ applied to the morphism $g$, and 
\begin{align*}\bfN(\rG_0)\Phi_\ti(x)\circ\big(\eta_{\rG_0}\big)_{X_1}=\big(\eta_{\rG_0}\big)_{X_0}\circ\Phi_\tj\bfM(\rG_0)(x)\end{align*}
by naturality of $\eta_{\rG_0}$ applied to the morphism $x$. 

The morphisms $\big(\vv{\eta}_{(\rG_1\xrightarrow{g}\rG_0)}\big)_{(X_1\xrightarrow{x}X_0)}$ thus defined are natural in $(\rG_1\xrightarrow{g}\rG_0)$ and $(X_1\xrightarrow{x}X_0)$ because all their components are instances of morphisms $\big(\eta_\rG\big)_X$ which are natural in $\rG$ and $X$. As the differential in $\vv{\bfN(\tj)}$ is defined component-wise, we see that $\vv{\eta}$ consists of dg morphisms. The equation 
\begin{align*}\eta_{\rH\rG}=(\id_{\bfN(\rH)}\circ_0 \eta_\rG)\circ_1(\eta_\rH\circ_0 \id_{\bfM(\rH)})\end{align*}
implies the corresponding equation for $\vv{\eta}$ since it holds component-wise.

We leave the verification of \eqref{2extend3} to the reader.
\end{proof}

For any dg $2$-category $\cC$  and $\ti$ one of its objects, we define the \textbf{$\ti$-th principal dg $2$-representation} $\bfP_\ti$, which sends 
\begin{itemize}
\item an object $\tj$ to $\overline{\cC(\ti, \tj)}$,
\item a $1$-morphism $\rG$ in $\cC(\tj, \tk)$ to the dg functor $\overline{\cC(\ti, \tj)}\to \overline{\cC(\ti, \tk)}$ induced by composition with $\rG$, 
\item a $2$-morphism to the induced dg morphism of functors.
\end{itemize}

Pretriangulated $2$-representations can also be extended to the dg idempotent completion $\cC^\circ$ of $\cC$. For this, given $\bfM\in \cC\tworep$, consider $\bfM^\circ \in \cC\tworep$ given by $\bfM^\circ(\ti)=\bfM(\ti)^\circ$ and extending the dg functors $\bfM(\rF)$ to $\bfM^{\circ}(\rF)$ via \eqref{functorcirc}. The dg $2$-representation $\bfM^\circ$ extends to a pretriangulated $2$-representation of $\cC^\circ$ where $\rF_e$, for $e\colon \rF\to \rF$ an idempotent in $\cC(\ti,\tj)$, is sent to the dg functor $\bfM^\circ(\rF_e)\colon \bfM^\circ(\ti)\to \bfM^\circ(\tj)$ defined by
\begin{align*}X\mapsto \bfM^\circ(\rF_e)(X)=(\bfM(\rF)(X))_{\bfM(e)_X}, \qquad f\mapsto \bfM(e)_Y\circ \bfM(\rF)(f)\circ \bfM(e)_X,\end{align*}
where $X,Y$ are objects and $f\colon X\to Y$ is a morphism in $\bfM^\circ(\ti)$. Functoriality of $\bfM(\rF_e)$ follows using naturality of $\bfM(e)$.

\begin{lemma}\label{extendtocirc}
There is a dg $2$-functor 
\begin{align*}(-)^\circ\colon \cC\tworep \to \cC^\circ \tworep,\end{align*}
where the assignments on dg $2$-representations, morphisms, and modifications extend the given structures from $\cC$ to $\cC^\circ$.
\end{lemma}
\begin{proof}
A morphism of dg $2$-representations $\Phi\colon \bfM\to \bfN$ naturally extends to a morphism of dg $2$-representations $\Phi^\circ\colon \bfM^\circ\to \bfN^\circ$ of $\cC$ using \eqref{functorcirc}, which is a functorial assignment. For a $1$-morphism $\rF$ and a dg idempotent $e\colon \rF\to \rF$, one defines 
\begin{align*}(\eta_{\rF_e})_X:= \bfN(e)_{\Phi^\circ(X)}\circ (\eta_{\rF})_X\circ \Phi^\circ\bfM(e)_X\end{align*}
and extends $\eta_{\rF_e}$ via \eqref{natcirc}.
This way, it follows that $\Phi^\circ$ indeed commutes with the action of $\cC^\circ$ defined above. Modifications of dg morphisms of dg $2$-representations are naturally extended using \eqref{natcirc}.
\end{proof}

The dg $2$-functor $(-)^\circ\colon \cC\tworep \to \cC^\circ \tworep$ is, in a certain sense, left dg $2$-adjoint to restriction $\bfN\mapsto \left.\bfN\right|_{\cC}$ along the inclusion of $\cC$ into $\cC^\circ$.
\begin{lemma}
There are dg equivalences 
\begin{align*}\adj{E}{\Hom_{\cC}(\bfM,\left.\bfN\right|_{\cC})}{\Hom_{\cC^\circ}(\bfM^\circ,\bfN)}{R},\end{align*}
where $\bfN$ is a dg idempotent complete pretriangulated $2$-representation of $\cC^\circ$ and $\bfM\in \cC\tworep$. 

In particular, the dg $2$-functor $(-)^\circ$ restricts to a dg biequivalence on the full $2$-sub\-cat\-e\-go\-ries of dg idempotent complete pretriangulated $2$-representations.
\end{lemma}
\begin{proof}
The dg functor $E$ is defined by extending a morphism of $2$-representations $\Phi\colon \bfM\to \left.\bfN\right|_{\cC}$ of $\cC$ by setting 
\begin{align*}E(\Phi)_\ti={\Phi_\ti}^\circ.\end{align*}
The target of this dg functor is $\bfN(\ti)\simeq \bfN(\ti)^\circ$ since $\bfN$ is dg idempotent complete. The dg functor $R$ is simply given by restricting to $\bfN(\ti)\subset \bfN(\ti)^\circ$. Given a dg morphism $\Psi\colon \bfM^\circ\to \bfN$ and an object $(X,e)$ in $\bfN(\ti)$, it is clear that $E(\Psi)_\ti(X,e)=(\Psi_\ti(X), \Psi_\ti(e))$ gives a splitting for the idempotent $\Psi_\ti(e)$ and is hence isomorphic to $\Psi_\ti(X,e)$.
\end{proof}

\subsection{Compact pretriangulated \texorpdfstring{$2$}{2}-representations}
\label{sec:compact2rep}

Given a pretriangulated $2$-rep\-re\-sentation $\bfM$ and an object $X$ of $\bfM(\ti)$, we define, for any object $\tj$ of $\cC$, \textbf{evaluation at $X$} to be the dg functor 
\begin{align*}\Eval_X\colon \ov{\cC}(\ti,\tj)\to \bfM(\tj),\qquad \rF\mapsto \bfM(\rF)X.\end{align*}

We say that a pretriangulated $2$-representation $\bfM$ is \textbf{compact} if for any objects $\ti,\tj$ of $\cC$ and any $X\in \bfM(\ti)$, the functors $\Eval_X$ satisfy the equivalent conditions of Lemma \ref{compactadj} (see also Lemma \ref{compactadjgen}), i.e., the extended dg functors $\vv{\Eval}_X$ have right adjoints. In this situation, we will also say that $\Eval_X$ has a compact right adjoint.

\begin{lemma}\label{adjEvonoverline}
Let $\bfM$ be a pretriangulated $2$-representation of a dg $2$-category $\cC$ and $X=(X'\oplus X'', \left(\begin{smallmatrix}0& \alpha \\ 0&0\end{smallmatrix}\right))
\in \bfM(\ti)$. Assume that $\Eval_{X'}, \Eval_{X''}$ have compact right adjoints.  Then $\Eval_X$ has a compact right adjoint.
\end{lemma}

\begin{proof}
For $\rG\in \cC(\ti,\tj), Y\in \bfM(\tj)$, we compute
\begin{align*}
\Hom_{\bfM(\tj)}& (\Eval_X(\rG), Y) \cong \Hom_{\bfM(\tj)} ((\rG X'\oplus\rG X'', \left(\begin{smallmatrix}0& \rG\alpha \\ 0&0\end{smallmatrix}\right)), Y)\\
&\cong \left(\Hom_{\bfM(\tj)} (\rG X'',  Y)\oplus\Hom_{\bfM(\tj)} (\rG X', Y), \left(\begin{smallmatrix}0& -\circ\rG\alpha \\ 0&0\end{smallmatrix}\right) \right)\\
&\cong \resizebox{0.85\textwidth}{!}{$\left(\Hom_{\vv{\cC}(\ti,\tj)} (\rG , \Eval^*_{X''} Y)\oplus\Hom_{\bfM(\ti)} (\rG, \Eval^*_{X'} Y), \left(\begin{smallmatrix}0&\Hom_{\vv{\cC}(\ti,\tj)}(\rG,(\Eval^*_\alpha)_Y) \\ 0&0\end{smallmatrix}\right) \right)$}\\
&\cong \Hom_{\vv{\cC}(\ti,\tj)} \left(\rG , \left(\Eval^*_{X''}\oplus \Eval^*_{X'},  \left(\begin{smallmatrix}0& \Eval^*_\alpha \\ 0&0\end{smallmatrix}\right)\right)Y\right),
\end{align*}
where in the third isomorphism we have used that $\Eval$ is, in fact, a bifunctor by definition of a $2$-representation.
This computation shows that $\Eval_X$ has the compact right adjoint $(\Eval^*_{X''}\oplus \Eval^*_{X'},  \left(\begin{smallmatrix}0& \Eval^*_\alpha \\ 0&0\end{smallmatrix}\right))$. 
\end{proof}

We have the following dg $2$-subcategories of $\cC\tworep$.
\begin{itemize}
\item
The $2$-subcategory of \emph{compact} pretriangulated $2$-representations denoted by $\cC\ctworep$. 
\item
The $2$-subcategory of \emph{compact} pretriangulated $2$-representations \emph{with generators} denot\-ed by $\cC\cgtworep$. 
\end{itemize}

We observe that $\bfP_\ti$ is in $\cC\tworep$. The following lemma establishes when $\bfP_\ti$ belongs to $\cC\ctworep$.

\begin{lemma}\label{Picompact}
Let $\cC$ be a dg $2$-category. Then all principal $2$-representations $\bfP_\ti$ are compact if and only if $\cC$ has compact left adjoints.
\end{lemma}
\begin{proof}
First, assume that $\cC$ has compact left adjoints. Let $\rF\in \cC(\ti,\tj)$ be a $1$-morphism. Using Lemma \ref{compactadj}, we need to show that the $\cC(\tj,\tk)$-module \begin{align*}\Eval_\rF^*(\rG^\vee)=\Hom_{\vv{\cC}(\ti,\tk)}((-)\rF,\rG)\end{align*} is a compact object for any $1$-morphism $\rG$ in $\cC(\ti,\tk)$. 

For the left adjoint ${}^*\rF\in \vv{\cC}(\tj,\ti)$ of $\rF$ , there is a natural dg isomorphism
\begin{align*}\Hom_{\vv{\cC}(\ti,\tk)}(\rH\rF,\rG)\xrightarrow{\sim}\Hom_{\vv{\cC}(\tj,\tk)}(\rH,\rG{}^*\rF), \quad \alpha \mapsto (\alpha\circ_0\id_{{}^*\rF})\circ_1(\id_{\rH}\circ_0 u_\rF),\end{align*}
for any $\rH\in \cC(\tj,\tk)$. This provides a dg isomorphism of dg $\cC(\tj,\tk)$-modules between $\Hom_{\vv{\cC}(\ti,\tk)}((-)\rF,\rG)$ and the image under $\Upsilon$ of $\rG{}^*\rF$, which is, by Lemma \ref{vv=compact}, a compact object in $\cC(\tj,\tk)\dgmod$.

Conversely, assume $\bfP_\ti$ is compact for any object $\ti$ of $\cC$. Then, in particular, for any $\rF\in \cC(\ti,\tj)$ the functor 
$\Hom_{\vv{\cC}(\tj,\tj)}((-)\rF,\one_\tj)$ is compact. Again using Lemma \ref{vv=compact}, there exists a $1$-morphism ${}^*\rF$ in $\vv{\cC}(\tj,\ti)$ such that $\Hom_{\vv{\cC}(\tj,\tj)}((-)\rF,\one_\tj)$ is dg isomorphic to $\Upsilon({}^*\rF)$ as a dg $\cC(\tj,\ti)$-module. Thus, there are dg isomorphisms
\begin{align*}\Hom_{\vv{\cC}(\tj,\tj)}(\rG\rF,\one_\tj)\cong \Hom_{\vv{\cC}(\tj,\ti)}(\rG,{}^*\rF),\end{align*}
natural in $\rG\in \cC(\tj,\ti)$. These dg functors and natural dg isomorphisms extend to inputs $\rG$ from $\vv{\cC}(\tj,\ti)$. In particular, there are natural dg isomorphisms
\begin{align*}\Hom_{\vv{\cC}(\tj,\tj)}(({}^*\rF)\rF,\one_\tj)\cong \Hom_{\vv{\cC}(\tj,\tj)}({}^*\rF,{}^*\rF).\end{align*}
We define $c_\rF$ to be the element corresponding to $\id_{{}^*\rF}$ under this isomorphism. Similarly, $\Hom_{\vv{\cC}(\ti,\tj)}((-)\rF,\rF)$ is compact. Applying this to the identity $1$-morphism $\one_\ti$, there exists a natural dg isomorphism 
\begin{align*}\Hom_{\vv{\cC}(\ti,\tj)}(\one_{\ti}\rF,\rF)\cong \Hom_{\vv{\cC}(\ti,\ti)}(\one_\ti,(\rF{}^*)\rF).\end{align*}
The image of the identity $\id_{\rF}$ gives the element $u_\rF$.
The adjunction axioms for $u_{\rF}$, $c_{\rF}$ now follow from the fact that the  isomorphisms 
$\Hom_{\vv{\cC}(\ti,\tk)}(\rG\rF,\rH)\cong \Hom_{\vv{\cC}(\tj,\tk)}(\rG,\rH{}^*\rF)$ are natural in $\rG,\rH$ adapting the argument of \cite[Section IV.1]{McL}.
\end{proof}

Thus, $\bfP_\ti$ is in $\cC\cgtworep$ provided that $\cC$ has compact left adjoints and generators (see Section \ref{dg2cats}).

\begin{example}
Let $R=\Bbbk[x]$, considered as a dg algebra concentrated in degree zero. We define the dg $2$-category $\cC_R$ to have one object $\bullet$ and morphism category $\cC_R(\bullet, \bullet)$ given by a strictification of $\widehat{\{R\oplus R\otimes_\Bbbk R\}}$ with horizontal composition induced by the tensor product over $R$, and $\one$ corresponding to $R$.

The dg $2$-category $\cC_R$ has its natural dg $2$-representation $\bfN$ on $\bfN(\bullet) = \ov{\{R\}}$ induced by viewing the dg bimodules in $\cC_R(\bullet, \bullet)$ as acting by tensor functors. This pretriangulated $2$-representation is not compact. Indeed, since 
\begin{align*}\Hom_{\bfN(\bullet)} (R\otimes_\Bbbk R\otimes_R R, R)\cong \Hom_{\Bbbk\plmod}(R, R)\end{align*}
is not a compact object in the dg category of dg $R$-$R$-bimodules,
the requirement
\begin{align*}\Hom_{\bfN(\bullet)} (R\otimes_\Bbbk R\otimes_R R, R) \cong \Hom_{\vv{\cC}(\bullet, \bullet)} (R\otimes_\Bbbk R, \Eval^*_R(R))\end{align*}
cannot be satisfied.

On the other hand, we have the trivial dg $2$-representation $\bfM$ with $\bfM(\ti) = \Bbbk\plmod$, on which $R\otimes_\Bbbk R$ acts by $0$, $R$ acts as the identity functor, with the endomorphism $x$ of $R$ acting by zero. This is a compact dg $2$-representation, since 
\begin{align*}\Hom_{\bfM(\bullet)} (\bfM(\rF)\Bbbk, \Bbbk)\cong 
\Hom_{\vv{\cC}(\bullet, \bullet)}(\rF, \Eval^*_\Bbbk(\Bbbk))\end{align*}
is satisfied for any $\rF$ when taking $\Eval^*_\Bbbk(\Bbbk)$ to be the diagram 
$\xymatrix{\one \ar[r]^{x}&\one}$ in $\vv{\cC}(\bullet, \bullet)$.
\end{example}

\subsection{Dg ideals of  \texorpdfstring{$2$}{2}-representations and dg \texorpdfstring{$2$}{2}-subrepresentations}\label{idealsect}
In this section, we collect facts about dg ideals used later in the paper.

A \textbf{dg ideal} of a pretriangulated $2$-representation $\bfM$ of a dg $2$-category $\cC$ is a collection of dg ideals $\bfI(\ti)\subset \bfM(\ti)$ which is closed under the $\cC$-action.

\begin{definition}\label{idealgen}
Let $\bfM$ be a pretriangulated $2$-representation of $\cC$, $\ti\in \cC$, and $f$ a morphism in $\bfM(\ti)$. Define the \textbf{dg ideal $\bfI_{\bfM}(f)$ generated} by $f$ to be the smallest dg ideal of $\bfM$ containing $f$. 
\end{definition}

If $\bfI$ is a dg ideal in a pretriangulated $2$-representation $\bfM$, we can form the quotient $\bfM/\bfI$ acting on 
\begin{align*}\left(\bfM/\bfI\right)(\ti)=\ov{\bfM(\ti)/\bfI(\ti)}.\end{align*}
In particular, given a dg $2$-subrepresentation $\bfN$ of  $\bfM$, we define the \textbf{quotient} $\bfM/\bfN$ to be the quotient of $\bfM$ by the dg ideal generated by $\bfN$.

Let $\bfM$ be a pretriangulated $2$-representation of $\cC$. A pretriangulated \textbf{$2$-sub\-rep\-re\-senta\-tion} $\bfN$ of $\bfM$ is a collection of thick subcategories $\bfN(\ti)\subseteq \bfM(\ti)$ for all objects $\ti$ of $\cC$ such that for all $1$-morphisms $\rG$ and any $2$-morphism $\alpha$ in $\cC(\ti,\tj)$, we have that 
\begin{align*}\bfN\rG(N)=\bfM\rG(N) &\in \bfN(\tj), & \bfN \rG(f)&=\bfM\rG(f)\in \bfN(\tj),&\bfN(\alpha)_N &=\bfM(\alpha)_N, \end{align*}
for all objects $N$ and morphisms $f$ in $\bfN(\ti)$. In particular, the thick subcategories $\bfN(\ti)$ are closed under the $\cC$-action, which is the restriction of the $\cC$-action given by the dg $2$-functor $\bfM$. Note that we require that $\bfN(\ti)$ is closed under forming biproducts and taking dg direct summands  that exist in $\bfM(\ti)$.

Given a collection of pretriangulated $2$-subrepresentations $\lbrace \bfN^\nu\rbrace_{\nu\in I}$ of a fixed pretriangulated $2$-representation $\bfM$, we define the \textbf{sum} $\sum_{\nu\in I} \bfN^\nu$, as the smallest pretriangulated $2$-subrepresentation of $\bfM$ containing all $\bfN^\nu$.

\subsection{Cyclic and quotient-simple dg \texorpdfstring{$2$}{2}-representations}

Recall that by the thick closure of a collection of objects in a dg category, we mean the smallest dg subcategory which is generated by the given objects under shifts, cones, direct sums, dg direct summands, and is closed under dg isomorphisms.

\begin{definition}\label{2repgen}
Given a pretriangulated $2$-representation $\bfM$ of $\cC$ and an object $X\in \bfM(\ti)$, for some $\ti\in \cC$, we denote by $\bfG_\bfM(X)$ the $2$-representa\-tion on the thick closure of $\{\bfM(\rG) X | \rG \in \cC(\ti,\tj), \tj\in \cC\}$ inside $\coprod_{\tj\in\cC}\bfM(\tj)$. We call $\bfG_\bfM(X)$ the dg $2$-subrepresenta\-tion \textbf{$\cC$-generated} by $X$. If $\bfG_\bfM(X)=\bfM$, we say $X$ $\cC$-generates $\bfM$. We call $\bfM$ \textbf{cyclic} if there exists an $X\in \bfM(\ti)$, for some $\ti\in \cC$, which $\cC$-generates~$\bfM$.
\end{definition}

\begin{lemma}\label{cyclicgen}
Let $\cC$ be a dg $2$-category and $\bfM$ a cyclic dg $2$-representation $\cC$-generated by $X\in \bfM(\ti)$. 
\begin{enumerate}[(a)]
\item\label{cyclicgen1} If $\cC$ has generators, then $\bfM\in \cC\gtworep$.
\item\label{cyclicgen2} If $\cC$ has compact left adjoints, and $\Eval_X$ has a compact right adjoint, then $\bfM\in \cC\ctworep$.
\end{enumerate}
\end{lemma}
\begin{proof}
To prove \eqref{cyclicgen1}, we first note that by assumption that $X$ $\cC$-generates $\bfM$, the thick closure of $\{\bfM(\rF) X\}$, where $\rF$ ranges over the $1$-morphisms in $\cC(\ti,\tj)$, is $\bfM(\tj)$. If $\cC(\ti,\tj)$ is the thick closure of $\rF_{\ti,\tj}$, then every $1$-morphism is a dg direct summand of an object in $\ov{\{\rF_{\ti,\tj}\}}$. Hence, every object in $\bfM(\tj)$ is a dg direct summand of an object in $\ov{\{\bfM(\rF_{\ti,\tj})X\}}$ since the dg $2$-functor $\bfM$ preserves shifts and cones, so $\bfM\in \cC\gtworep$.

To prove \eqref{cyclicgen2}, we first argue that, given the assumptions, for any $\rF\in \cC(\ti,\tj)$, $\Eval_{\bfM(\rF) X}$ also has a compact right adjoint. For this, we show that, for any object $Y\in \bfM(\tk)$, the object $\Eval_{\bfM(\rF) X}^*(Y)$ is dg isomorphic to $\Eval_{X}^*(Y){}^*\rF$ and hence compact. This follows from the following computation \begin{align*}
\Hom_{\bfM(\tk)}(\Eval_{\bfM(\rF) X}(\rG),Y)&=\Hom_{\bfM(\tk)}(\Eval_{X}(\rG\rF ),Y)\\
&\cong \Hom_{\vv{\cC}(\ti,\tk)}(\rG\rF,\Eval_X^*(Y))\\
&\cong \Hom_{\vv{\cC}(\tj,\tk)}(\rG,\Eval_X^*(Y){}^*\rF)
\end{align*}
of the right adjoint of $\Eval_{\bfM(\rF) X}$, for $\rG\in \cC(\tj,\tk)$.
Here, the first equality follows using $\bfM(\rG\rF)=\bfM(\rG)\bfM(\rF)$, the second dg isomorphism uses the right adjoint of $\Eval_X$, and the last dg isomorphism uses Lemma \ref{Picompact}.
Thus, by Lemma \ref{compactadj}, $\Eval_{\bfM(\rF) X}$ has a compact right adjoint. For $X'$ in the thick closure of the $\{\bfM(\rF) X\}$, the right adjoint of $\Eval_{X'}$ has compact right adjoint by Lemma \ref{adjEvonoverline} and additivity.
Therefore $\Eval_{X'}$ has a compact right adjoint for any $X'\in \bfM(\ti)$ and $\bfM$ is compact.
\end{proof}

Given a dg $2$-subrepresentation $\bfN$ of a pretriangulated $2$-representation $\bfM$ we define the \textbf{weak closure} $\bfN^{\Diamond}$ as the dg $2$-subrepresentation defined on the full dg subcategories $\bfN^\Diamond(\ti)\subseteq \bfM(\ti)$ on objects whose identities are contained in the ideal generated by $\bfN$. A dg $2$-subrepresentation $\bfN$ is \textbf{weakly closed} in $\bfM$ if $\bfN=\bfN^\Diamond$. 

We say that $X\in\bfM(\ti)$, for some $\ti$ in $\cC$, \textbf{ weakly $\cC$-generates $\bfM$} if $\bfG_\bfM(X)^\Diamond=\bfM$. 

We say $\bfM$ is \textbf{weakly transitive} if the weak closure of any dg $2$-subrepresentation of $\bfM$ equals $\bfM$. Equivalently, $\bfM$ is weakly transitive if it is $\cC$-generated by any non-zero object $X$ in any $\bfM(\ti)$.

We call a pretriangulated $2$-representation $\bfM$ \textbf{quotient-simple} provided it has no proper nonzero dg ideals. Note that, by definition, if $\bfM$ is quotient simple, it is also weakly transitive.

\subsection{Homotopy \texorpdfstring{$2$}{2}-representations}
\label{triang2rep-sec}

Let $\cC$ be a dg $2$-category.
In this section, we describe how a pretriangulated $2$-representation $\bfM$ descends to a $2$-representations $\bfK \bfM$ on the corresponding triangulated categories. For this, we define $\bfK\bfM(\ti):=\K(\bfM(\ti))$ for any object $\ti$ in $\cC$ and 
\begin{align*}\bfK\bfM(\rF):= \K(\bfM(\rF))\colon \bfK\bfM(\ti)\to \bfK\bfM(\tj),\end{align*}
for any $1$-morphism in $\cC(\ti,\tj)$. If $\alpha\colon \rF\to \rG$ is a dg $2$-morphism, then $\alpha$ induces a natural transformation $\bfK\bfM(\alpha)\colon \bfK\bfM(\rF)\to \bfK\bfM(\rG)$. By construction, $\bfK\bfM$ is an additive $2$-representation of $\cZ\cC$, cf. Section \ref{dg2cats}.

\begin{lemma}
The $2$-representation $\bfK\bfM$ is a $2$-representation of the homotopy $2$-category $\cK\cC$.
\end{lemma}
\begin{proof}
If $\alpha\colon \rF\to \rG$ is null-homotopic, i.e., $\del(\beta)=\alpha$ for some morphism $\beta\colon \rF\to\rG$, then
\begin{align*}\bfK\bfM(\alpha)=\bfK\bfM(\del \beta)=\del(\bfK\bfM(\beta)),\end{align*}
and hence $\bfK\bfM(\alpha)=0$ as a natural transformation from $\bfK\bfM(\rF)$ to $\bfK\bfM(\rG)$. Thus, $\bfK\bfM$ descends to the quotient $\cK\cC$ of $\cZ\cC$.
\end{proof}

We call $\bfK\bfM$ the \textbf{homotopy $2$-representation} derived from $\bfM$.
We refrain from giving a general definition of a triangulated $2$-representation of $\cK\cC$ but note that the homotopy $2$-representations $\bfK\bfM$ defined above are triangulated in the sense that they define $2$-functors to the target $2$-category $\triancat$ whose 
\begin{itemize}
\item objects are triangulated categories of the form $\K(\ov{\C})$;
\item $1$-morphisms are functors between such categories that are induced from dg functors and hence are functors of triangulated categories;
\item $2$-morphisms are natural transformations of such functors coming from dg natural transformations of the corresponding dg categories.
\end{itemize}
We will restrict ourselves to triangulated $2$-representations that have dg enhancements as described above. 
We note that the idempotent completion $(\bfK\bfM)^\circ$ is again triangulated by \cite{BS}.

The dg $2$-subcategory $\triancat_{\mathrm{g}}$ of $\triancat$ consists of such triangulated categories which have a generator (also called a classical generator, see e.g. \cite[Definition~1.10]{LO}), i.e., are equal to the thick closure of a single object. We say a $2$-representation of the form $\bfK\bfM$
\textbf{has generators} if its target is
$\triancat_{\mathrm{g}}$, that is, if $\bfK\bfM(\ti)$ has a generator for any object $\ti$. We note that if $\bfM$ has generators as a pretriangulated $2$-representation, then, in particular, $\bfK\bfM$ has generators as a triangulated $2$-representation.

We will also use the following weaker notion of equivalence of dg $2$-representations. 
Let $\bfM$ and $\bfN$ be pretriangulated $2$-representations. 
A morphism $\Phi\colon \bfM\to \bfN$ 
of dg $2$-representations is a 
\textbf{quasi-equivalence} 
if the dg functors 
$\Phi_{\ti}\colon \bfM(\ti)\to \bfN(\ti)$
are quasi-equivalences (see Section \ref{homotopy-sec}). 
In this case, we say that $\bfM$ and $\bfN$ are \textbf{ quasi-equivalent}. In particular, given that our $2$-representations are pretriangulated, the dg morphism $\Phi$ induces an equivalence $\bfK\Phi$ of triangulated $2$-representations if and only if $\Phi$ is a quasi-equivalence.

We say that $X$ is a \textbf{$\cC$-quasi-generator} for a pretriangulated $2$-representation $\bfM$ if the inclusion of $\bfG_\bfM(X)$ into $\bfM$ is a quasi-equivalence (and hence $\bfK\bfG_\bfM(X)$ and $\bfK\bfM$ are equivalent).

For instance, if $\bfM$ is acyclic, the zero object is a $\cC$-quasi-generator for $\bfM$.

\section{Dg algebra \texorpdfstring{$1$}{1}-morphisms}\label{alg1morsec}

Throughout this section, assume that $\cC$ is a  dg $2$-category 
and $\bfM$ a 
pretriangulated $2$-representation in $\cC\tworep$.

\subsection{Compact modules over dg algebra \texorpdfstring{$1$}{1}-morphisms}

We say that $\rA$ is a \textbf{dg algebra $1$-morphism} in $\cC(\ti,\ti)$ if it comes with dg morphisms $u_{\rA}\colon \one_\ti\to \rA$, $m_{\rA}\colon \rA\rA\to \rA$ satisfying the usual axioms of a unitary product internal to $\cC$. Similarly, a morphism of dg algebra $1$-morphisms is just a $2$-morphism commuting with the multiplication and unit morphisms. 
Given $\rA$, define $\prmod{\cC(\ti,\tj)}\rA$ to be the category of \textbf{right dg $\rA$-modules} in $\cC(\ti,\tj)$.

\begin{lemma}\label{modhoms}
For $1$-morphisms $\rG$ in $\cC(\ti,\tj)$ and $\rY\in \prmod{\cC(\ti,\tj)}\rA$, there is a dg isomorphism
\begin{align*}\Hom_{\prmod{\cC(\ti,\tj)}\rA}(\rG\rA,\rY)\cong \Hom_{\cC(\ti,\tj)}(\rG, \rY),\end{align*}
natural in $\rG$ and $\rY$.
\end{lemma}

\begin{proof}
The proof is an adaptation of the purely formal argument in \cite[7.8.12]{EGNO}, noting that the mutually inverse morphisms
\begin{align*}f\mapsto (f\circ (\id_{\rG}\circ_0 u_{\rA})) \qquad \text{ and } \qquad   (\rho_\rY\circ (g\circ_0 \id_{\rA} ))           \mapsfrom g\end{align*}
are dg morphisms since $u_{\rA}$ and $\rho_\rY$, the right $\rA$-action on $\rY$, are.
\end{proof}

We now construct pretriangulated $2$-representations of modules over a dg algebra $1$-morphism $\rA$ in $\vv{\cC}(\ti,\ti)$. 

\begin{definition}[$\vv{\bfM}_{\rA}$, $\bfM_{\rA}$] \label{MAdef}
Let $\rA$ be a dg algebra $1$-morphism in $\vv{\cC}(\ti,\ti)$.
\begin{enumerate}
\item Define $\vv{\bfM}_{\rA}(\tj)=\prmod{\vv{\cC}(\ti,\tj)}\rA$, for any object $\tj$ of $\cC$. Left multiplication by $1$- and $2$-morphisms in $\cC$ 
induces a natural structure of a pretriangulated $2$-representation of $\cC$ on the $\vv{\bfM}_{\rA}(\tj)$ which we denote by 
$\vv{\bfM}_{\rA}$. 
\item Define $\bfM_ {\rA}(\tj)$ to be the thick closure of the set $\{\rG\rA \vert \rG\in\cC(\ti,\tj)\}$ in $\prmod{\vv{\cC}(\ti,\tj)}\rA$. The collection of dg subcategories $\bfM_ {\rA}(\tj)$ forms a pretriangulated $2$-sub\-represen\-ta\-tion of $\vv{\bfM}_{\rA}$ of $\cC$, which we denote by $\bfM_ {\rA}$.
\end{enumerate}
\end{definition}

\begin{lemma}\label{overlineMAprop-new}
The dg $2$-representation $\vv{\bfM_ {\rA}}$ is dg equivalent to $\vv{\bfM}_ {\rA}$.
\end{lemma}
\begin{proof}
For any $\tj$, the dg category $\vv{\bfM}_ {\rA}(\tj)$ is closed under taking conical cokernels of dg morphisms. Thus, the embedding  
$I\colon \bfM_ {\rA}(\tj)\to \vv{\bfM}_ {\rA}(\tj)$ extends to a fully faithful dg functor $\vv{I}\colon \vv{\bfM_ {\rA}}(\tj)\to \vv{\bfM}_ {\rA}(\tj)$. 

Now, any object $\rX\in \vv{\bfM}_ {\rA}$ is dg isomorphic to the (conical) cokernel of the dg morphism
\begin{align*}\id_{\rX}\circ_0 m -a\circ_0\id_{\rA}\colon \rX\rA\rA\to \rX\rA.\end{align*}
where $a\colon \rX\rA\to \rA$ denotes the right $\rA$-action on $\rX$.
Both modules $\rX\rA\rA$, $\rX\rA$ are in the image of $\vv{I}$. But the image of $\vv{I}$ is closed under taking cokernels of dg morphisms. This follows as cokernels of dg morphisms exist in $\vv{\bfM_ {\rA}}(\tj)$ by Lemma \ref{cokerlemma} and these cokernels are preserved by $\vv{I}$. Thus, $\rX$ is in the essential image of $\vv{I}$. This proves the claim. 
\end{proof}

\begin{lemma}\label{MAcompact}
Let $\rA$ be a dg algebra $1$-morphism in $\vv{\cC}(\ti,\ti)$.
\begin{enumerate}[(a)]
\item\label{MAcompact1} If $\cC$ has generators, then $\bfM_{\rA}\in \cC\gtworep$.
\item\label{MAcompact2} If $\cC$ has compact left adjoints, 
then $\bfM_{\rA}\in \cC\ctworep$. 
\end{enumerate}
\end{lemma}
\begin{proof}
If $\cC(\ti,\tj)$ has a generator $\rF_{\ti,\tj}$, then, by definition, $\bfM_{\rA}(\tj)$ is the thick closure of $\rF_{\ti,\tj}\rA$ and thus $\bfM_{\rA}$ has generators, i.e. $\bfM_{\rA}\in \cC\gtworep$. This proves \eqref{MAcompact1}.

To prove \eqref{MAcompact2}, we use Lemma \ref{cyclicgen}\eqref{cyclicgen2} to see that it suffices to show that $\Eval_{\rA}$ has a compact right adjoint.  
By Lemmas \ref{modhoms} and \ref{overlineMAprop-new}, for any $\rG\in\vv{\cC}(\tj,\tk), \rY\in \vv{\bfM}_\rA(\tk)$, we obtain natural isomorphisms
\begin{equation*}\begin{split}
\Hom_{\vv{\bfM}_\rA(\tk)}(\vv{\Eval}_{\rA}(\rG), \rY)&=\Hom_{\vv{\bfM}_\rA(\tk)}(\rG\rA, \rY)\\
&\cong \Hom_{\vv{\cC}(\ti,\tk)}(\rG, \rY)\\
\end{split}\end{equation*}
thus $\Eval_{\rA}$ has the forgetful functor on $\vv{\bfM}_\rA\to\vv{\cC}$ as right adjoint, which is a compact right adjoint by construction.
\end{proof}

Similarly to dg modules, one also defines dg bimodules in a general dg $2$-category. Working with $\vv{\cC}$, the existence of cokernels ensures that these dg bimodules can be composed. Indeed, let $\rA$, $\rB,\rC$ be dg algebra $1$-morphisms in $\vv{\cC}$ and
let  $\rM$ be a dg $\rA$-$\rB$-bimodule $1$-morphism with (commuting) left and right actions $\lambda_\rM$ and $\rho_\rM$ and let $\rN$ be a dg $\rB$-$\rC$-bimodule $1$-morphism with (commuting) left and right actions $\lambda_\rN$ and $\rho_\rN$. We define the \textbf{relative composition} $\rM \circ_\rB \rN$ as the cokernel of the dg morphism
\begin{align*}\rM \rB \rN \xrightarrow{\rho_\rM\circ_0\id_\rN-\id_\rM\circ_0\lambda_\rN}  \rM \rN,\end{align*}
which exists in $\vv{\cC}$ by Lemma \ref{cokerlemma}.
Note that $\rM \circ_\rB \rN$ is now a dg $\rA$-$\rC$-bimodule.

\subsection{The algebra structure on internal homs \texorpdfstring{$[X,X]$}{[X,X]}}\label{inthomsec}

Let $X\in \bfM(\ti)$ be nonzero. We consider the evaluation dg functors
\begin{align*}\Eval_X\colon \cC(\ti,\tj)\to \bfM(\tj), \quad \rF\mapsto \bfM(\rF)X,\end{align*}
which we may extend to $\overline{\cC}(\ti,\tj)$ defined at the beginning of Section \ref{sec:compact2rep}.

If $\Eval_X$ has a compact right adjoint (i.e., satisfies the equivalent conditions of Lemma \ref{compactadj}), then we denote the right adjoint dg functor of $\vv{\Eval}_X$ by 
\begin{align*}[X,-]\colon\vv{\bfM}(\tj) \to   \vv{\cC}(\ti,\tj),\end{align*}
called the \textbf{internal hom}. We note that $[X,-]$ is natural in $X$. 
Furthermore, for the restriction of $[X,-]$ to $\ov{\bfM}(\tj)$, and an object $(\bigoplus_i Y_i, \alpha)$ in $\ov{\bfM}(\tj)$, we have 
\begin{align*}\Big[X, \Big(\bigoplus_i Y_i, \alpha\Big)\Big] \cong \Big(\bigoplus_i [X,Y_i], [X,\alpha]\Big).\end{align*}

For $X\in \bfM(\ti),Y\in \vv{\bfM}(\tj)$, we define 
\begin{align*}\eval_{X,Y} \colon \vv{\bfM}([X,Y])X\to Y\end{align*}
as the image of the identity morphism under the dg isomorphism
\begin{align*}\Hom_{\vv{\cC}(\ti,\tj)}([X,Y], [X,Y])\cong \Hom_{\vv{\bfM}(\tj)}( \vv{\bfM}([X,Y])X,Y).\end{align*}
Note that, as the image of a dg morphism under a dg isomorphism, $\eval_{X,Y}$ is also a dg morphism.

\begin{lemma}\label{algebralemma}
For nonzero $X\in \bfM(\ti)$ such that $\Eval_X$ has a compact right adjoint, the object $\rA_X:=[X,X]\in \vv{\cC}(\ti,\ti)$ carries the structure of a dg algebra $1$-morphism.
\end{lemma}

\begin{proof}
We define the unit morphism $u_{\rA_X}$ as the image of the identity on $X$ under the dg isomorphism
\begin{align*}\Hom_{\vv{\bfM}(\ti)}(X,X)\cong \Hom_{\vv{\cC}(\ti,\ti)}(\one_\ti, \rA_X).\end{align*}
We define the multiplication $m_{\rA_X}$ as the image of the dg morphism  \begin{align*}\eval_{X,X}\circ\vv{\bfM}(\rA_X)(\eval_{X,X})\colon \vv{\bfM}(\rA_X)\vv{\bfM}(\rA_X)X \to \vv{\bfM}(\rA_X)X \to X\end{align*}
under the dg isomorphism
\begin{align*}\Hom_{\vv{\bfM}(\ti)}(\vv{\bfM}(\rA_X\rA_X)X, X) \cong  \Hom_{\vv{\cC}(\ti,\ti)}(\rA_X\rA_X, \rA_X).\end{align*}
It is a routine exercise to check the algebra axioms.
\end{proof}

Note that if $\bfM$ is compact, i.e., $\bfM\in \cC\ctworep$, then we can associate an algebra $1$-morphism $\rA_X$ to \emph{any} nonzero object $X\in \bfM(\ti)$.

\begin{lemma}\label{rightmodule}
For any $Y\in \vv{\bfM}(\tj)$, the object $[X,Y]$ in $\vv{\cC}(\ti,\tj)$ has the structure of a right dg module over $\rA_X$. Moreover, for any morphism $f\colon Y_1\to Y_2 \in \vv{\bfM}(\tj)$, the morphism $[X,f]$ commutes with the right action of $\rA_X$ and satisfies $[X,\del f]=\del[X,f]$.
\end{lemma}

\begin{proof}
We define the right action $\rho_{[X,Y]}\colon [X,Y]\rA_X \to [X,Y]$ as the image of
 \begin{align*}\eval_{X,Y}\circ\vv{\bfM}([X,Y])(\eval_{X,X})\colon \vv{\bfM}([X,Y])\vv{\bfM}(\rA_X)X \to \vv{\bfM}([X,Y])X \to Y\end{align*}
under the dg isomorphism
\begin{align*}\Hom_{\vv{\bfM}(\tj)}(\vv{\bfM}([X,Y]\rA_X)X, Y) \cong  \Hom_{\vv{\cC}(\ti,\tj)}([X,Y]\rA_X, [X,Y]).\end{align*}
It is, again, a routine exercise to check that this morphism satisfies the axioms of a right action. Compatibility with differentials follows from $\eval_{X,Y}$ being a dg morphism.
\end{proof}

Note that Lemma \ref{rightmodule} can be reformulated as saying that $[X,-]$ is a dg functor from $\vv{\bfM}(\tj)$ to $\vv{\bfM}_{\rA_X}(\tj)$ for any $\tj\in \cC$.

\subsection{An equivalence of dg \texorpdfstring{$2$}{2}-representations}\label{mor2rep-sect}

Throughout this subsection, we will assume that $\cC$ has compact left adjoints, see Definition \ref{Cgenadj}(\ref{Cgenadj2}). We additionally assume that $\bfM$ is compact though, strictly speaking, we could fix $X$ and only require that $\Eval_X$ has a compact right adjoint.

\begin{lemma}\label{Fout}
For any $\rG\in \ov{\cC}(\ti,\tj)$, $X\in \bfM(\ti)$, and $Y\in \bfM(\tj)$, we have a dg isomorphism \begin{align*}[X,\bfM(\rG)Y]\cong \rG[X,Y]\end{align*} in $\vv{\cC}(\ti,\tj),$ natural in all variables.
\end{lemma}

\begin{proof}
For any $\rH\in \ov{\cC}(\ti,\tj)$, there is a sequence of dg isomorphisms, natural in all variables,
\begin{equation*}
\begin{split}
\Hom_{\vv{\cC}(\ti,\tj)}(\rH, [X,\bfM(\rG)Y])&\cong \Hom_{\bfM(\tj)}(\bfM(\rH) X,\bfM(\rG)Y)\\
&\cong \Hom_{\vv{\bfM}(\ti)}(({}^*\vv{\bfM}(\rG))\vv{\bfM}(\rH) X,Y)\\
&\cong \Hom_{\vv{\bfM}(\ti)}(\vv{\bfM}(({}^*\rG)\rH) X,Y)\\
&\cong \Hom_{\vv{\cC}(\ti,\tj)}(({}^*\rG)\rH ,[X,Y])\\
&\cong \Hom_{\vv{\cC}(\ti,\tj)}(\rH ,\rG[X,Y]),
\end{split}
\end{equation*}
where the second and third dg isomorphisms follow from $2$-functoriality of $\vv{\bfM}$, which, in particular, implies that $\vv{\bfM}({}^*\rG)$ is left adjoint to $\vv{\bfM}(\rG)$.
\end{proof}

\begin{proposition}\label{2repmor}
There is a morphism of dg $2$-representations from $\vv{\bfM}$ to $\vv{\bfM}_{\rA_X}$ induced by the functor $[X, -]$. 
\end{proposition}

\begin{proof}
For notational simplicity, we denote the dg functor $[X, -]\colon \vv{\bfM}(\tj) \to \vv{\bfM}_{\rA_X}(\tj)$, which indeed has the latter target by Lemma \ref{rightmodule},  by $\Phi$.
By Lemma \ref{Fout}, for each $1$-morphism $\rG$, there is a dg isomorphism of functors $\eta_\rG\colon \Phi\circ \vv{\bfM}(\rG) \to \vv{\bfM}_{\rA_X}(\rG)\circ \Phi$, natural in $\rG$.

To prove that the dg functor $\Phi$ together with the natural dg isomorphism $\eta$ indeed induces a morphism of $2$-representations, it hence suffices to check that \begin{align*}\eta_{\rG\rH} = (\id_{\vv{\bfM}_{\rA_X}(\rG)}\circ_0 \eta_\rH)\circ_1(\eta_\rG \circ_0 \id_{\vv{\bfM}(\rH)}),\end{align*}
for any objects $\tk,\tl$ and $1$-morphism $\rH\in \ov{\cC}(\tj,\tk)$, $\rG\in \ov{\cC}(\tk,\tl)$, $\rK\in \ov{\cC}(\ti,\tl)$.

This follows from commutativity of 
\begin{align*}
\xymatrix{
\Hom_{\vv{\bfM}(\tl)}(\vv{\bfM}(\rK)X,\vv{\bfM}(\rG\rH)Y)\ar^{\sim}[rr]\ar^{\sim}[d]&&\Hom_{\vv{\bfM}(\tk)}(\vv{\bfM}(({}^*\rG)\rK)X,\vv{\bfM}(\rH)Y)\ar^{\sim}[d]\\
\Hom_{\vv{\bfM}(\tj)}(\vv{\bfM}({}^*(\rG\rH)\rK)X,Y)\ar^{\sim}[d] &&\Hom_{\vv{\bfM}(\tj)}(\vv{\bfM}(({}^*\rH)({}^*\rG)\rK)X,Y)\ar^{\sim}[d]\\
\Hom_{\vv{\cC}(\ti,\tj)}({}^*(\rG\rH)\rK,[X,Y]) \ar^{\sim}[d]&&\Hom_{\vv{\cC}(\ti,\tj)}(({}^*\rH)({}^*\rG)\rK,[X,Y])\ar^{\sim}[d]\\
\Hom_{\vv{\cC}(\ti,\tl)}(\rK,\rG\rH[X,Y]) &&\Hom_{\vv{\cC}(\ti,\tk)}(({}^*\rG)\rK,\rH[X,Y])\ar_{\sim}[ll]\\
}\end{align*}
in exactly the same vein as in \cite[Lemma 6]{MMMT}.
\end{proof}

\begin{proposition}\label{fullyfaithfulmor}
The morphism of dg $2$-representations constructed in Proposition \ref{2repmor} is fully faithful when restricted to $\bfG_\bfM(X)$, the pretriangulated $2$-subrepresentation $\cC$-generated by $X$.
\end{proposition}

\begin{proof}
First, note that 
\begin{equation*}
\begin{split}
\Hom_{\vv{\bfM}_{\rA_X}(\tj)}([X,\bfM(\rG) X],[X,\bfM(\rH) X])&\cong \Hom_{\vv{\bfM}_{\rA_X}(\tj)}(\rG\rA_X,[X,\bfM(\rH) X])\\
&\cong\Hom_{\vv{\cC}(\ti,\tj)}(\rG, [X,\bfM(\rH) X])\\
&\cong\Hom_{\vv{\bfM}(\tj)}(\bfM(\rG)X,\bfM(\rH) X)\\
&\cong\Hom_{\bfM(\tj)}(\bfM(\rG)X,\bfM(\rH) X)
\end{split}
\end{equation*}
where the second isomorphism is given by Lemma \ref{modhoms}.
This shows that the functor $[X,-]$ is full and faithful for objects of the form $\bfM(\rG) X$.

From the definition of morphism spaces in $\bfM(\tj)\simeq \overline{\bfM}(\tj)$, it is immediate that 
\begin{align*}\Hom_{\vv{\bfM}_{\rA_X}(\tj)}([X,Y],[X,Z])\cong\Hom_{\bfM(\tj)}(Y,Z),\end{align*}
for all $Y,Z$ in the thick closure of $\{\bfM(\rG)X| \rG\in \cC(\ti,\tj)\}$, completing the proof.
\end{proof}

The dg $2$-representation $\bfG_\bfM(X)$ may not be dg equivalent to $\bfM_{\rA_X}$ as the latter is closed under dg idempotents but the former is not. However, when replacing $\bfM$ by the dg idempotent completion $\bfM^\circ$, we obtain an equivalence of dg $2$-representation between $\bfG_{\bfM^\circ}(X)$ and $\bfM_{\rA_X}$.

\begin{proposition}\label{prop:GMequivMA}
The morphism of dg $2$-representations from Proposition \ref{2repmor} induces an equivalence of dg $2$-representation $\bfG_{\bfM^\circ}(X)\to \bfM_{\rA_X}$.
\end{proposition}
\begin{proof}
Note that $\rF \rA\cong [X,\rF X]$ and $[X,-]$ commutes with taking dg direct summands. Hence, the full image of $[X,-]$ is the thick closure of the set $\{\rF\rA_X\,|\,\rF\}$.
\end{proof}

As $X$ $\cC$-generating $\bfM$ implies that $X$ also $\cC$-generates $\bfM^\circ$, we obtain the following result. 

\begin{corollary}\label{cor:stronggenMAequiv} If the object $X$ $\cC$-generates $\bfM$, then the dg $2$-representations $\bfM_ {\rA_X}$ and $\bfM^{\circ}$ are dg equivalent.
\end{corollary}

Observe that, if $\bfM(\tj)=\widehat{\{X_\tj\}}$, then $\bfM_{\rA_X}(\tj)\simeq \widehat{\{[X,X_\tj]\}}$
 and $\bfM_ {\rA_X}(\tj)$ is hence in $\gcsfcat$. Hence, if such $X_\tj$ exists for any $\tj$, then $\bfM_ {\rA_X}\in \cC\gtworep$.

\begin{corollary}
Provided that $X$ $\cC$-generates $\bfM$, the morphism of dg $2$-representations $\vv{\bfM}\to\vv{\bfM}_{\rA_X}$ constructed in Proposition \ref{2repmor} is a dg equivalence. 
\end{corollary}

\begin{proof}
We note that the equivalences from Corollary \ref{cor:stronggenMAequiv} and Lemma \ref{overlineMAprop-new} give an equivalence of $\vv{\bfM^\circ}\simeq \vv{\bfM}$ and $\vv{\bfM_{\rA_X}}\simeq\vv{\bfM}_{\rA_X}$.
\end{proof}

\begin{corollary}\label{corMisMA}
If $\bfM$ is cyclic, then there exists a dg algebra $1$-morphism $\rA$ in $\cC$, such that
$\bfM^\circ$ is dg equivalent to $\bfM_\rA$.
\end{corollary}

\begin{proof} This is seen by choosing a $\cC$-generator $X$ of $\bfM$, and setting $\rA=\rA_X$.
\end{proof}

If $\bfM$ only has a weak $\cC$-generator, the following weaker statement still holds true.

\begin{proposition}
Let $X$ be a weak $\cC$-generator for $\bfM$. Then $[X,-]$  restricts to a full and faithful morphism of pretriangulated $2$-representations 
$\bfM^\Diamond\to \bfM_{\rA_X}^\Diamond.$ Here, the target is the weak closure taken in $\vv{\bfM}_{\rA_X}$.
\end{proposition}
\begin{proof}
Assume given an object $Y\in\bfM^\Diamond(\ti)$, we find an object $Y'\in \bfM(\ti)$ together with (non-dg) morphisms $\iota\colon Y\to Y', p\colon Y'\to Y$ such that $\id_Y=p\iota$. As $[X,-]$ is a functor, $[X,p][X,\iota]$ is the identity on $[X,Y]$ and hence the latter is contained in the closure $\bfM_{\rA_X}^\Diamond$. The arguments in the proof of Proposition \ref{fullyfaithfulmor} can again be used to prove that the functor $[X,-]$ to the closure $\bfM^\Diamond(\ti)$ is fully faithful. For this, observe that, for $1$-morphisms $\rG,\rG'\in \ov{\cC}(\ti,\tj),$ and objects $Y,Y'\in \bfM(\tj)$ together with morphisms 
$\id_Y=Y\xrightarrow{\iota} \bfM(\rG)X\xrightarrow{p}Y$ and $\id_{Y'}={Y'}\xrightarrow{\iota'} \bfM(\rG')X\xrightarrow{p'}{Y'}$, the diagram
\begin{align*}
\xymatrix{
\Hom(\bfM(\rG)X,\bfM(\rG')X)\ar@/^/[d]^{[X,p']\circ(-)\circ[X,\iota]}\ar[rr]^{[X,-]}&& \Hom([X,\bfM(\rG)X],[X,\bfM(\rG')X])\ar@/^/[d]^{p'\circ(-)\circ\iota}\\
\Hom(Y,\bfM(\rG')X)\ar@/^/[u]^{[X,\iota'](-)\circ [X,p]}\ar[rr]^{[X,-]}&& \Hom([X,Y],[X,\bfM(\rG')X])\ar@/^/[u]^{\iota'\circ(-)\circ p}.
}
\end{align*}
commutes. The horizontal maps are the application of a dg functor and hence commute with differentials. Although the vertical maps are not necessarily dg morphisms, they display the lower row as direct summands of the upper row. Hence, the lower horizontal map is the restriction of an isomorphism and hence a dg isomorphism.
\end{proof}
The morphism $[X,-]$ of dg $2$-representations, when extended to the weak closures, is not necessarily an equivalence. For example, take $\cC$ to be the trivial dg $2$-category with single object $\bullet$, $1$-morphism $\one=\one_\bullet$ and $\End_{\cC(\bullet,\bullet)}(\one)=\Bbbk$. Let $\bfM(\bullet)$ consist of acyclic objects in $\Bbbk\plmod$. Then any generator of $\bfM$ yields $\bfM_{\rA_X}^\Diamond(\bullet)\simeq \Bbbk\plmod$, that is, $\bfM_{\rA_X}^\Diamond$ is the natural pretriangulated $2$-representation, into which $\bfM=\bfM^\Diamond$ embeds as a proper pretriangulated $2$-subrepresentation.

\subsection{Morphisms of dg \texorpdfstring{$2$}{2}-representations and dg algebra \texorpdfstring{$1$}{1}-morphisms}

In this section, we again assume that the dg $2$-category $\cC$ has compact left adjoints, see Definition \ref{Cgenadj}(\ref{Cgenadj2}), and that $\bfM$, $\bfN$ are compact pretriangulated $2$-representations.

\begin{lemma}\label{bimodulelemma}
Let $\rA,\rB$ be dg algebra $1$-morphisms in $\vv{\cC}$ and assume given a morphism $\Phi\colon \bfM_\rA\to \bfM_\rB$ of dg $2$-representations. Then there exists a dg $\rA$-$\rB$-bimodule $\rX$ and a dg isomorphism $\Phi\cong (-)\circ_{\rA}\rX$.
\end{lemma}
\begin{proof} Given $\rA,\rB$, let $\ti,\tj$ denote the objects such that $\rA\in \vv{\cC}(\ti,\ti)$, $\rB\in \vv{\cC}(\tj,\tj)$.
Assume that we have a morphism of dg $2$-representations $\Phi\colon \bfM_\rA\to\bfM_\rB$. We note that $\Phi$ induces a morphism $\vv{\bfM}_{\rA}\to \vv{\bfM}_{\rB}$ of dg $2$-representations of $\vv{\cC}$ by Lemma \ref{2extend}\eqref{2extend2}, which we will also simply denote by $\Phi$. Let $\Phi_\tk$ denote the dg functor $\vv{\bfM}_\rA(\tk)\to\vv{\bfM}_\rB(\tk)$ underlying this morphism of dg $2$-representations. Define $\rX:=\Phi_\ti(\rA)\in \vv{\cC}(\tj,\ti)$. As $\rA\in \bfM_\rA(\ti)$, $\rX$ is a right $\rB$-module in $\bfM_\rB(\ti)$. Furthermore, $\rX$ carries a left $\rA$-module structure via the composition
\begin{align*}\rA\Phi_{\ti}(\rA) = \vv{\bfM}_\rB(\rA)\Phi_{\ti}(\rA) \cong  \Phi_{\ti}(\vv{\bfM}_\rA(\rA)\rA)=\Phi_{\ti}(\rA\rA) \xrightarrow{\Phi_{\ti}(\mu_\rA)} \Phi_{\ti}(\rA),\end{align*}
where the dg isomorphism is given by the natural transformation $\eta$ included in the data of a morphism of $2$-representations. Thus, $\rX$ is an $\rA$-$\rB$-bimodule in $\vv{\cC}(\tj,\ti)$, with $\rX_\rB\in \bfM_\rB(\ti)$.

Next, we show that there exists a dg isomorphism $\Phi\cong (-)\circ_{\rA}\rM$. Indeed, for any object $\rY$ in $\bfM_{\rA}(\tk)$, we have
\begin{align*}
\rY\circ_{\rA}\rX&=\coker \left(\rY\rA \rX\xrightarrow{\rho_{\rY}\circ_0\id_{\rX}-\id_{\rY}\circ_0 \lambda_{\rX}} \rY\rX\right)\\
&\cong \coker \left(\Phi_{\tk}(\rY\rA \rA)\xrightarrow{\Phi_{\tk}(\rho_{\rY})\circ_0\id_{\rA}-\id_{\rY}\circ_0 \Phi_{\tk}(\mu_{\rA})} \Phi_{\tk}(\rY\rA)\right)\\
&\cong \Phi_{\tk}\left(\coker \left(\rY\rA \rA\xrightarrow{\rho_{\rY}\circ_0\id_{\rA}-\id_{\rY}\circ_0 \mu_{\rA}} \rY\rA\right)\right)\\
&\cong \Phi_{\tk}(\rY).
\end{align*}
Here, the second step uses that $\Phi$ commutes with the $\vv{\cC}$-action and the definition of the left action on $\rX=\Phi_\ti(\rA)$, the third step uses that the dg functor $\Phi_\tk$ preserves cokernels of dg morphisms (as it is induced by a dg functor from $\bfM_\rA(\tk)$ to $\bfM_\rB(\tk)$), and the final step uses that $\rY$ and $\rY\circ_\rA \rA$ are dg isomorphic. The dg isomorphisms used above are all natural with respect to morphisms of right $\rA$-modules in $\vv{\cC}$. In fact, a morphism of such modules induces morphisms of the diagrams of the displayed cokernels, and the dg isomorphism $\eta$ used in the third step is natural in $\rY$.
\end{proof}

Denoting by $\cC\intworep$ the dg $2$-category of internal dg $2$-representations, it follows from Lemma \ref{corMisMA} and Lemma \ref{bimodulelemma} that $\cC\intworep$ is dg biequivalent to the $2$-full dg $2$-subcategory of $\cC\cgtworep$ consisting of pretriangulated $2$-representations which are dg idempotent complete and cyclic.

\begin{proposition}\label{repmorvsalgmor}
Let $\Phi\colon \bfM\to \bfN$ be a morphism of dg $2$-representations and $\alpha\colon \rA\to\rB$ be a dg morphism of dg algebra $1$-morphisms in $\vv{\cC}(\ti,\ti)$. 
\begin{enumerate}[(a)]
\item\label{repmorvsalgmor1} For any object $X$ in $\bfM(\ti)$,  $\Phi$ induces a dg morphism of dg algebra $1$-morphisms
$\alpha_\Phi\colon \rA_X\to \rA_{\Phi_\ti X}.$ 
\item\label{repmorvsalgmor2} The morphism $\alpha$ induces an essentially surjective morphism of dg $2$-representa\-tions $\Phi^\alpha\colon \bfM_\rA\to\bfM_\rB$.
\item\label{repmorvsalgmor3} We have $\alpha=\alpha_{\Phi^\alpha}$.
\item\label{repmorvsalgmor4} $\Phi$ is full if and only if $\alpha_\Phi$ is an epimorphism in $\vv{\cC}(\ti,\ti)$.
\item\label{repmorvsalgmor5} $\Phi^\alpha$ is full if and only if $\alpha$ is an epimorphism in $\vv{\cC}(\ti,\ti)$.
\end{enumerate}
\end{proposition}
\begin{proof}
\begin{enumerate}[(a)]
\item 
For simplicity of notation, we set $Y=\Phi_\ti X$. Then, for any $\rG\in \cC(\ti,\ti)$, the functor $\Phi_\ti$ induces a morphism
\begin{align*}\Hom_{\bfM(\ti)}(\bfM(\rG)X,X)\to \Hom_{\bfN(\ti)}(\bfN(\rG)Y,Y),\end{align*} natural in $\rG$.
Transferring this via the adjunction isomorphism, we obtain a morphism 
\begin{align*}\Hom_{\vv{\cC}(\ti,\ti)}(\rG,[X,X])\to \Hom_{\vv{\cC}(\ti,\ti)}(\rG,[Y,Y]),\end{align*} again natural in $\rG$.
This implies that there is a dg morphism $\alpha_\Phi\colon [X,X]\to [Y,Y]$, such that the second morphism of morphism spaces is given by post-composition with $\alpha_\Phi$. Thanks to naturality of all constructions, $\alpha_\Phi$ preserves the algebra structure.
\item Let $\rM\in \bfM_\rA(\ti)$. We define $\Phi^\alpha(\rM)=\rM\circ_{\rA} \rB$, noting the natural structure on $\rB$ as a left dg $\rA$-module. It is immediate that this assignment induces a morphism of dg $2$-representations. Since $\bfM_\rB$ is thick closure of the $\rG\rB$, for all $1$-morphisms $\rG$ in $\cC$, essential surjectivity follows from $\rG\rB\cong \rG\rA\circ_{\rA}\rB$.
\item To prove $\alpha=\alpha_{\Phi^\alpha}$, consider the diagram
\begin{align*}\xymatrix@C=0pt{ 
\Hom_{\vv{\cC}(\ti,\ti)}(\rG,\rA) \ar[dd]_{\alpha_{\Phi^\alpha}\circ -}\ar@{=}[r]^{\sim}&\Hom_{\bfM_\rA(\ti)}(\rG\rA,\rA)\ar[d]^-{\Phi^\alpha}
\\
&\Hom_{\bfM_\rB(\ti)}(\rG\rA\circ_\rA\rB,\rA\circ_\rA\rB)\ar[d]^-{\sim}
\\
\Hom_{\vv{\cC}(\ti,\ti)}(\rG,\rB)\ar@{=}[r]^{\sim}
&\Hom_{\bfM_\rB(\ti)}(\rG\rB,\rB) 
}\end{align*}
and trace $g\in \Hom_{\vv{\cC}(\ti,\ti)}(\rG,\rA)$ along the path going down, to the right, and then up again. The image of $g$ going down is $\mu_\rA\circ(g\circ_0\id_\rA)$, which is mapped to $(\mu_\rA\circ_\rA\id_\rB)\circ(g\circ_0\id_\rA\circ_\rA\id_\rB)$ by $\Phi^\alpha$, and then in the natural isomorphism is identified with $\mu_\rB\circ(\alpha\circ_0\id_\rB)\circ(g\circ_0\id_\rB) = \mu_\rB\circ((\alpha\circ_1 g)\circ_0\id_\rB)$. The latter is then sent to $\alpha\circ_1 g$ going back up, and hence $\alpha_{\Phi^\alpha}=\alpha$, as claimed.

\item Assume $\Phi$ is full. Then for any $\rG$ in $\cC(\ti,\ti)$, the morphism \begin{align*}\Phi_\ti\colon \Hom_{\bfM(\ti)}(\bfM(\rG)X,X)\to \Hom_{\bfN(\ti)}(\bfN(\rG)Y,Y)\end{align*} is an epimorphism, implying that the morphism \begin{align*}\Hom_{\vv{\cC}(\ti,\ti)}(\rG,[X,X])\to \Hom_{\vv{\cC}(\ti,\ti)}(\rG,[Y,Y]),\end{align*} is also an epimorphism. This implies that $\alpha_\Phi$ is an epimorphism.
Conversely, notice that for $\rG$ in $\cC(\ti,\ti)$, under the Yoneda embedding, $\Hom_{\vv{\cC}(\ti,\ti)}(\rG,-)$ is sent to $\Hom_{{\cC}(\ti,\ti)\plmod}(\rG^\vee,-)$, which is exact and thus $\Hom_{\vv{\cC}(\ti,\ti)}(\rG,-)$ is exact. Thus, if $\alpha_\Phi$ is an epimorphism,  so is $\Hom_{\vv{\cC}(\ti,\ti)}(\rG,\alpha_\Phi)$ as well as the induced morphism 
\begin{align*}\Hom_{\bfM(\ti)}(\bfM(\rG)X,X)\to \Hom_{\bfN(\ti)}(\bfN(\rG)Y,Y).\end{align*}
\item This follows from $\alpha=\alpha_{\Phi^\alpha}$ and \eqref{repmorvsalgmor4}.\qedhere
\end{enumerate}
\end{proof}

We call a dg algebra $1$-morphism $\rA$ \textbf{simple} if any dg morphisms of algebra $1$-morphisms $\rA\to\rB$ in $\vv{\cC}$ which is an epimorphism in $\vv{\cC}$ is necessarily an isomorphism.

The following is an analogue of \cite[Corollary~12]{MMMZ}.
\begin{lemma}\label{simplelemma}
Let $\rA$ be a dg algebra $1$-morphism in $\vv{\cC}$. Then $\rA$ is simple if and only if $\bfM_{\rA}$ is quotient-simple. 
\end{lemma}
\begin{proof}
Assume $\rA$ is simple and assume there is a full and essentially surjective morphism of dg $2$-representations $\Phi\colon \bfM_\rA \to \bfN$. Let $X=\Phi(\rA)$, then $\bfN\cong \bfM_\rB$ for $\rB=[X,X]$. By Lemma \ref{repmorvsalgmor}\eqref{repmorvsalgmor1}, this induces a morphism of dg algebra $1$-morphisms $\rA\to\rB$, which is an epimorphism in $\vv{\cC}$ by Lemma \ref{repmorvsalgmor}\eqref{repmorvsalgmor4}, since $\Phi$ is full. By assumption, this implies that this epimorphism is an isomorphism, and $\Phi$ is an equivalence.

Conversely, assume that $\bfM_{\rA}$ is quotient-simple and we have a morphisms of dg algebra $1$-morphisms $\alpha\colon\rA\to\rB$, which is an epimorphism in $\vv{\cC}$. Then we obtain an essentially surjective morphism of dg $2$-representations $\Phi^\alpha\colon\bfM_{\rA}\to \bfM_{\rB}$ by Lemma \ref{repmorvsalgmor}\eqref{repmorvsalgmor2}, which is full by Lemma \ref{repmorvsalgmor}\eqref{repmorvsalgmor5}, since $\alpha$ is an epimorphism. By quotient-simplicity of $\bfM_{\rA}$, we see that $\Phi^\alpha$ is also faithful and hence a dg equivalence, which implies that $\alpha=\alpha_{\Phi^\alpha}$ is a dg isomorphism.
\end{proof}

Denote by $\cC_\ti$ the $2$-category on the single object $\ti$ with morphism category $\cC(\ti,\ti)$.
We have the following corollary, cf.~\cite[Corollary 4.10]{MMMT}.

\begin{corollary}\label{cor4.10}
There is a bijection between equivalence classes of cyclic dg $2$-representa\-tions $\bfM\in \cC\ctworep$ which are $\cC$-generated by some object $X\in \bfM(\ti)$ and 
equivalence classes of cyclic dg $2$-representations in $\cC_\ti\ctworep$. 

This bijection preserves the class of quotient-simple dg $2$-representations.
\end{corollary}

\begin{proof}
Let $\bfM\in \cC\ctworep$ be $\cC$-generated by $X\in \bfM(\ti)$. Then $\rA_X\in \vv{\cC}(\ti,\ti)$ and hence $\rA_X$ gives rise to a dg $2$-representation of $\cC_\ti$.
Similarly, any dg algebra $1$-morphism in $\vv{\cC_\ti}$ is also a dg algebra $1$-morphism in $\vv{\cC}$, giving the converse map. This yields the first statement. 

Since simplicity of $\rA_X$ is a property formulated inside of the dg category $\vv{\cC}(\ti,\ti)$, $\rA_X$ is a simple dg algebra $1$-morphism in $\cC_\ti$ if and only if it is a simple dg algebra $1$-morphism in $\cC$. Thus, the second statement follows from Lemma \ref{simplelemma}.
\end{proof}

\subsection{Morita equivalence}\label{sec:Morita}

Two dg algebra $1$-morphisms will be called \textbf{dg Morita equivalent} provided that $\bfM_\rA$ and $\bfM_\rB$ are dg equivalent as pretriangulated $2$-representations of $\cC$.  We emphasize that this notion is distinct from the usual notion of dg Morita equivalence for dg algebras, which considers quasi-equivalences of derived categories of dg modules.

\begin{proposition}\label{dgMoreq}
Let $\rA\in \vv{\cC}(\ti,\ti)$ and $\rB\in \vv{\cC}(\tj,\tj)$ be two dg algebra $1$-morphisms. The following are equivalent:
\begin{enumerate}[(a)]
\item\label{dgMoreq1} $\rA$ and $\rB$ are dg Morita equivalent;
\item\label{dgMoreq2} There exist a dg $\rA$-$\rB$-bimodule $\rX\in \vv{\cC}(\tj,\ti)$ and a dg $\rB$-$\rA$-bimodule $\rY\in \vv{\cC}(\ti,\tj)$ with $\rY_\rA\in \bfM_\rA(\tj)$ and $\rX_\rB\in \bfM_\rB(\ti)$ such that $\rX \circ_\rB \rY\cong \rA$ and $\rY \circ_\rA \rX\cong \rB$.
\end{enumerate}
\end{proposition}

\begin{proof}
The implication \ref{dgMoreq2} $\Rightarrow$ \ref{dgMoreq1} is obvious, since $-\circ_\rA \rX$ and $-\circ_\rB \rY$ provide mutually inverse dg equivalences between $\bfM_\rA$ and $\bfM_\rB$.

For the implication \ref{dgMoreq1} $\Rightarrow$ \ref{dgMoreq2} assume that we have a dg equivalence $\Phi\colon \bfM_\rA\to\bfM_\rB$. Let $\Phi_\tk$ be the dg functor $\bfM_\rA(\tk)\to\bfM_\rB(\tk)$ underlying this morphism of dg $2$-representa\-tions. Define $\rX:=\Phi_\ti(\rA)$ be the dg $\rA$-$\rB$-bimodule inducing $\Phi$ by Lemma \ref{bimodulelemma}. 
We similarly obtain a dg $\rB$-$\rA$-bimodule $\rY$  from the dg functor $\Psi_\tj\colon \bfM_{\rB}(\tj)\to \bfM_{\rA}(\tj)$ which is part of the given dg equivalence. The given dg isomorphisms $\Psi_\tk\circ\Phi_\tk\cong \Id_{\bfM_\rA(\tk)}$ and $\Phi_\tk\circ\Psi_\tk\cong \Id_{\bfM_\rB(\tk)}$ provide dg isomorphisms
\begin{gather*}
A\cong \Psi_\ti\circ\Phi_\ti(\rA)\cong\Psi_\ti(\rX)\cong \rX\circ_\rB \rY, \quad \text{and} \quad 
\rB\cong \Phi_\tj\circ\Psi_\tj(\rB)\cong\Phi_\tj(\rY)\cong \rY\circ_\rA \rX
\end{gather*}
as required.
\end{proof}

\begin{corollary}\label{eq2reps}
Let $\cC$ be a  pretriangulated $2$-category with compact left adjoints. Further, let $\bfM$, $\bfN$ be two compact pretriangulated $2$-representations $\cC$-generated by $X$ and $Y$, respectively. Then the dg $2$-representations $\bfM^\circ$ and $\bfN^\circ$ are dg equivalent if and only if the dg algebra $1$-morphisms $\rA_X$ and $\rA_Y$ are dg Morita equivalent.
\end{corollary}

\subsection{Morita quasi-equivalence}\label{sec:quasi-Morita}

Let $\rA\in\vv{\cC}(\ti,\ti)$ and $\rB\in \vv{\cC}(\tj,\tj)$ be dg algebra $1$-morphisms and $\Phi\colon \bfM_{\rA}\to \bfM_{\rB}$ a morphism of dg $2$-representations. Let $\rX\in\vv{\cC}(\tj,\ti)$ be the dg $\rA$-$\rB$ bimodule inducing $\Phi$ by Lemma \ref{bimodulelemma}. 

\begin{proposition}\label{triangequiv}
The morphism of dg $2$-representations $\Phi\colon \bfM_{\rA}\to \bfM_{\rB}$ is a quasi-equi\-va\-lence if and only if the following conditions hold:
\begin{enumerate}[(a)]
\item\label{triangequiv1} the natural maps $\Hom_{\bfM_\rA(\tk)}(\rF\rA,\rG\rA)\to \End_{\bfM_\rB(\tk)}(\rF\rX,\rG\rX)$ are quasi-iso\-mor\-phisms for all $1$-morphisms $\rF,\rG \in \cC( \ti, \tk)$ for any $\tk\in \cC$;
\item\label{triangequiv2} $\rX$ cyclically generates the homotopy $2$-representation $\bfK\bfM_{\rB}$.
\end{enumerate}
\end{proposition}
 If we have generators $\rF_{\ti,\tk}$ of $\cC(\ti,\tk)$, for $\tk\in \cC$, we can simplify \eqref{triangequiv1} to requiring a quasi-isomorphism $\End_{\bfM_\rA(\tk)}(\rF_{\ti,\tk}\rA)\to \End_{\bfM_\rB(\tk)}(\rF_{\ti,\tk}\rX)$.
\begin{proof}

First, assume that $\Phi$ is a quasi-equivalence. As $\Phi_{\tk}$ sends $\rF\rA$ to $\rF\rA\circ_A \rX \cong\rF\rX$, condition \eqref{triangequiv1} follows. Condition \eqref{triangequiv2} follows since $\rA$ cyclically generates $\bfK\bfM_\rA$.

For the converse, first note that since $\Phi_\tk$ is a dg functor, it descends to a collection of triangle functors $\bfK\Phi_\tk \colon \bfK\bfM_\rA(\tk) \to \bfK\bfM_\rB(\tk)$. Then \eqref{triangequiv1} implies that each $\bfK\Phi_\tk$ is fully faithful, by  \cite[Lemma 4.2 (a)]{Ke1}, since every object in $\bfK\bfM_\rA(\tk)$ is in the thick closure of the the $\rF\rA$, for $\rF\in \cC(\ti,\tk)$. Moreover, $\bfK\Phi_\tk$ is dense by \eqref{triangequiv2}. 
\end{proof}

We also obtain the following analogue of Corollary \ref{cor:stronggenMAequiv} for $\cC$-quasi-generators.
\begin{corollary}\label{cor:quasi-gen}
Let $\bfM$ be a pretriangulated $2$-representation with a $\cC$-quasi-generator $X$. Then the morphism of dg $2$-representations $[X,-]\colon \bfM \to\vv{\bfM}_{\rA_X} $ induces an equivalence between 
$(\bfK\bfM)^\circ$ and $(\bfK {\bfM}_{\rA_X})^\circ$.
\end{corollary}
\begin{proof}
By assumption that $X$ is a $\cC$-quasi-generator for $\bfM$, it follows that the inclusion of $\bfK\bfG_{\bfM}(X)$ into $\bfK\bfM$ is an equivalence of homotopy $2$-representations.

We claim that $(\bfK\bfG_{\bfM}(X))^\circ$ is equivalent to 
$(\bfK\bfM_{\rA_X})^\circ$. Note that the functors
\begin{align*}\K[X,-]\colon \bfK\bfM(\ti)\to \bfK \vv{\bfM}_{\rA_X}(\ti)\end{align*} restrict to fully faithful functors 
$\bfK\bfG_{\bfM}(X)(\ti)\to \bfK\bfM_{\rA_X}(\ti)$ by Proposition \ref{fullyfaithfulmor}. The dg isomorphisms $[X,\rG X]\cong \rG \rA_X$ are natural in the second component and descend to the homotopy categories.
Clearly, all objects $\rG \rA_X$ are in the image of these functors. Using fully faithfulness, passing to the idempotent completions implies the claim. Thus, we conclude that the homotopy $2$-representations $(\bfK \bfM)^\circ$ and $(\bfK{\bfM}_{\rA_X})^\circ$ are equivalent.
\end{proof}

\begin{remark}
If the categories $\bfK\bfM(\ti)$ have a bounded t-structure, then they are idempotent complete by \cite{LC} and hence $\bfK\bfM\simeq (\bfK\bfM)^\circ$. In this case, assuming the setup of Corollary \ref{cor:quasi-gen}, consider the fully faithful functors $\K[X,-]\colon \bfK\bfG_{\bfM}(X)(\ti)\to \bfK\bfM_{\rA_X}(\ti)$. Given an idempotent $e$ in the image, $e=[X,f]$ for a morphism $f$ in $\bfK\bfM(\ti)$ which is also an idempotent and hence splits off a dg direct summand $Y$ in $\bfK\bfG_{\bfM}(X)(\ti)\simeq \bfK\bfM(\ti)$. Thus, $e$ splits off the object $[X,Y]$ as a dg direct summand and, hence, $\bfK\bfM_{\rA_X}(\ti)$ is idempotent complete. Hence, if $\bfM$ has a bounded t-structure, the morphism of dg $2$-representations $[X,-]\colon \bfM\to  {\bfM}_{\rA_X}$ is a quasi-equivalence provided that $X$ is a $\cC$-quasi-generator. 

Algebraic conditions on the categories $\bfM(\ti)$ which imply the existence of a bounded t-structure are provided in \cite[Theorem~1]{Sch}. Namely, if $A$ is a positively graded dg algebra with $A_0$ semisimple such that $\del(A_0)=0$, then $\K(A\csf)$ is closed under taking direct summands.
\end{remark}

\subsection{Restriction and pushforward}

Let $\cC, \cD$ be dg $2$-categories and $\bsfF\colon \cC \to \cD$ a dg $2$-functor.

We obtain a dg $2$-functor $\bsfR = \bsfR_{\bsfF}\colon \cD\intworep\to \cC\tworep$ by precomposition, i.e. $\bsfR (\bfM) = \bfM\circ\bsfF$.

On the other hand, we can consider a dg $2$-functor $\bsfP=\bsfP_{\bsfF}\colon \cC\intworep\to\cD\intworep$, by defining  $\bsfP(\bfM_\rA)= \bfM_{\bsfF(\rA)}$.

\begin{lemma}\label{lem:unit}
For $\bfM_\rA\in \cC\intworep$, there is a morphism of dg $2$-representations $\Phi\colon\bfM\to\bsfR\bsfP\bfM$ induced by $\bsfF$.
\end{lemma}

\begin{proof}
Writing $\bsfR\bsfP\bfM = \bsfR\bfM_{\bsfF(\rA)}$, we define $\Phi_\tj$ by mapping an object $\rX\in \bfM(\tj)$ to $\bsfF(\rX)$ and similarly on morphisms. It is straightforward to check that this defines a morphism of dg $2$-representations.
\end{proof}

Let $\bfM_\rB\in \cD\intworep$. Assuming that $\rB \in\bfM(\ti)$ is a $\cC$-generator for $\bsfR\bfM_{\rB}$, we have a canonical choice of internal dg $2$-representation of $\cC$ dg equivalent to $\bsfR\bfM_{\rB}$ defined by $\bfM_{\rC'}$, for $\rC'={}_{\cC}[\rB,\rB]$, where we distinguish internal homs in $\vv\cC$ and $\vv\cD$ by denoting them by ${}_{\cC}[-.-]$ and ${}_{\cD}[-.-]$, respectively. We can then define $\bsfP\bsfR\bfM_\rB$ as $\bsfP\bfM_{\rC'}$.

\begin{lemma}\label{lem:counit}
Let $\bfM=\bfM_\rB\in \cD\intworep$ and assume that $\rB \in\bfM(\ti)$ is a $\cC$-generator for $\bsfR\bfM$. Then there is an essentially surjective morphism of dg $2$-representations $\Psi\colon\bsfP\bsfR\bfM_\rB \to\bfM_\rB$.
\end{lemma}

\begin{proof}
Set $\rC'={}_{\cC}[\rB,\rB]$, the dg algebra $1$-morphism associated to $\rB$, and $\rC=\bsfF(\rC')$. By the assumption that $\rB$ is a $\cC$-generator for $\bsfR\bfM$, the latter is dg equivalent to $\bfM_{\rC'}$ by Corollary \ref{cor:stronggenMAequiv} as the restriction is dg idempotent complete. Then, by construction, we have an equivalence of dg $2$-representations $\bsfP\bsfR\bfM_\rB \to \bfM_\rC$. 
We define the map $\varphi \colon \rC\to \rB$ as the image of $\id_{\rC'}$ under the chain of isomorphisms
\begin{equation*}
\begin{split}
\Hom_{\vv{\cC}(\ti,\ti)}(\rC',\rC')&\cong \Hom_{\vv{\bsfR\bfM}(\ti)}(\bsfR\bfM(\rC')\rB,\rB)\\
&=\Hom_{\vv{\bfM}(\ti)}(\bsfF(\rC')\rB,\rB)=\Hom_{\vv{\bfM}(\ti)}(\rC \rB,\rB)\\
&\cong \Hom_{\vv{\cD}(\ti,\ti)}(\rC,\rB).
\end{split}
\end{equation*}
Recalling that the inverse dg isomorphisms \begin{align*}\Hom_{\vv{\bfM}(\ti)}(\rF \rB,\rB)\cong \Hom_{\vv{\cD}(\ti,\ti)}(\rF,\rB)\end{align*} for $\rF\in \cC(\ti,\ti)$ are given by

\begin{equation}\label{isoexplicit}
\begin{array}{rrcll}
\alpha \colon & g &\mapsto & g\circ_1(\id_\rF\circ_0 u_{\rB})&\\
&\mathrm{m}_\rB\circ_1(f \circ_0\id_\rB)& \mapsfrom & f & \colon \beta
\end{array},
\end{equation}
we see that, explicitly, $\varphi = \alpha(\eval^{\cC}_{\rB\rB})= \eval^{\cC}_{\rB\rB}\circ_1(\id_\rC\circ_0 u_{\rB})$.

We claim that this is an algebra homomorphism, i.e. $\varphi\circ_1 \mathrm{m}_\rC = \mathrm{m}_{\rB}\circ_1(\varphi\circ_0\varphi)$. In order to prove this, we consider the images of both sides of the claimed equality under the dg isomorphism 
\begin{equation*}
\Hom_{\vv{\cD}(\ti,\ti)}(\rC\rC,\rB) \cong \Hom_{\vv{\bfM}(\ti)}(\rC\rC \rB,\rB). 
\end{equation*}
The strategy is to show that the morphisms on both sides of the claimed equation correspond to $\eval_{\rB\rB}^{\cC}\circ_1(\id_{\rC}\circ_0 \eval_{\rB\rB}^{\cC})$.

First, observe that
\begin{equation}\label{claim1}
\begin{split}
\mathrm{m}_\rB\circ_1(\varphi\circ_0 \id_\rB)&= \beta(\varphi)
= \beta\alpha(\eval_{\rB\rB}^{\cC}) = \eval_{\rB\rB}^{\cC}.
\end{split}
\end{equation}

On the one hand, $\mathrm{m}_{\rB}\circ_1(\varphi\circ_0\varphi)$ corresponds to 
\begin{align*}
\mathrm{m}_{\rB}\circ_1((\mathrm{m}_{\rB}\circ_1(\varphi\circ_0\varphi))\circ_0 \id_\rB).\end{align*}
Consider the diagram
\begin{align*}
\xymatrix{
\rC\rC \rB \ar^{\varphi\circ_0 \id_{\rC \rB}}[rr]\ar_{\id_\rC\circ_0\eval^{\cC}_{\rB\rB}}[drr]&&\rB\rC \rB \ar^{\id_\rB\circ_0\varphi\circ_0\id_\rB}[rr]\ar[drr]^{\id_\rB\circ_0\eval^{\cC}_{\rB\rB}}&&\rB\rB \rB \ar^{\mathrm{m}_\rB\circ_0\id_\rB}[drr]\ar^{\id_\rB\circ_0\mathrm{m}_\rB}[d]&&\\
&& \rC \rB \ar^{\varphi\circ_0\id_\rB}[rr]\ar@/_1pc/_{\eval^{\cC}_{\rB\rB}}[drrrr]&&\rB \rB\ar^{\mathrm{m}_\rB}[drr] &&\rB \rB\ar^{\mathrm{m}_\rB}[d]\\
&&&& &&\rB,
}
\end{align*}
where the two triangles commute by \eqref{claim1}, the leftmost square commutes by the interchange law, and the rightmost square commutes by associativity of $\mathrm{m}_\rB$.
 This shows that $\mathrm{m}_{\rB}\circ_1(\varphi\circ_0\varphi)$ corresponds to $\eval_{\rB\rB}^{\cC}\circ_1(\id_\rC\circ_0 \eval_{\rB\rB}^{\cC})$ as claimed.

On the other hand, $\varphi\circ_1 \mathrm{m}_\rC$ corresponds to 
\begin{equation*}
\begin{split}
\mathrm{m}_\rB\circ_1((\varphi\circ_1 \mathrm{m}_\rC)\circ_0 \id_\rB) &= \mathrm{m}_\rB\circ_1(\varphi\circ_0 \id_\rB)\circ_1( \mathrm{m}_\rC\circ_0 \id_\rB)\\
&=\eval_{\rB\rB}^{\cC}\circ_1( \mathrm{m}_\rC\circ_0 \id_\rB),
\end{split}
\end{equation*}
where the second equality uses \eqref{claim1}.  Now consider the definition of $ \mathrm{m}_\rC = \bsfF(\mathrm{m}_{\rC'})$. Then the map $\mathrm{m}_\rC\circ_0 \id_\rB \colon \rC\rC \rB\to \rC \rB$ in $\vv{\bfM}(\ti)$ is really the map $\bsfF(\mathrm{m}_{\rC'}) \circ_0\id_\rB\colon \bsfF(\rC'\rC') \rB\to \bsfF(\rC')\rB$ in $\vv{\bfM}(\ti)$, which, by definition of $\bsfR$, describes the action of $\mathrm{m}_{\rC'}$ via $\bsfR\bfM_\rB$.
By associativity of the action of $\bsfR\bfM_\rB(\rC')$ via $\eval_{\rB\rB}^{\cC}$, we have 
\begin{align*}
\eval_{\rB\rB}^{\cC} \circ_1 (\bsfF(  \mathrm{m} _{\rC'})  ) \circ_0\id_\rB)  = \eval_{\rB\rB}^{\cC} \circ_1(\bsfF(\id_{\rC'})\circ_0\eval_{\rB\rB}^{\cC} ) = \eval_{\rB\rB}^{\cC} \circ_1(\id_{\rC}\circ_0\eval_{\rB\rB}^{\cC} ) 
\end{align*}
as claimed.

Since $\varphi$ defines a left dg $\rC$-module structure on $\rB$, we obtain a morphism of dg $2$-representations $\bfM_\rC\to \bfM_\rB$ given by $-\circ_{\rC}\rB$. 
 As this morphism sends $\rC$ to $\rB$, it is essentially surjective.
\end{proof}

\subsection{Local quasi-equivalences and dg \texorpdfstring{$2$}{2}-representations}\label{sec:local-quasi}

Throughout this subsection, we use the same setup as in the previous one and additionally assume that the dg $2$-functor $\bsfF\colon \cC\to \cD$ is a local quasi-equivalence and, moreover, essentially surjective on objects. That is, for any object $\tj$ in $\cD$, there exists an object $\ti$ in $\cC$ such that $\bfP_\tk=\cD(\tk,-)$ is dg equivalent to $\bfP_{\bsfF(\ti)}=\cD(\bsfF(\ti),-)$ through composing with $1$-morphisms in $\cD$.

\begin{proposition}\label{prop-unit-quasi}
Under the assumptions on $\bsfF$ as above, $\Phi\colon\bfM_\rA\to\bsfR\bsfP\bfM_\rA$ is a quasi-equivalence of dg $2$-representations.
\end{proposition}

\begin{proof}
Recall that, if $\rA\in \cC(\ti,\ti)$, $\bfM_\rA(\tj)$ is the  thick closure of the $\rG\rA$ with $\rG\in \cC(\ti,\tj)$ and $\bsfR\bsfP\bfM_\rA(\tj)$ is the thick closure of the $\rH\bsfF(\rA)$ for $\rH\in \cD(\bsfF(\ti),\bsfF(\tj))$. Since $\K\bsfF_{\ti,\tj}$ is essentially surjective, we can find  $\rG\in \cC(\ti,\tj)$ such that $\rH\in \cD(\bsfF(\ti),\bsfF(\tj))$ is quasi-isomorphic to $\bsfF(\rG)$ and hence $\rH\bsfF(\rA)$ is quasi-isomorphic to $\bsfF(\rG)\bsfF(\rA)= \bsfF(\rG\rA)$, so $\K\Phi(\ti)$ is also essentially surjective.
Moreover, we have a commutative diagram
\begin{align}\label{eq:commudiag-quasi}
\xymatrix{
\Hom_{\bfM_\rA(\tj)}(\rG\rA,\rG'\rA) \ar[d] \ar^{\cong}@{-}[rr]&&\Hom_{\vv{\cC}(\ti,\tj)}(\rG,\rG'\rA) \ar[d]\\
\Hom_{\bsfR\bsfP\bfM_\rA(\tj)}(\bsfF(\rG\rA),\bsfF(\rG'\rA)) \ar^{\cong}@{-}[rr]&& \Hom_{\vv{\cD}(\ti,\tj)}(\bsfF(\rG),\bsfF(\rG'\rA)).
}
\end{align}
Since the right vertical arrow is a quasi-isomorphism by assumption, so is the left vertical arrow. This proves the lemma.
\end{proof}

\begin{proposition}\label{prop-counit-quasi}
Assume, in addition, that $\rB$ is a $\cC$-generator for $\bsfR\bfM_\rB$. Then the dg functor $\Psi\colon\bsfP\bsfR\bfM_\rB\to\bfM_\rB$ is a quasi-equivalence.
\end{proposition}

\begin{proof}
It suffices to prove that $-\circ_\rC\rB$ is a quasi-equivalence. By Lemma \ref{lem:counit}, this morphism of dg $2$-representation is essentially surjective. Thus, it suffices to show that it induces quasi-isomorphisms of morphism spaces.
Recall that $\varphi = \eval^{\cC}_{\rB\rB}\circ_1(\id_\rC\circ_0 u_{\rB})$.

We first claim that the diagram 
\begin{equation}\label{phicomp}\vcenter{\hbox{
\xymatrix{
\Hom_{\bfM_\rC(\tj)}(\rG\rC, \rH\rC)\ar@{-}^{\cong}[rr] \ar_{-\circ_\rC\id_\rB}[d]&&\Hom_{\vv{\cD}(\tj)}(\rG, \rH\rC)\ar^{(\id_{\rH}\circ_0\varphi)\circ_1-}[d]\\
\Hom_{\bfM_\rB(\tj)}(\rG\rB, \rH\rB)\ar@{-}^{\cong}[rr] &&\Hom_{\vv{\cD}(\tj)}(\rG, \rH\rB)
}}}
\end{equation}
commutes. Indeed, using $\eval^{\cD}_{\rC\rC}= \mathrm{m}_\rC$, a morphism $f\in \Hom_{\vv{\cD}(\tj)}(\rG, \rH\rC)$ corresponds to 
\begin{align*}\rG\rC \xrightarrow{f\circ_0\id_\rC} \rH\rC\rC \xrightarrow{\id_\rH\circ_0\mathrm{m}_\rC} \rH\rC\end{align*}
in $\Hom_{\bfM_\rC(\tj)}(\rG\rC, \rH\rC)$, which, under $-\circ_\rC\id_\rB$ maps to 
\begin{align*}\rG\rB \xrightarrow{f\circ_0\id_\rB} \rH\rC\rB\xrightarrow{\id_\rH\circ_0\eval^{\cC}_{\rB\rB}} \rH\rB.\end{align*}

Here we have used that $\mathrm{m}_\rC\circ_\rC\id_\rB = \eval^{\cC}_{\rB\rB}$ since both, by construction, correspond to the action on $\rB$ when viewed as a left $\rC$-module.

The above morphism then corresponds to 
\begin{align*}\rG\xrightarrow{\id_\rG\circ_0u_\rB}\rG\rB \xrightarrow{f\circ_0\id_\rB} \rH\rC\rB\xrightarrow{\id_\rH\circ_0\eval^{\cC}_{\rB\rB}} \rH\rB\end{align*}
in $\Hom_{\vv{\cD}(\tj)}(\rG, \rH\rB)$. Given 
\begin{align*}(f\circ_0\id_\rB)\circ_1(\id_\rG\circ_0u_\rB) = (\id_{\rH\rC}\circ_0 u_\rB)\circ_1 f,\end{align*}
we see that 
\begin{equation*}\begin{split}(\id_\rH\circ_0\eval^{\cC}_{\rB\rB})\circ_1(f\circ_0\id_\rB)\circ_1(\id_\rG\circ_0u_\rB) &=  ( \id_\rH\circ_0 (\eval^{\cC}_{\rB\rB}\circ_1(\id_\rC\circ_0 u_{\rB})))\circ_1 f \\
&= (\id_{\rH}\circ_0\varphi)\circ_1 f,
\end{split}\end{equation*}
proving that the diagram in \eqref{phicomp} commutes.

It thus suffices to show that the right vertical map in \eqref{phicomp} is a quasi-isomorphism. To see this, first assume $\rG=\bsfF(\rG')$ and $\rH=\bsfF(\rH')$ and consider the diagram
\begin{align*}\xymatrix@R=0.5cm{
\Hom_{\vv{\cC}(\ti,\tj)}(\rG', \rH'\rC')\ar^{\cong}[d]\ar^{\bsfF(-)}[rr]&& \Hom_{\vv{\cD}(\tj)}(\rG, \rH\rC)\ar@/^2pc/[ddll]^{(\id_{\rH}\circ_0\varphi)\circ_1-}\\
\Hom_{\bfM_\rB(\tj)}(\rG \rB, \rH\rB)\ar^{\cong}[d]&&\\
\Hom_{\vv{\cD}(\tj)}(\rG, \rH\rB)&&.
}\end{align*}
Under the vertical isomorphisms,  $f\in \Hom_{\vv{\cC}(\ti,\tj)}(\rG', \rH'\rC)$ is first mapped to $(\id_\rH\circ_0\eval^{\cC}_{\rB \rB})\circ_1 (\bsfF(f)\circ_0\id_\rB)$ and subsequently, using the same arguments as before, to 
\begin{equation*}\begin{split}
(\id_\rH\circ_0\eval^{\cC}_{\rB\rB})\circ_1 &(\bsfF(f)\circ_0\id_\rB)\circ_1(\id_\rG\circ_0 u_\rB) 
= (\id_{\rH}\circ_0\varphi)\circ_1 \bsfF(f),
\end{split}\end{equation*}
and hence the diagram commutes. Given that the vertical maps are isomorphisms, and the horizontal map is a quasi-isomor\-phism, the map $(\id_{\rH}\circ_0\varphi)\circ_1-$ is also a quasi-isomor\-phism as desired in this case.

For general $\rG,\rH$, there exist $\rG', \rH'$ and zigzags of quasi-isomorphisms between $\rG$ and $\bsfF(\rG')$ as well as $\rH$ and $\bsfF(\rH')$. The quasi-isomorphism 
\begin{align*}(\id_{\bsfF(\rH')}\circ_0\varphi)\circ_1 -\colon \Hom_{\vv{\cD}(\tj)}(\bsfF(\rG'), \bsfF(\rH')\rC) \to \Hom_{\vv{\cD}(\tj)}(\bsfF(\rG'), \bsfF(\rH')\rB)\end{align*} propagates along these to imply that 
\begin{align*}(\id_{\rH}\circ_0\varphi)\circ_1 -\colon \Hom_{\vv{\cD}(\tj)}(\rG, \rH\rC) \to \Hom_{\vv{\cD}(\tj)}(\rG, \rH\rB)\end{align*} is also a quasi-isomorphism.
This completes the proof.
\end{proof}

Even if $\rB$ does not $\cC$-generate $\bsfR\bfM_\rB$ as assumed in Lemma~\ref{lem:counit} and Proposition~\ref{prop-counit-quasi}, $\bsfF$ being a local quasi-equivalence and essentially surjective on objects implies that it is equivalent to a $\cC$-\emph{quasi}-generator.

\begin{lemma}\label{lem:Bquasi-gen}
The dg algebra $1$-morphism $\rB$ corresponds to a $\cC$-quasi-generator for $\bsfR\bfM_\rB$.
\end{lemma}
\begin{proof}
By assumption that $\bsfF$ is essentially surjective on objects, we can assume w.l.o.g. that $\rB\in \cD(\bsfF(\tk),\bsfF(\tk))$, for an object $\tk$ of $\cC$. 
We need to show that the inclusion $\rI\colon \bfG_{\bsfR \bfM_{\rB}}(\rB)\hookrightarrow \bsfR\bfM_\rB$ is a quasi-equivalence. As $\rI$ is fully faithful, it remains to show that $\bfK\rI$ is essentially surjective. That is, we claim that every  object $X$ in $\bfK\bsfR\bfM_\rB(\tj)$ is isomorphic to an object in the thick closure of $\lbrace \bsfF(\rG) \rB|\rG\in \cC(\tj,\tk) \rbrace$ in the homotopy category $\bfK\bsfR\bfM_\rB(\tj)$. Note that $X\in \bsfR\bfM_\rB(\tj)=\bfM_\rB(\bsfF(\tj))$ is dg isomorphic to an object in the thick closure of $\lbrace\rH\rB\;|\;\rH\in \cD(\bsfF(\tj),\bsfF(\tk)\rbrace$. 
Since any $1$-morphism $\rH$ of $\cD$ is homotopy isomorphic to $\bsfF(\rG)$ for some $1$-morphism $\rG$ of $\cC$, the claim follows.
\end{proof}

Pullback and pushforward preserve quasi-equivalences of internal dg $2$-representations, see Section \ref{sec:quasi-Morita}.

\begin{proposition}\label{prop:preserve-quasi-eq}
Retain the assumptions on $\bsfF\colon \cC\to\cD$ of this section.
\begin{enumerate}[(a)]
\item \label{prop:preserve-quasi-eqa}
For any quasi-equivalence $\Phi\colon \bfM_{\rA_1}\to\bfM_{\rA_2}$ of internal dg $2$-representations of $\cC$, $\bsfP(\Phi)$ is a quasi-equivalence.
\item \label{prop:preserve-quasi-eqb}
For any quasi-equivalence $\Psi\colon \bfM_{\rB_1}\to\bfM_{\rB_2}$ of internal dg $2$-representations of $\cD$, $\bsfR(\Psi)$ is a quasi-equivalence.
\end{enumerate}
\end{proposition}
\begin{proof}
To prove Part \eqref{prop:preserve-quasi-eqa}, recall Proposition \ref{triangequiv} and denote by $\rX$ a dg bimodule giving the quasi-equivalence $\Phi\colon \bfM_{\rA_1}\to \bfM_{\rA_2}$. That is, we have a isomorphism $\Phi(\rA_1)\cong \rX$ in $\bfK\bfM_{\rA_2}$, which implies that $\bsfP(\Phi)(\bsfF(\rA_1))\cong \bsfF(\rX)$. By Proposition \ref{triangequiv}, $\rA_2$ appears in the thick closure of  $\lbrace \rG \rX | \rG \rbrace$ in $\bfK\bfM_{\rA_2}$, for $\rG$ va\-rying over $1$-morphisms of $\cC$. Thus, $\bsfF(\rA_2)$ appears in the thick closure of  $\lbrace \bsfF(\rG)\bsfF(\rX) | \rG \rbrace$ in $\bfK\bfM_{\bsfF(\rA_2)}$. As $\bsfF(\rA_2)$ generates $\bsfP\bfM_{\rA_2}$ by definition, we see that the second condition in Proposition \ref{triangequiv} holds for $\bsfF(\rX)$. 

As in the diagram \eqref{eq:commudiag-quasi}, we have quasi-isomorphisms
\begin{align*}\Hom_{\bfM_{\rA_1}(\tk)}(\rF\rA_1,\rG\rA_1)\longrightarrow \Hom_{\bfM_{\bsfF(\rA_1)}(\bsfF\tk)}(\bsfF(\rF)\bsfF(\rA_1),\bsfF(\rG)\bsfF(\rA_1)).\end{align*}
Thus, we have the following commutative diagram of dg morphisms
\begin{align*}
\xymatrix{
\Hom_{\bfM_{\rA_1}(\tk)}(\rF\rA_1,\rG\rA_1)\ar[r]\ar[d]&\End_{\bfM_{\rA_2}(\tk)}(\rF\rX,\rG\rX)\ar[d]\\
\Hom_{\bfM_{\bsfF\rA_1}(\bsfF\tk)}(\bsfF(\rF)\bsfF(\rA_1),\bsfF(\rG)\bsfF(\rA_1))\ar[r]&\End_{\bfM_{\bsfF\rA_2}(\bsfF\tk)}(\bsfF(\rF)\bsfF(\rX),\bsfF(\rG)\bsfF(\rX)),
}
\end{align*}
where the vertical morphisms are quasi-isomorphisms, and the top vertical morphism is a quasi-isomorphism as $\Phi\colon\bfM_{\rA_1}\to \bfM_{\rA_2}$ is a quasi-equivalence. Thus, the bottom vertical morphism is a quasi-isomorphism.  For general $1$-morphisms $\rH,\rK$ in $\cD$, we have homotopy isomorphisms $\bsfF(\rF)\cong \rH$ and $\bsfF(\rG)\cong \rK$.
Composing with these homotopy isomorphisms we obtain the commutative diagram
\begin{align*}
\xymatrix{
\Hom_{\bsfP\bfM_{\rA_1}(\tk)}(\bsfF(\rF)\bsfF(\rA_1),\bsfF(\rG)\bsfF(\rA_1))\ar[d]\ar[r]&\End_{\bsfP\bfM_{\rA_2}(\tk)}(\bsfF(\rF)\bsfF(\rX),\bsfF(\rG)\bsfF(\rX))\ar[d]\\
\Hom_{\bsfP\bfM_{\rA_1}(\tk)}(\rH\bsfF(\rA_1),\rK\bsfF(\rA_1))\ar[r]&\End_{\bsfP\bfM_{\rA_2}(\tk)}(\rH\bsfF(\rX),\rK\bsfF(\rX)),
}
\end{align*}
where the vertical morphisms are quasi-isomorphisms by the above. As the top horizontal morphism is a quasi-isomorphism, so is the bottom morphism.

Hence, both conditions of Proposition \ref{triangequiv} are satisfied for the dg bimodule $\bsfF(\rX)$, which implies that the morphism $\bsfP(\Phi)$ from $\bsfP\bfM_{\rA_1}\cong \bfM_{\bsfF(\rA_1)}$ to $\bsfP\bfM_{\rA_2}\cong \bfM_{\bsfF(\rA_2)}$ is a quasi-equivalence and Part \eqref{prop:preserve-quasi-eqa} follows. 

To prove Part \eqref{prop:preserve-quasi-eqb}, assume given a quasi-equivalence $\Psi\colon \bfM_{\rB_1}\to \bfM_{\rB_2}$ of internal dg $2$-representations of $\cD$. Since $\bsfF$ is essentially surjective on objects, without loss of generality, $\rB_1\in \cD(\bsfF(\ti),\bsfF(\ti))$ and $\rB_2\in \cD(\bsfF(\tj),\bsfF(\tj))$ for objects $\ti,\tj$ in $\cC$. By Proposition \ref{triangequiv}, $\Psi$ corresponds to composing with a dg $\rB_1$-$\rB_2$-bimodule $\rY$ in $\vv{\cD}$ which descends to a cyclic generator for  $\bfK\bfM_{\rB_2}$ and induces quasi-isomorphisms on morphism spaces. We observe that 
\begin{align*}\Psi(\rB_1)\cong \rY\in \big(\bfG_{\bfM_{\rB_2}}(\rB_2)\big)(\bsfF(\tj))=\big(\bfG_{\bsfR\bfM_{\rB_2}}(\rB_2)\big)(\tj).\end{align*}
By Lemma \ref{lem:Bquasi-gen}, $\rB_2$ cyclically generates $\bfK\bsfR\bfM_{\rB_2}$. 
Now, as $\bsfF$ is a local quasi-equivalence, $\rB_2$ is in the thick closure of $\lbrace \bsfF(\rG)\rY | \rG\rbrace$ in $\bfK\bsfR\bfM_{\rB_2}$. Thus, $\rY$ cyclically generates $\bfK\bsfR\bfM_{\rB_2}$. 
Moreover, the quasi-isomorphisms 
\begin{align*}\Hom_{\bfM_{\rB_1}(\tk)}(\rH \rB_1,\rK\rB_1)\xrightarrow{\Psi}\Hom_{\bfM_{\rB_2}(\tk)}(\rH \rY,\rK\rY),\end{align*}
for $1$-morphisms  $\rH,\rK$ in $\cD$, induced by $\Psi$ provide quasi-isomorphisms
\begin{align*}
\xymatrix@R=0.5pc{
\Hom_{\bsfR\bfM_{\rB_1}(\tth)}(\bsfF(\rF) \rB_1,\bsfF(\rG)\rB_1)\ar@{=}[d]\ar[r]^{\bsfR\Psi}&\Hom_{\bsfR\bfM_{\rB_2}(\tth)}(\bsfF(\rF)\rY,\bsfF(\rG)\rY)\ar@{=}[d]\\
\Hom_{\bfM_{\rB_1}(\bsfF\tth)}(\bsfF(\rF)\rB_1,\bsfF(\rG)\rB_1)\ar[r]^{\Psi}&\Hom_{\bfM_{\rB_2}(\bsfF\tth)}(\bsfF(\rF)\rY,\bsfF(\rG)\rY)\\
}
\end{align*}
when restricting to acting by $1$-morphisms of the form $\bsfF(\rF)$, $\bsfF(\rG)$, for $1$-morphisms $\rF,\rG$ in $\cC$, and $\tth$ an object of $\cC$. Thus, $\bsfR(\Psi)$ is a quasi-equivalence of dg $2$-representations of $\cC$.
\end{proof}

The results of this section imply the following corollary.

\begin{corollary}\label{cor:quasi-2-equiv}
Assume given a dg $2$-functor $\bsfF\colon \cC\to \cD$ that is a local quasi-equivalence and essentially surjective on objects. 
\begin{enumerate}[(a)]
\item\label{cor:quasi-2-equiva}
For every internal dg $2$-representation $\bfM_{\rB}$ of $\cD$ there exists an internal dg $2$-representation $\bfM_{\rA}$ of $\cC$ and a quasi-equivalence of dg $2$-representations of $\cD$ from $\bsfP\bfM_\rA$ to $\bfM_{\rB}$.
\item \label{cor:quasi-2-equivb}
For every internal dg $2$-representation $\bfM_{\rA}$ of $\cC$ there is a quasi-equivalence of dg $2$-representations of $\cC$ from $\bfM_{\rA}$ to $\bsfR\bfM_{\bsfF(\rA)}$.
\end{enumerate}
\end{corollary}

\begin{proof}
We first proof Part \eqref{cor:quasi-2-equiva}.
Assume given an internal dg $2$-representation $\bfM_\rB$ of $\cD$ and, w.l.o.g., $\rB\in \cD(\bsfF(\tk),\bsfF(\tk))$. Set $\rC'={}_{\cC}[\rB,\rB]$ and $\rC=\bsfF(\rC')$, then $\bsfP\bfM_{\rC'}=\bfM_\rC$.  The proof of Lemma \ref{lem:counit} shows that, even if $\rB$ is not a $\cC$-generator for $\bsfR\bfM_\rB$, there is a natural morphism of dg $2$-representations $\Psi\colon \bfM_\rC \to \bfM_\rB$ which, by the proof of Proposition \ref{prop-counit-quasi}, is quasi-equivalence. Setting $\rA=\rC'$, Part \eqref{cor:quasi-2-equiva} follows.

Part \eqref{cor:quasi-2-equivb} directly follows from Proposition \ref{prop-unit-quasi}.
\end{proof}


\section{Examples}
\label{sec:examples}

\subsection{The dg \texorpdfstring{$2$}{2}-category \texorpdfstring{$\cC_A$}{C A}}\label{CAsection}

Let $A=A_\tone\times\cdots \times A_\tn$ be a finite-dimensional dg algebra, where each $A_\ti$ is indecomposable as a $\Bbbk$-algebra. 

Let $\A_\ti$ be a small dg category dg equivalent to $\widehat{\{A_\ti\}}$ inside $A_\ti\dgmod$, and set $\A = \coprod_{\ti=\tone}^\tn \A_\ti$. Note that, by definition, $\A$ is a pretriangulated category.

\begin{definition}
We define $\cC_A$  as the dg $2$-category with 
\begin{itemize}
\item objects $\mathtt{1}, \dots, \mathtt{n}$ where we identify $\mathtt{i}$ with $\A_\ti$;
\item $1$-morphisms in $\cC_A(\ti,\tj)$ are all functors dg isomorphic to tensoring with dg $A_\tj$-$A_\ti$-bimodules in the thick closure of $A_\tj\otimes_\Bbbk A_\ti$, if $\ti\neq\tj$, and in the thick closure of  of $A_\ti\oplus A_\ti\otimes_\Bbbk A_\ti$, if $\ti=\tj$;
\item $2$-morphisms all natural transformations of such functors.
\end{itemize}
\end{definition}

We define the \textbf{natural $2$-representation} $\bfN$ of $\cC_A$ as its defining action on $\A$, that is,  $\bfN(\ti) = \A_\ti$, $\bfN(\rF) = \rF$ for a $1$-morphism $\rF$, and $\bfN(\alpha) = \alpha$ for a $2$-morphism $\alpha$. By construction, $\cC_A$ is a pretriangulated $2$-category, and $\bfN$ is in $\cC_A\tworep$. It is cyclic with any $A_\ti\in \bfN(\ti)=\A_\ti$ as a $\cC_A$-generator.

We further denote by $\rF_{\tj,\ti}$ the functor given by tensoring with $A_\tj\otimes_\Bbbk A_\ti$.

\begin{lemma}\label{lem:AforN}
Choosing $A_\ti$ as a $\cC_A$-generator for $\bfN$, the corresponding algebra dg $1$-morphism $\rA_{A_\ti}$ is given by (tensoring with) the dg bimodule $A_\ti^*\otimes_{\Bbbk}A_\ti$.
\end{lemma}

We remark that $\rA_{A_\ti}$ is indeed a dg $1$-morphism in $\vv{\cC_A}$ by our assumption of finite-dimensionality of $A_\ti$. In fact, $\cC_A$ has compact left adjoints and $\bfN$ is compact as required in Section \ref{mor2rep-sect}.

\begin{proof}
First, notice that since $A_\ti\in \bfN(\ti)$, the dg algebra $1$-morphism $\rA_{A_\ti}$ is thus in $\vv{\cC_A}(\ti,\ti)$ and we only need to apply $\Eval_{A_\ti}$ to the generator of this dg category, which is given by $\one_\ti\oplus \rF_{\ti,\ti}$. We thus verify that
\begin{equation*}
\begin{split}
\Hom_{\vv{\cC_A}(\ti,\ti)}(\one_\ti\oplus \rF_{\ti,\ti}, \rA_{A_\ti})&\cong
\Hom_{A_\ti\text{-}\mathrm{mod}^{\mathrm{dg}}\text{-}A_\ti}(A_\ti\oplus A_\ti\otimes_\Bbbk A_\ti, A_\ti^*\otimes_{\Bbbk}A_\ti)\\
&\cong\Hom_{A_\ti\text{-}\mathrm{mod}^{\mathrm{dg}}\text{-}A_\ti}(A_\ti\oplus A_\ti\otimes_\Bbbk A_\ti, \Hom_{\Bbbk}(A_\ti^\Bbbk)\otimes_\Bbbk A_\ti)\\
&\cong \Hom_{A_\ti\text{-}\mathrm{mod}^{\mathrm{dg}}\text{-}A_\ti}(A_\ti\oplus A_\ti\otimes_\Bbbk A_\ti, \Hom_{\Bbbk}(A_\ti,A_\ti))\\
&\cong \Hom_{A_\ti\text{-}\mathrm{mod}^{\mathrm{dg}}}(A_\ti\oplus A_\ti\otimes_\Bbbk A_\ti, A_\ti)\\
&\cong \Hom_{\vv{\bfN}(\ti)}(\Eval_{A_\ti}(\one_\ti\oplus \rF_{\ti,\ti}), A_\ti),
\end{split}
\end{equation*}
which proves the claim.
\end{proof}
Hence, as $\bfN(\tj)$ is dg idempotent complete for any object $\tj$, Corollary~\ref{cor:stronggenMAequiv} implies that $\bfN$ is dg equivalent to the internal dg $2$-representation $\bfM_{\rA_{A_{\ti}}}$, for any $\ti$.

As a first example, we can consider the dg $2$-category $\cC_\Bbbk$. In this case, $\vv{\cC}_\Bbbk$ and $\cC_\Bbbk$ are dg biequivalent and a dg algebra $1$-morphism in $\cC_\Bbbk$ corresponds to a finite-dimensional dg algebra over $\Bbbk$. Results of \cite{Or2} imply the following corollary.

\begin{proposition}\label{lem:Ck-simpletrans}
There is a unique non-acyclic quotient-simple pretriangulated $2$-repre\-sen\-tations of $\cC_\Bbbk$ up to dg equivalence.
\end{proposition}
\begin{proof}
By Lemma \ref{simplelemma}, a quotient simple $2$-representation $\bfM$ corresponds to a dg $\Bbbk$-algebra $A$ that does not have any proper dg ideals. Using \cite{Or2}, this in particular implies that the dg ideal denoted by $J_+$ in loc.~cit. is zero, or $J_+=A$, in which case the dg algebra and hence the dg $2$-representation are acyclic. Thus, if $\bfM$ is not acyclic, the (ungraded) $\Bbbk$-algebra underlying $A$ is a product of matrix rings by \cite[Proposition 2.16]{Or2}.
It is remarked in the proof of loc.~cit.  that the central idempotents defining the individual factors are annihilated by $\del$, and hence simplicity of $A$ implies that the $\Bbbk$-algebra underlying $A$ is isomorphic to $M_n(\Bbbk)$ for some $n$.

Given a (non-zero) dg $\Bbbk$-module $V$, $V^*\otimes V$ is a dg $\Bbbk$-algebra with underlying $\Bbbk$-algebra isomorphic to $M_n(\Bbbk)$. As a dg algebra $1$-morphism in $\cC_\Bbbk$, $V^*\otimes V$ appears as the internal hom $[V,V]$ obtained from $\bfN$ with $V$ as a (weak) $\cC_\Bbbk$-generator similarly to Lemma \ref{lem:AforN}. In fact, all differentials on matrix rings are of this form by \cite[Proposition~2.15]{Or2}. Thus, $A$ is isomorphic to $V^*\otimes V$ as a dg algebra, for some dg $\Bbbk$-module $V$. 

It remains to show that the pretriangulated $2$-representations $\bfM_{V^*\otimes V}$  are dg equivalent for any choice of $V$. For this, we apply Proposition \ref{dgMoreq} to $V$ viewed as a $\Bbbk$-$V^*\otimes V$-bimodule and $V^*$ viewed as a $V^*\otimes V$-$\Bbbk$-bimodule. Since $V\otimes_{V^*\otimes V} V^*\cong \Bbbk$ is an isomorphism of dg $\Bbbk$-$\Bbbk$-bimodules, it follows that $V^*\otimes V$ and $\Bbbk$ are dg Morita equivalent and hence $\bfM_{V^*\otimes V}$ is dg equivalent to the natural $2$-representation $\bfN\cong \bfM_{\Bbbk}$ of $\cC_\Bbbk$. This proves the claim.
\end{proof}

\begin{example}
Consider the dg $\Bbbk$-algebra $D=\Bbbk[x]/(x^2)$ with $\del(x)=1$. Then $D$ has no proper dg ideals. The \emph{internal dg ideal} $I_-$ of $D$ is equal to zero, while the \emph{external dg ideal} $I_+=D$. This implies that $D$ is acyclic \cite[Section~2.1]{Or2}. In particular, $D$ has no proper dg ideals, but is not given by a matrix algebra.
\end{example}

\subsection{Pretriangulated hulls of finitary \texorpdfstring{$2$}{2}-categories and \texorpdfstring{$2$}{2}-representations}
\label{prehullsec}

Let $\cD$ be a $\Bbbk$-finitary $2$-category in the sense of \cite{MM1}. We can view $\cD$ a dg $2$-category with zero differential. Let $\cC=\ov{\cD}$ be the pretriangulated $2$-category associated to it. Any finitary $2$-representation $\bfM$ of $\cD$ extends naturally to a pretriangulated $2$-representation $\ov{\bfM}$ of $\cC$.
It is easy to see that given an object $X\in \bfM(\ti)$ and viewing it as an object in $\ov{\bfM}(\ti)$ by placing it in degree $0$, the associated dg algebra $1$-morphism $\rA_X$ in $\vv{\cC}(\ti, \ti)$ is simply the algebra $1$-morphism obtained from the finitary $2$-representation $\bfM$, again interpreted as an object in  $\vv{\cC}(\ti, \ti)$ by placing it in degree zero with zero differential.

For example, if $A=A_\tone\times\ldots\times A_\tn$, for $A_\ti$ finite-dimensional connected $\Bbbk$-algebras, one defines the $\Bbbk$-finitary $2$-category $\cC_A$ following \cite[Section~4.5]{MM3}. Then the pretriangulated hull of $\cC_A$ recovers the dg category $\cC_{A_0}$ defined in Section \ref{CAsection}, for the dg algebra $A_0$ given by $A$ with zero differential. The dg category $\cC_{A_0}(\ti,\tj)$ is dg equivalent to that of bounded complexes of $A_\tj$-$A_\ti$-bimodules that are projective (if $\ti\neq \tj$) or a direct sum of projective $A_\ti$-$A_\ti$-bimodules and direct sums of copies of $A_\ti$, if $\ti=\tj$.

\begin{question}
Are there quotient-simple dg $2$-representations of $\cC$ which are not dg equivalent to ones of the form $\ov{\bfM}$ for a simple transitive $2$-representation $\bfM$ of $\cD$?
\end{question}

We remark that this question has a negative answer in case $\cD=\cC_\Bbbk$ by Proposition~\ref{lem:Ck-simpletrans}.

\subsection{Pretriangulated hulls of finite multitensor categories}

We can consider a (strict) finite multitensor category $\cD$ as in \cite{EGNO}, viewed as a $2$-category with one object $\bullet$, and consider its pretriangulated hull $\cC=\ov{\cD}$ as in the preceding section. For any $2$-representation $\bfN$ of (in \cite{EGNO} called a (strict) module category over) $\cD$ gives rise to a pretriangulated $2$-representation $\bfM=\ov{\bfN}$ of $\cC$. 

Now assume that $\cD$ is a strict finite multitensor category. In particular, $\cD$ is abelian, locally finite, and every $1$-morphism has left and right adjoints. These structures extend to the pretriangulated hull $\cC$.
Therefore, the evaluation dg functor
$\Eval_X\colon \cC(\bullet,\bullet)\to \bfM(\bullet)$ from Section~\ref{inthomsec}
has a right adjoint $[X,-]\colon \bfM(\bullet)\to \cC(\bullet,\bullet)$ without extending to $\vv{\cC}$ (see \cite[Sections~4.2,~7.3]{EGNO}).
Our results in Section~\ref{alg1morsec} generalize aspects of \cite[Section~7]{EGNO}.
Similarly to Section \ref{prehullsec}, algebra $1$-morphisms associated to $\bfM=\ov{\bfN}$ can be chosen inside $\cD$ (i.e., concentrated in degree zero).

\subsection{Categorification of \texorpdfstring{$\mathbb{Z}[i]$}{Z[i]}}
\label{sec:Zi}

In \cite{Ti}, Tian constructs an explicit monoidal dg category $\I$ which decategorifies to $\mathbb{Z}[i]$, and its action on a dg category $DGP(R)$, which categorifies the natural action of $\mathbb{Z}[i]$ on $\mathbb{Z}^{\oplus 2}$. Here, we explicitly compute a dg algebra $1$-morphism associated to this dg $2$-representation. We adopt the notation in loc.~cit.

Let $X=P(x)$, the projective $R$-module at vertex $x$, which generates the dg $2$-representa\-tion. Using the fact that $\I = \ov{\{Q^0,Q^1=Q\}}$ (see \cite[Definition 2.2]{Ti} together with the isomorphism $Q^2\cong Q^0[1]$) it suffices to compute morphism spaces to $[X,X]$ from $Q^0$ and $Q^1$. We compute 
\begin{equation*}\begin{split}
\Hom_{\I}(Q^0, [X,X]) &\cong \Hom_{\I}(Q^{0}\otimes_R P(x),P(x))\cong \Hom_{\I}(P(x),P(x))\cong \Bbbk[d]/(d^2)\\
\Hom_{\I}(Q^1, [X,X]) &\cong \Hom_{\I}(Q^1\otimes_R P(x),P(x))\cong \Hom_{\I}(P(y),P(x))\cong \Bbbk[-1],
\end{split}\end{equation*}
where $d$ is the degree $1$-endomorphism of $P(x)$.
This implies that the dg algebra $1$-mor\-phism in $\vv{\I}$ is given by $Q^1[-2]\xrightarrow{\alpha_0\circ\overline{hf}}Q^0$.
The multiplication is given by 
\begin{align*}
\xymatrix{
Q^1[-2]\oplus Q^1[-2]\ar[d]^{(\id,\id)}\ar[rr]^-{(\alpha_0\circ \ov{hf},\alpha_0\circ \ov{hf})}&&Q^0\ar[d]^{\id}\\
Q^1[-2]\ar[rr]^{\alpha_0\circ \ov{hf}}&&Q^0.
}
\end{align*}
As multiplication is induced by the identity on $Q_0$, there is a dg algebra epimorphism given by
\begin{align*}\xymatrix{
Q^1[-2]\ar^{\alpha_0\circ\overline{hf}}[rr]\ar^{\alpha_1}[d]&&Q^0\ar^{\id}[d]\\
Q^1[-1]\ar^{\overline{hf}}[rr]&&Q^0
}\end{align*} 
(note that $\alpha_0\circ\overline{hf} = \overline{hf}\circ \alpha_1$). The given dg $2$-representation is thus not quotient-simple. Indeed, it has an ideal given by the radicals of $R$ and $M$.
Note that this dg ideal satisfies the conditions from Lemma \ref{quotientpretri}\eqref{quotientpretri3}.
The quotient is given by factoring out the radicals of $R$ and $M$, i.e. acting on $DGP(R')$ where $R'=\Bbbk e_x\times \Bbbk e_y$ and the action of $Q_1$ given by the $R'$-$R'$-bimodule $M' = \Bbbk e_y\otimes e_x\oplus \Bbbk e_x\otimes e_y[1]$.

\subsection{Braid group categorification}
\label{sec:braid}

The setup developed in this paper can be applied to existing categorified braid group actions from the literature. Firstly,
in \cite{KS}, the authors consider a categorification of the braid group of type $\mathtt{A}_n$ using complexes of bimodules over the symmetric zigzag algebra $Z=Z_n$ on $n$-vertices. We denote the resulting dg $2$-category by $\cB_n$ with generating $1$-morphisms $T_i= Ze_i\otimes e_iZ \to Z$, $T_i'= Z\to Ze_i\otimes e_iZ$, as well as the identity bimodule.

Note that any pretriangulated $2$-representation of $\cB_n$ is necessarily closed under cones, and hence extends to a dg $2$-representation of $\cC_Z$, see Section~\ref{CAsection}. In fact, $\cC_Z$ is the dg idempotent completion of the $\ov{\cB_n}$.

Given that any dg ideal in a dg $2$-representation of $\cB_n$ gives rise to a dg ideal in the extended dg $2$-representation of $\cC_Z$ and vice versa, there is a one-to-one correspondence between quotient-simple pretriangulated $2$-representations of $\cB_n$ and $\cC_Z$.  Thus, it follows from Lemma \ref{lem:AforN} that the dg algebra $1$-morphism of the defining action corresponds to the dg algebra $1$-morphism $Z\otimes_\Bbbk Z$ in $\vv{\cB_n}$.

More generally, Rouquier \cite{Ro0} considers a categorification of the braid group associated to an arbitrary finitely-generated Coxeter system $(W,S,V)$ in the homotopy category of Soergel bimodules. One could alternatively consider the dg $2$-category $\cB_W$ generated by the complexes denoted in loc.~cit. by $F_s$, which has a natural action on the category of bounded chain complexes of graded $\Bbbk[V]$-modules. The homotopy category of this dg $2$-category gives rise to Rouquier's categorification of the braid group, which he conjecturally strictifies in \cite{Ro1}.

As before, we can describe the dg idempotent completion of $\ov{\cB_W}$ as the monoidal category of chain complexes of Soergel bimodules $\cS_W$, and we again restrict or extend pretriangulated $2$-representations to pass from one to the other. Note that in both cases the dg ideal generated by the total invariants $\Bbbk[V]^W$ acts by zero on every quotient-simple dg $2$-representation, so we can pass to the quotient dg $2$-categories defined in the same way, but over the coinvariant algebra $C=\Bbbk[V]/(\Bbbk[V]^W_+)$ instead of $\Bbbk[V]$, which, provided $W$ is finite, is a finite-dimensional symmetric algebra, see e.g. \cite{MMMTZ}, and references therein, for details. Then, for finite $W$, the natural (or defining) dg $2$-representation of $\cB_W$ on chain complexes of $C$-modules simply corresponds to the pretriangulated hull of the cell $2$-representation of $\cS_W$ associated to the longest element $w_0$ of the (finite) Coxeter group whose Kazhdan Lusztig basis element is categorified by $C\otimes C$, see \cite[Sections 2.1--2.2]{E} for more explanations, where this bimodule is described as $C\otimes_{C^S} C$ with $C^S$ denoting the total invariants in the coinvariants and hence is isomorphic to $\Bbbk$ (more generally for $J\subseteq S$, $C\otimes_{C^J}C$ decategorifies to the Kazhdan--Lusztig basis element corresponding to the longest element of the parabolic subgroup determined by $J$ --- this uses the Decategorification Theorem \cite{EW}). The corresponding dg algebra $1$-morphism is the simply given by $C\otimes C$, which is indeed a $1$-morphism in $\vv{\cB_W} = \vv{\cS_W}$.

\end{document}